\documentclass[reqno,11pt]{amsart}
\usepackage{amsmath,amssymb,latexsym,soul,cite,mathrsfs,accents}
\usepackage[dvipsnames]{xcolor}
\usepackage{color,enumitem,graphicx}
\usepackage[colorlinks=true,urlcolor=blue,
citecolor=red,linkcolor=blue,linktocpage,pdfpagelabels,
bookmarksnumbered,bookmarksopen,hypertexnames=false]{hyperref}
\usepackage[english]{babel}
\usepackage[left=2.6cm,right=2.6cm,top=2.9cm,bottom=2.9cm]{geometry}
\usepackage{tensor}
\newtheorem{theorem}{Theorem}
\newtheorem{lemma}[theorem]{Lemma}

\newtheorem{proposition}[theorem]{Proposition}
\newtheorem{remark}[theorem]{Remark}
\newtheorem{definition}[theorem]{Definition}
\newtheorem{conjecture}[theorem]{Conjecture}

\newenvironment{acknowledgement}{\noindent\textbf{Acknowledgments.}}{}

\newcommand{\innerthmname}{}

\theoremstyle{definition}

\makeatletter
\def\namedlabel#1#2{\begingroup
	#2%
	\def\@currentlabel{#2}%
	\phantomsection\label{#1}\endgroup
}

\def\XXint#1#2#3{{\setbox0=\hbox{$#1{#2#3}{\int}$ }
		\vcenter{\hbox{$#2#3$ }}\kern-.6\wd0}}

\newcommand*\owedge{\mathpalette\@owedge\relax}
\newcommand*\@owedge[1]{%
	\mathbin{%
		\ooalign{%
			$#1\m@th\bigcirc$\cr
			\hidewidth$#1\m@th\wedge$\hidewidth\cr
		}%
	}%
}
\makeatother

\newcommand{\ud}{\mathrm{d}}
\newcommand{\loc}{\mathrm{loc}}

\DeclareMathOperator{\supp}{supp}

\allowdisplaybreaks[3]

\title[Hartree equations with isolated singularities]{Asymptotic behavior of solutions to critical Hartree equations with isolated singularities}
\thanks{This work was partially supported by Funda\c c\~ao de Amparo \`a Pesquisa do Estado de S\~ao Paulo (FAPESP), Conselho Nacional de Desenvolvimento Cient\'ifico e Tecnol\'ogico (CNPq), National Science Foundation of China (NSFC), and Natural Science Foundation of Zhejiang Province (ZJNSF).
	J.H.A. was supported by FAPESP \#2020/07566-3, \#2021/15139-0, and \#2023/15567-8 and CNPq \#409764/2023-0, \#443594/2023-6, \#441922/2023-6, and \#306014/2025-4.
	P.P. was supported by FAPESP \#2022/16097-2, and CNPq \#313773/2021-1 and \#441922/2023-6.
	M.Y. was supported by NSFC \#12471114 and ZJNSF \#LZ26A010002.}
\author[J.H. Andrade]{Jo\~{a}o Henrique Andrade}
\author[T. Feng]{Tao Feng}
\author[P. Piccione]{Paolo Piccione}
\author[M. Yang]{Minbo Yang*}

\address[J.H. Andrade]{Institute of Mathematics and Statistics,
	University of S\~ao Paulo
	\newline\indent
	05508-090, S\~ao Paulo-SP, Brazil}
\email{\href{mailto:andradejh@ime.usp.br}{andradejh@ime.usp.br}}

\address[T. Feng]{School of Mathematical Sciences,
	Zhejiang Normal University
	\newline\indent
	321004, Jinhua-ZJ, People’s Republic of China}
\email{\href{mailto:fengtao@zjnu.edu.cn}{fengtao@zjnu.edu.cn}}

\address[P. Piccione]{
	Department of Mathematics,
	School of Sciences, Great Bay University
	\newline\indent
	523000, Dongguan-GD, People’s Republic of China
	\newline\indent
	and
	\newline\indent
	School of Mathematical Sciences,
	Zhejiang Normal University
	\newline\indent
	321004, Jinhua-ZJ, People’s Republic of China
	\newline\indent
	and
	\newline\indent
	(permanent address) Institute of Mathematics and Statistics,
	University of S\~ao Paulo
	\newline\indent
	05508-090, S\~ao Paulo-SP, Brazil}
\email{\href{mailto:paolo.piccione@usp.br}{piccione@ime.usp.br}}

\address[M. Yang]{School of Mathematical Sciences,
	Zhejiang Normal University
	\newline\indent
	321004, Jinhua-ZJ, People’s Republic of China}
\email{\href{mailto:mbyang@zjnu.edu.cn}{mbyang@zjnu.edu.cn}}

\thanks{*Corresponding author.}

\makeatletter
\@namedef{subjclassname@2020}{\textup{2020} Mathematics Subject Classification}
\makeatother


\newcommand{\pffR}{\mathbb R}
\newcommand{\pffOo}{\Omega_k^{\mathrm{out}}}
\newcommand{\pffOi}{\Omega_k^{\mathrm{in}}}
\newcommand{\pffOt}{\widetilde\Omega_k^{\mathrm{out}}}
\newcommand{\pffSph}{\mathbb S^{n-1}}
\newtheorem*{pffclaima}{Claim A (compactness and limiting profile)}
\newtheorem*{pffclaimb}{Claim B (Kelvin comparison for the limiting profile)}

\newcommand{\pffproofstep}[1]{\par\medskip\noindent\textbf{#1}\quad}
\newcommand{\pffclaimpart}[2]{\par\medskip\phantomsection\label{p45:claim:#1}\noindent\textbf{Step #1 (#2).}\quad}
\newcommand{\pffproofsubstep}[2]{\par\smallskip\noindent\phantomsection\label{p45:step:#1}\textit{Step #1 (#2).}\quad}

\begin{document}
	
	\begin{abstract}
		We investigate the qualitative properties of positive classical solutions $u$ to critical Hartree equations on punctured domains, assuming that $u(x)\to+\infty$ as $x\to0$. For the global problem set in $\mathbb R^n\setminus\{0\}$, we prove that such solutions are radial and nonincreasing in $|x|$. For the local problem in $B_R\setminus\{0\}$, we establish the upper bound $u(x)\leqslant C|x|^{-(n-2)/2}$ near the origin and the asymptotic radial symmetry estimate
		\begin{equation*}
			u(x)=\bar u(|x|)\bigl(1+\mathcal O(|x|)\bigr)\qquad\text{as }x\to0,
		\end{equation*}
		where $\bar u$ is the spherical average of $u$. Our approach combines integral representations, blow-up analysis, and the method of moving spheres. A main difficulty is controlling the nonlocal Hartree potential under Kelvin inversion and in the blow-up limit. We also give an alternative proof of the classification of positive entire classical solutions in $\mathbb{R}^{n}$, which is used to identify the blow-up profile.
	\end{abstract}
	
	\maketitle
	
	\begin{center}
		\begin{minipage}{8cm}
			\footnotesize
			\tableofcontents
		\end{minipage}
	\end{center}
	
	\numberwithin{equation}{section}
	\numberwithin{theorem}{section}
	
	\section{Introduction}\label{section1}
	We investigate the qualitative properties of positive singular solutions
	of the critical Hartree equation
	\begin{equation*}\tag{$\mathcal P_{n,\alpha,R}$}\label{ourlocalPDE}
		-\Delta u=v_R[u]f(u)
		\quad\text{in }B_R^*:=B_R\setminus\{0\},
	\end{equation*}
	where $n\geqslant3$, $0<\alpha<n$, $0<R<\infty$, and $B_R=B_R(0)$.
	The nonlinearities are given by
	\begin{equation*}
		F(t):=\frac{|t|^p}{p},\qquad
		f(t):=|t|^{p-1},\qquad
		p:=\frac{n+\alpha}{n-2}.
	\end{equation*}
	For a measurable set $\Omega\subset\mathbb R^n$, we define
	the Hartree potential of $u$ over $\Omega$ by
	\begin{equation*}
		v_\Omega[u](y):=
		\int_\Omega \mathcal R_\alpha(y-z)F(u(z))\,\ud z,
		\qquad \mathcal R_\alpha(z):=|z|^{\alpha-n}.
	\end{equation*}
	When $\Omega=B_R$, we write $v_R[u]:=v_{B_R}[u]$. 
	
	The corresponding global equation with an isolated singularity at the origin, also considered in this paper, is
	\begin{equation}\tag{$\mathcal P_{n,\alpha,\infty}$}\label{ourlimitPDE}
		-\Delta u=(\mathcal R_\alpha*F(u))f(u)
		\quad\text{in }\mathbb R^n\setminus \{0\},
	\end{equation}
	where the convolution is taken over the whole space:
	\begin{equation*}
		(\mathcal R_\alpha*F(u))(y)
		=\int_{\mathbb R^n}|y-z|^{\alpha-n}F(u(z))\,\ud z.
	\end{equation*}
	
	A classical solution is a function with continuous derivatives up to the order of the equation that satisfies the equation pointwise on its domain. For the positive classical solutions considered in this paper, the finiteness of the Hartree integral is implicit in this requirement. Indeed, since $u$ is $C^2$ and strictly positive on the domain of the equation,
	\begin{equation}\label{bound of the Hartree integral}
	v_\Omega[u](x)=\frac{-\Delta u(x)}{u(x)^{p-1}}<\infty,
	\end{equation}
	where $\Omega$ denotes the integration domain.
	For solutions on punctured domains, this holds away from the origin;
	for entire classical solutions, it holds everywhere.
	No separate finite-energy or global $L^p$ assumption is imposed.
	
	Throughout this paper, a positive classical solution of \eqref{ourlocalPDE} or \eqref{ourlimitPDE} is called singular if 
    \begin{equation}\tag{$\mathcal{S}$}\label{eq:blowup}
    \lim_{x\to0}u(x)=+\infty.
    \end{equation}
	For the global problem, we prove that every positive singular classical solution $u\in\mathcal C^2(\mathbb R^n\setminus\{0\})$ of \eqref{ourlimitPDE} is radially symmetric about the origin and nonincreasing in the radial variable (Theorem~\ref{thm1}). For the local problem, we establish the upper bound $u(x)\leqslant C|x|^{-(n-2)/2}$ near the origin for positive singular classical solutions of \eqref{ourlocalPDE} (Proposition~\ref{prop:upperbound}) and use this bound to prove asymptotic radial symmetry:
	\begin{equation*}
		u(x)=\bar u(|x|)\bigl(1+\mathcal O(|x|)\bigr)\quad\text{as }x\to0,
	\end{equation*}
	where $\bar u$ is the spherical average and the error is uniform in direction (Theorem~\ref{thm2}).
	
	Hartree and Choquard equations describe quantum systems with nonlocal interactions. A standard form in three dimensions is
	\begin{equation*}
		-\Delta u+u=(\mathcal R_\alpha*u^2)u\quad\text{in }\mathbb R^3,
	\end{equation*}
	where $0<\alpha<3$, $\mathcal R_\alpha(x)=|x|^{\alpha-3}$, and the convolution is over $\mathbb R^3$.
	For $\alpha=2$, the kernel is the Coulomb kernel $|x|^{-1}$.
	Early studies of electron--lattice interactions include the work of
	Fr\"ohlich~\cite{froelichA,froelichB}.
	In Pekar's polaron model, the electron density $u^2$ polarizes the
	surrounding crystal, and the resulting field acts back on the
	electron~\cite{pekar}.
	The Coulomb equation also arises in a Hartree--Fock approximation
	of a plasma~\cite{MR471785}.
	Symmetry and classification of positive solitary solutions were
	studied by Ma and Zhao~\cite{MR2592284}; see Moroz and Van
	Schaftingen~\cite{MR3625092} for a broader survey of Choquard equations.
	
	For the critical exponent $p=(n+\alpha)/(n-2)$,
	equation \eqref{ourlocalPDE} is invariant under scaling,
	with the radius $R$ scaled accordingly.
	This exponent also arises from the Hardy--Littlewood--Sobolev
	inequality; see~\cite{MR1544927,MR165337}.
	
	We recall the sharp inequality of Lieb \cite{MR717827} in the form needed here. For $f_1,f_2\in L^{\frac{2n}{n+\alpha}}(\mathbb R^n)$,
	\begin{equation*}\label{HLS}\tag{$\mathrm{HLS}$}
		\left|
		\int_{\mathbb R^n}\int_{\mathbb R^n}
		f_1(x)|x-y|^{\alpha-n}f_2(y)\,\ud x\,\ud y
		\right|
		\leqslant H(n,\alpha)
	\|f_1\|_{L^{2n/(n+\alpha)}(\mathbb R^n)}
	\|f_2\|_{L^{2n/(n+\alpha)}(\mathbb R^n)}
	\end{equation*}
	where
	\begin{equation*}
		H(n, \alpha)=\pi^{\frac{n-\alpha}{2}}
		\frac{\Gamma(\alpha/2)}{\Gamma((n+\alpha)/2)}
		\left[\frac{\Gamma(n)}{\Gamma(n/2)}\right]^{\alpha/n}.
	\end{equation*}
	For nonzero nonnegative functions, equality holds precisely when $f_1=c_1f_*$ and $f_2=c_2f_*$, where $c_1,c_2>0$ and
	\begin{equation*}
		f_*(x)=\left(\frac{\mu}{\mu^2+|x-x_0|^2}\right)^{\frac{n+\alpha}{2}}
		\quad\text{for some }\mu>0,\ x_0\in\mathbb R^n.
	\end{equation*}
	Combining \eqref{HLS} with the Sobolev inequality yields
	\begin{equation}\label{hartreeinequality}
		\int_{\mathbb R^n}|\nabla u|^2\,\ud x
		\geqslant K_n(\alpha)
		\left(\int_{\mathbb R^n}(\mathcal R_\alpha*|u|^p)|u|^p\,\ud x\right)^{1/p},
		\qquad u\in\mathcal D^{1,2}(\mathbb R^n).
	\end{equation}
	Here $\mathcal D^{1,2}(\mathbb R^n)$ is the completion of $C_c^\infty(\mathbb R^n)$ under the norm $\|\nabla u\|_{L^2}$. The optimal constant is $K_n(\alpha)=S_nH(n,\alpha)^{-1/p}$, where
	\begin{equation*}
		S_n=\inf_{0\ne u\in\mathcal D^{1,2}(\mathbb R^n)}
		\frac{\int_{\mathbb R^n}|\nabla u|^2\,\ud x}
		{\left(\int_{\mathbb R^n}|u|^{\frac{2n}{n-2}}\,\ud x\right)^{\frac{n-2}{n}}}
		=n(n-2)\left[\pi^{n/2}\frac{\Gamma(n/2)}{\Gamma(n)}\right]^{2/n}.
	\end{equation*}
	This infimum is attained, for example, by the Talenti--Aubin bubbles \cite{MR0463908,MR0448404} (see also \cite{rodemich})
	\begin{equation*}
		[n(n-2)]^{\frac{n-2}{4}}
		\left(\frac{\mu}{1+\mu^2|x-x_0|^2}\right)^{\frac{n-2}{2}}.
	\end{equation*}
	After multiplication by a suitable positive constant, the positive functions attaining equality in \eqref{hartreeinequality} satisfy the Euler--Lagrange equation
	\begin{equation*}\tag{$\mathcal P_{n,\alpha}$}\label{ourlimitPDEnonsing}
		-\Delta u=(\mathcal R_\alpha*F(u))f(u)\quad\text{in }\mathbb R^n.
	\end{equation*}
	The convolution in this equation is over $\mathbb R^n$.
	
	Related work on critical Hartree equations concerns positive entire solutions under integrability or finite-energy assumptions \cite{MR4027015,MR3978520}, Dirichlet problems in bounded domains \cite{MR3817173,MR4632157}, and the dynamics of radial solutions at the ground-state energy threshold \cite{MR3334067}. For equation \eqref{ourlimitPDEnonsing},
	Le~\cite{jmaa2020123859} classified positive classical solutions
	in $\mathbb R^n$ for the full range $0<\alpha<n$, without assuming finite energy, using the method of moving planes
	in integral form. Dai, Liu and Qin~\cite{MR4226994} obtained classification results for nonnegative classical solutions of a broader class of static Schr\"odinger--Hartree--Maxwell type equations, including the second-order critical equation considered here, using the method of moving spheres in integral form. We recall the classification for \eqref{ourlimitPDEnonsing} below and present another proof in Subsection~\ref{subsection2.2} using a direct moving-spheres argument based on the method of Li and Zhang \cite{MR2001065} (see also Li and Zhu~\cite{MR1369398}). We also use this classification and arguments from its proof in establishing the upper bound near the origin for \eqref{ourlocalPDE}.
	\begin{theorem}\label{thmB}
		Let $n\geqslant3$ and $\alpha\in(0,n)$. If $u\in\mathcal C^2(\mathbb R^n)$ is a positive classical solution of \eqref{ourlimitPDEnonsing}, then there exist $x_0\in\mathbb R^n$ and $\mu>0$ such that
		\begin{equation*}
			u(x)=C_n(\alpha)
			\left(\frac{\mu}{1+\mu^2|x-x_0|^2}\right)^{\frac{n-2}{2}},
		\end{equation*}
		where
		\begin{equation}\label{alphaconstant}
			C_n(\alpha)=\left(\frac{pn(n-2)}{J_{n,\alpha}}\right)^{\frac1{2p-2}},
			\qquad J_{n,\alpha}=\pi^{n/2}
			\frac{\Gamma(\alpha/2)}{\Gamma((n+\alpha)/2)}.
		\end{equation}
		These solutions are called spherical solutions, or bubbles.
	\end{theorem}
	
	For the global equation \eqref{ourlimitPDE}, we prove the following symmetry result.
	\begin{theorem}[Symmetry]\label{thm1}
		Let $n\geqslant3$ and $\alpha\in(0,n)$. If
		$u\in\mathcal C^2(\mathbb R^n\setminus\{0\})$
		is a positive classical solution of \eqref{ourlimitPDE} satisfying \eqref{eq:blowup}, then $u=u(|x|)$ is radially symmetric about the origin and nonincreasing in the radial variable.
	\end{theorem}
	
	\begin{remark}
		An example of a positive classical solution of \eqref{ourlimitPDE} satisfying \eqref{eq:blowup} is 
		\begin{equation*}
			u(x)=A|x|^{-(n-2)/2},
		\end{equation*}
		where
		\begin{equation*}
			A=\left[
			\frac{p(n-2)^2\Gamma\bigl((n-\alpha)/2\bigr)
			\Gamma\bigl((n+\alpha)/4\bigr)^2}
			{4\pi^{n/2}\Gamma(\alpha/2)
			\Gamma\bigl((n-\alpha)/4\bigr)^2}
			\right]^{\frac{1}{2p-2}}.
		\end{equation*}
	\end{remark}
	
	We recently became aware of the arXiv preprint by Cai, Gu and
	Jevnikar~\cite{arxiv:2508.07893}.
	They used moving planes to study symmetry and monotonicity
	of positive weak solutions
	$u\in\mathcal D^{1,2}(\mathbb R^n\setminus\Gamma)
	\cap C(\mathbb R^n\setminus\Gamma)$
	to Hartree equations under convexity and growth conditions
	on the nonlinearity.
	Here $\Gamma$ is a compact singular set of zero $2$-capacity
	contained in a hyperplane, and their kernel parameter range
	corresponds to $0<\alpha<2$ in our notation.
	When the origin is the only nonremovable singularity,
	they obtain radial symmetry and strict radial monotonicity.
	For the critical power $f(t)=t^{p-1}$,
	our Theorem~\ref{thm1} establishes radial symmetry and radial
	monotonicity of positive classical solutions satisfying
	\eqref{eq:blowup} for all $0<\alpha<n$,
	without assuming finite energy or global $L^p$ integrability.
	Our proof uses a global integral representation and moving spheres.
	
	For the local equation \eqref{ourlocalPDE}, we prove the following upper bound and asymptotic symmetry result.
	\begin{theorem}[Upper bound and asymptotic radial symmetry]\label{thm2}
		Let $n\geqslant3$ and $0<\alpha<n$. Suppose that $0<R<\infty$ and that $u\in\mathcal C^2(B_R^*)$ is a positive classical solution of \eqref{ourlocalPDE} satisfying \eqref{eq:blowup}. Then there exist $r_0\in(0,R)$ and $C>0$ such that
		\begin{equation*}
			u(x)\leqslant C|x|^{-(n-2)/2}
			\quad\text{for }0<|x|<r_0.
		\end{equation*}
		Moreover,
		\begin{equation*}
			u(x)=\bar u(|x|)\bigl(1+\mathcal O(|x|)\bigr)
			\quad\text{as }x\to0,
		\end{equation*}
		where
		\begin{equation*}
			\bar u(r)=\frac1{|\mathbb S^{n-1}|}
			\int_{\mathbb S^{n-1}}u(r\theta)\,\ud\sigma(\theta)
		\end{equation*}
		is the spherical average of $u$. The error estimate is uniform in the angular variable.
	\end{theorem}
	Related results on isolated singularities were obtained by
	Ghergu and Taliaferro~\cite{MR3487256}.
	They considered nonnegative functions
	$u\in C^2(\mathbb R^n\setminus\{0\})\cap L^\lambda(\mathbb R^n)$
	satisfying
	\begin{equation*}
		0\leqslant-\Delta u(x)
		\leqslant
		\left(\int_{\mathbb R^n}
		\frac{u(y)^\lambda}{|x-y|^a}\,\ud y\right)u(x)^\sigma
		\quad\text{in }B_2\setminus\{0\},
	\end{equation*}
	where $n\geqslant3$, $0<a<n$, $\lambda>0$, and $\sigma\geqslant0$.
	They established pointwise bounds near the origin
	under suitable conditions on the parameters.
	Our work concerns the critical Hartree equation with
	the nonlocal integral taken over $B_R$, for which we establish
	an upper bound near the isolated singularity
	and asymptotic radial symmetry.
	
	\smallskip
	
	We briefly outline the main ideas in the proofs of
	Theorems~\ref{thm1} and~\ref{thm2}. Classical symmetry arguments for second-order elliptic equations use the maximum principle and the moving-plane method; see Gidas, Ni and Nirenberg~\cite{MR544879}. Our proofs use moving spheres and integral representations.
	
	For Theorem~\ref{thm1}, the two main difficulties are establishing
	the global integral representation and comparing the Hartree potential with its Kelvin transform. We first obtain the global representation
	\begin{equation}\tag{$\mathcal P_{n,\alpha,\infty}^\prime$}\label{ourlimitPDEdual}
		u=\Phi_n*[(\mathcal R_\alpha*F(u))f(u)]
		\quad\text{in }\mathbb R^n\setminus\{0\},
	\end{equation}
	where $\Phi_n$ is the Newton kernel defined in Section~\ref{sec:preliminaries}. Both convolutions are over $\mathbb R^n$. The critical exponent preserves the form of the integral system under Kelvin inversion. Splitting the integrals
	across a moving sphere and inverting the inside part yields
	comparison identities with positive kernels for both the
	solution and its potential. These identities allow us to
	continue the comparison until the sphere reaches the origin,
	giving radial symmetry and radial monotonicity.
	Related integral comparisons appear in
	\cite{MR4304557,arxiv:1901.01678}.
	
	For Theorem~\ref{thm2}, we first establish the local
	integral representation
	\begin{equation*}
		u(y)=\int_{B_\tau}\Phi_n(y-z)v_R[u](z)f(u(z))\,\ud z+h(y)
		\quad\text{in }B_\tau\setminus\{0\},
	\end{equation*}
	where $0<\tau<R$ and $h$ is positive and harmonic in $B_\tau$.
	The Hartree potential remains defined over $B_R$. Using this representation, we then prove the pointwise
	upper bound by a blow-up argument.
	In passing to the blow-up limit, the far part of the Hartree
	integral may contribute a nonnegative constant.
	An auxiliary Liouville lemma shows that this constant is zero,
	so Theorem~\ref{thmB} identifies the limit as a bubble.
	Moving-spheres comparisons for the rescaled solutions would
	then force this limit to be constant, giving the contradiction
	that proves the upper bound.
	
	With the upper bound established, we use it to control the
	additional integral terms caused by the change of integration
	domain under Kelvin inversion.
	The resulting comparison of the Hartree potentials, together
	with the maximum principle and Hopf's lemma, yields
	asymptotic radial symmetry.

	\smallskip
	
	Similar questions about isolated singularities have been studied for positive solutions of the critical equation
	\begin{equation*}\tag{$\mathcal I_{n,\sigma,R}$}\label{ourlocalPDEsigma}
		(-\Delta)^{\sigma/2}u=u^{\frac{n+\sigma}{n-\sigma}}
		\quad\text{in }B_R^*,\qquad 0<\sigma<n,
	\end{equation*}
	where $0<R<\infty$ and $(-\Delta)^{\sigma/2}$ has Fourier symbol $|\xi|^\sigma$. When $\sigma=2m$ with $m$ a positive integer, this operator is $(-\Delta)^m$. Otherwise it is nonlocal, so its value in $B_R^*$ also depends on $u$ outside $B_R$.
	
	For $\sigma=2$, equation \eqref{ourlocalPDEsigma} becomes
	\begin{equation*}
		-\Delta u=u^{\frac{n+2}{n-2}}\quad\text{in }B_R\setminus\{0\},\qquad n\geqslant3.
	\end{equation*}
	This is the Euclidean Yamabe equation, which prescribes constant
	scalar curvature. Chen and Li~\cite{MR1338474} studied the related
	problem of prescribing Gaussian or scalar curvature on spheres.
	
	The radial reduction leads to an Emden--Fowler equation;
	see~\cite{emden,fowler} for the classical theory.
	For the same equation in $\mathbb R^n\setminus\{0\}$,
	Caffarelli, Gidas and Spruck~\cite{MR982351} proved that positive
	solutions singular at the origin are radially symmetric.
	These solutions form the Fowler family and can be classified by
	ordinary differential equation analysis;
	see also \cite[Section~2]{MR1666838}.
	For the local problem, Caffarelli, Gidas and Spruck~\cite{MR982351}
	proved asymptotic radial symmetry near the origin.
	C. Li~\cite{MR1374197} simplified this argument and established
	asymptotic radial symmetry for positive singular solutions of
	$-\Delta u=g(u)$ under weaker assumptions on $g$.
	Korevaar, Mazzeo, Pacard and Schoen~\cite{MR1666838} obtained a more
	precise expansion in terms of radial singular solutions.
	
	For $\sigma=4$ and $n\geqslant5$, Lin~\cite{MR1611691} classified
	positive entire solutions of the critical biharmonic equation.
	Frank and K\"onig~\cite{MR3869387} classified positive solutions
	with a nonremovable singularity at the origin of
	\begin{equation*}
		\Delta^2u=u^{\frac{n+4}{n-4}}\quad\text{in }\mathbb R^n\setminus\{0\}.
	\end{equation*}
	Guo, Huang, Wang and Wei~\cite{arXiv:1708.04660} studied Delaunay
	solutions of this equation in the punctured space.
	For the fourth-order equation \eqref{ourlocalPDEsigma} on $B_R^*$,
	Ratzkin~\cite{arxiv.2001.07984} proved refined asymptotics under the
	additional assumption that the associated conformal metric has
	positive scalar curvature. See also~\cite{MR4793741} for fourth-order
	systems on punctured balls.
	
	For $0<\sigma<2$, Caffarelli, Jin, Sire and Xiong~\cite{MR3198648}
	proved pointwise bounds and asymptotic radial symmetry near an
	isolated singularity for the fractional equation~\eqref{ourlocalPDEsigma}.
	In this range, DelaTorre, del Pino, Gonz\'alez and Wei~\cite{MR3694655}
	constructed Delaunay-type solutions on $\mathbb R^n\setminus\{0\}$
	when $n\geqslant2+\sigma$.
	
	A related integral equation on the punctured ball is
	\begin{equation*}\tag{$\mathcal I^\prime_{n,\sigma,R}$}\label{ourlocalPDEsigmadual}
		u(x)=c_{n,\sigma}\int_{B_R}|x-y|^{\sigma-n}
		u(y)^{\frac{n+\sigma}{n-\sigma}}\,\ud y+h(x)
		\quad\text{in }B_R^*,
	\end{equation*}
	where $h\in\mathcal C^1(B_R)$ is positive and
	\begin{equation*}
		c_{n,\sigma}=\frac{\Gamma((n-\sigma)/2)}
		{2^\sigma\pi^{n/2}\Gamma(\sigma/2)}.
	\end{equation*}
	For positive solutions $u\in\mathcal C(B_R^*)\cap L^{(n+\sigma)/(n-\sigma)}(B_R)$ of \eqref{ourlocalPDEsigmadual}, Jin and Xiong \cite[Theorem~1.3]{arxiv:1901.01678} proved the bound $u(x)\leqslant C|x|^{-(n-\sigma)/2}$ and asymptotic radial symmetry near the origin. For $\sigma=2m<n$ with $m\geqslant1$ an integer, they also obtained these conclusions for positive solutions $u\in\mathcal C^{2m}(B_R^*)$ of \eqref{ourlocalPDEsigma} satisfying $(-\Delta)^k u\geqslant0$ in $B_R^*$ for $k=1,\ldots,m-1$. Their proof first establishes a representation of the form \eqref{ourlocalPDEsigmadual}, with $R$ replaced by a smaller radius and $h>0$, and then applies the integral-equation result; see \cite[Theorem~1.1 and Section~2]{arxiv:1901.01678}.
	
	The corresponding whole-space integral equation is
	\begin{equation*}\tag{$\mathcal I^\prime_{n,\sigma,\infty}$}\label{ourglobalPDEsigmadual}
		u(x)=c_{n,\sigma}\int_{\mathbb R^n}|x-y|^{\sigma-n}
		u(y)^{\frac{n+\sigma}{n-\sigma}}\,\ud y
		\quad\text{in }\mathbb R^n\setminus\{0\}.
	\end{equation*}
	Chen, Li and Ou \cite{MR2200258} classified the positive regular solutions of this integral equation on all of $\mathbb R^n$. In the punctured space, they proved radial symmetry about the origin for positive solutions whose only nonremovable singularity is at the origin \cite{MR2122171}. For radial solutions, the change of variables $t=\ln r$, $\psi(t)=r^{(n-\sigma)/2}u(r)$ yields an integral equation on $\mathbb R$. Using a variational argument, Jin and Xiong \cite[Theorem~1.4]{arxiv:1901.01678} found a positive, smooth, nonconstant $T$-periodic profile $\psi$ for every sufficiently large $T$. These profiles define global radial singular solutions of \eqref{ourglobalPDEsigmadual}.
	
   Motivated by the periodic profiles described above and the
   refined asymptotics for $-\Delta u=u^{(n+2)/(n-2)}$,
   we propose the following conjecture.
   
   \begin{conjecture}\label{conj:periodicprofile}
   	Let $n\geqslant3$ and $0<\alpha<n$.
   	\begin{enumerate}
   		\item[(i)]
   		There exist $T>0$ and a positive, nonconstant,
   		$T$-periodic function $U\in C^\infty(\mathbb R)$
   		such that
   		\begin{equation*}
   			u_*(x):=|x|^{-(n-2)/2}U(-\log|x|),
   			\qquad x\in\mathbb R^n\setminus\{0\},
   		\end{equation*}
   		is a positive classical solution of the global
   		equation \eqref{ourlimitPDE}.
   		
   		\item[(ii)]
   		Let $0<R<\infty$ and let
   		$u\in C^2(B_R\setminus\{0\})$ be a positive classical
   		solution of \eqref{ourlocalPDE} satisfying
   		\eqref{eq:blowup}.
   		Then there exist $\gamma>0$ and a positive radial
   		classical solution $u_*$ of the global equation
   		\eqref{ourlimitPDE} such that
   		\begin{equation*}
   			u(x)=u_*(x)\bigl(1+O(|x|^\gamma)\bigr)
   			\qquad\text{as }x\to0,
   		\end{equation*}
   		uniformly in $x/|x|$.
   		The solution $u_*$ has the form given in (i),
   		with $U$ either constant or nonconstant periodic.
   	\end{enumerate}
   \end{conjecture}

	\smallskip
	
	The paper is organized as follows. Section~\ref{sec:preliminaries} contains the Kelvin transform, kernel estimates, a proof of Theorem~\ref{thmB}, and the Liouville lemma used in the blow-up argument. Section~\ref{sec:existence} proves Theorem~\ref{thm1}. In Section~\ref{sec:asymptotics}, we establish the local integral representation, the pointwise upper bound, and the asymptotic radial-symmetry estimate in Theorem~\ref{thm2}.
	
	\section{Preliminaries}\label{sec:preliminaries}
	We first fix the notation and prove Theorem~\ref{thmB}, then establish an auxiliary Liouville lemma for the blow-up argument in Section~\ref{sec:asymptotics}.
	The nonlinearities $F,f$ and the exponent $p$ are as defined at the beginning of Section~\ref{section1}.
	\begin{itemize}
		\item $c,C>0$ denote constants whose dependence is specified when used;
		\item for nonnegative quantities, $a_1\lesssim a_2$ means $a_1\leqslant Ca_2$,
		$a_1\gtrsim a_2$ means $a_1\geqslant ca_2$, and
		$a_1\simeq a_2$ means both inequalities, with constants independent of
		the variables under consideration;
		\item $u=\mathcal O(g)$ as $x\rightarrow x_0$ means that
		$|u(x)|\leqslant C|g(x)|$ for some $C>0$ and all $x$ in a punctured
		neighborhood of $x_0$;
		\item $u=\mathrm o(g)$ as $x\rightarrow x_0$ means
		$\lim_{x\rightarrow x_0}u(x)/g(x)=0$, provided that $g$ does not vanish
		in a punctured neighborhood of $x_0$.
	\end{itemize}
	
	\subsection{Kelvin transform}
	For $x\in\mathbb R^n$ and $\mu>0$, write
	\begin{equation*}
		\mathcal{I}_{x,\mu}(z):=x+\frac{\mu^2(z-x)}{|z-x|^2},\qquad z\ne x.
	\end{equation*}
	\begin{definition}\label{def:kelvintransform}
		Let $\Omega\subset\mathbb R^n$ be open, $u\in \mathcal{C}^2(\Omega)$, and
		$x\in\mathbb R^n$, $\mu>0$.
		For $z\in\Omega^{x,\mu}:=\mathcal{I}_{x,\mu}(\Omega\setminus\{x\})$,
		define the Kelvin transform
		\begin{equation*}
			(u)_{x,\mu}(z):=\left(\frac{\mu}{|z-x|}\right)^{n-2}
			u\bigl(\mathcal I_{x,\mu}(z)\bigr).
		\end{equation*}
		For convenience, we write $u_{x,\mu}$ for $(u)_{x,\mu}$ and $u_\mu$ for $u_{0,\mu}$.
	\end{definition}
	
	We use the Hartree kernel $\mathcal R_\alpha$ defined in the
	introduction and the Newton kernel
	\begin{equation*}
		\Phi_n(z)=c_n|z|^{2-n},\qquad z\ne0,
	\end{equation*}
	where $c_n=1/((n-2)|\mathbb S^{n-1}|)$.
	Here $-\Delta\Phi_n=\delta_0$ in the distributional sense.
	We use the Hartree potential $v_\Omega[u]$ defined in Section~\ref{section1},
	with $v_R[u]=v_{B_R}[u]$ for balls centered at the origin.
	For fixed $u$ and $\Omega$, write $v:=v_\Omega[u]$ and define its Kelvin transform by
	\begin{equation}\label{prelim:potentialkelvin}
		v_{x,\mu}(y):=\left(\frac{\mu}{|y-x|}\right)^{n-\alpha}
		v(\mathcal I_{x,\mu}(y)),
		\qquad y\in\Omega^{x,\mu}.
	\end{equation}
	
	The next proposition describes how the above Kelvin transforms act on the equation
	and its integral form.
	
	\begin{proposition}\label{prop:conformalinvariance}
		Let $\Omega\subset\mathbb R^n$ be open.
		If $u\in \mathcal{C}^2(\Omega)$ satisfies
		\begin{equation*}
			-\Delta u=v_\Omega[u]f(u)
			\quad\text{in }\Omega,
		\end{equation*}
		then, for every $x\in\mathbb R^n$ and $\mu>0$, its Kelvin transform satisfies
		\begin{equation*}
			-\Delta u_{x,\mu}
			=v_{\Omega^{x,\mu}}[u_{x,\mu}]f(u_{x,\mu})
			\quad\text{in }\Omega^{x,\mu}.
		\end{equation*}
		If, in addition,
		\begin{equation*}
			u(y)=\int_\Omega\Phi_n(y-z)v_\Omega[u](z)f(u(z))\,\ud z
			\quad\text{in }\Omega,
		\end{equation*}
		then, in $\Omega^{x,\mu}$,
		\begin{equation*}
			u_{x,\mu}(y)=\int_{\Omega^{x,\mu}}\Phi_n(y-z)v_{\Omega^{x,\mu}}[u_{x,\mu}](z)f(u_{x,\mu}(z))\,\ud z.
		\end{equation*}
	\end{proposition}
	\begin{proof}
		Write $v=v_\Omega[u]$. The change of variables
		$z=\mathcal I_{x,\mu}(\zeta)$ and $(n-2)p=n+\alpha$ give
		\begin{equation}\label{prelim:potentialcovariance}
			v_{x,\mu}(y)
			=\int_{\Omega^{x,\mu}}|y-\zeta|^{\alpha-n}F(u_{x,\mu}(\zeta))\,\ud\zeta
			=v_{\Omega^{x,\mu}}[u_{x,\mu}](y).
		\end{equation}
		The differential equation follows from this identity,
		$(n-\alpha)+(n-2)(p-1)=n+2$, and the standard Laplace--Kelvin identity
		\begin{equation}\label{prelim:laplacekelvin}
			-\Delta u_{x,\mu}(y)
			=\left(\frac{\mu}{|y-x|}\right)^{n+2}
			(-\Delta u)(\mathcal I_{x,\mu}(y)),
			\qquad y\in\Omega^{x,\mu}.
		\end{equation}
		The integral representation follows by the same change of variables
		in the integral representation of $u$.
	\end{proof}
	The moving-sphere arguments below involve the difference between the Riesz kernel and its Kelvin transform. We record its positivity outside the inversion sphere and a two-sided estimate that will be used throughout the paper.
	
	\begin{lemma}[Kelvin kernel]\label{lm:differencecomparison}
		Let $n\geqslant3$, $0<\alpha<n$, $x\in\mathbb R^n$, and $\mu>0$. For $|y-x|,|z-x|>\mu$ with $y\ne z$, define
		\begin{equation*}
			K_\alpha(x,\mu;y,z):=|y-z|^{\alpha-n}
			-\left(\frac\mu{|y-x|}\right)^{n-\alpha}
			|\mathcal I_{x,\mu}(y)-z|^{\alpha-n}.
		\end{equation*}
		Put
		\begin{equation}\label{prelim:kernelparameter}
			s_{\mathcal K}=s_{\mathcal K}(x,\mu;y,z):=\frac{(|y-x|^2-\mu^2)(|z-x|^2-\mu^2)}{\mu^2|y-z|^2}.
		\end{equation}
		Then $K_\alpha>0$ and
		\begin{equation}\label{prelim:kernelestimate}
			c|y-z|^{\alpha-n}\min\{s_{\mathcal K},1\}
			\leqslant K_\alpha(x,\mu;y,z)
			\leqslant C|y-z|^{\alpha-n}\min\{s_{\mathcal K},1\},
		\end{equation}
		where $c,C>0$ depend only on $n,\alpha$.
	\end{lemma}
	\begin{proof}
		The inversion identity
		\begin{equation}\label{prelim:kerneldistance}
			\frac{|y-x|^2}{\mu^2}|\mathcal I_{x,\mu}(y)-z|^2-|y-z|^2
			=\frac{(|y-x|^2-\mu^2)(|z-x|^2-\mu^2)}{\mu^2}
		\end{equation}
		gives $K_\alpha=|y-z|^{\alpha-n}[1-(1+s_{\mathcal K})^{-(n-\alpha)/2}]$.
		The function in brackets vanishes at $s_{\mathcal K}=0$, and its derivative has positive upper and lower bounds on $[0,1]$ depending only on $n,\alpha$.
		Thus, by the mean value theorem, it is comparable to $s_{\mathcal K}$ for $0<s_{\mathcal K}\leqslant1$.
		For $s_{\mathcal K}\geqslant1$, it lies between the positive constant $1-2^{-(n-\alpha)/2}$ and $1$.
		These two cases show that it is positive and comparable to $\min\{s_{\mathcal K},1\}$, which proves~\eqref{prelim:kernelestimate}.
	\end{proof}
	
	For $|y-x|=\mu<|z-x|$, identity~\eqref{prelim:kerneldistance} implies $K_\alpha(x,\mu;y,z)=0$. We use $K_2$ for the case $\alpha=2$.
	
	\subsection{Classification of positive entire solutions}\label{subsection2.2}
	We prove Theorem~\ref{thmB} for \eqref{ourlimitPDEnonsing} on $\mathbb R^n$,
	with $n\geqslant3$, $0<\alpha<n$, and $p=(n+\alpha)/(n-2)$.
	The moving-spheres argument uses the following two calculus lemmas.
	\begin{lemma}\label{lm:msfundamentallemma}
		Let $u\in \mathcal{C}^1(\mathbb R^n)$, $n\geqslant1$, and $\nu>0$.
		\begin{enumerate}
			\item[(i)] If $u>0$ and for every $x\in\mathbb R^n$ there exists a radius $\mu(x)>0$
			such that
			\begin{equation*}
				\left(\frac{\mu(x)}{|y-x|}\right)^\nu
				u\left(x+\frac{\mu(x)^2(y-x)}{|y-x|^2}\right)=u(y)
				\quad \text{ for }y\ne x,
			\end{equation*}
			then there exist constants $C,\rho>0$ and a point $x_0\in\mathbb R^n$ such that
			$u(y)=C(1+\rho^2|y-x_0|^2)^{-\nu/2}$ for every $y$.
			\item[(ii)] If for every $x\in\mathbb R^n$ and $\mu>0$,
			\begin{equation}\label{msfundamentallemma2}
				\left(\frac{\mu}{|y-x|}\right)^\nu
				u\left(x+\frac{\mu^2(y-x)}{|y-x|^2}\right)\leqslant u(y)
				\quad \text{ for }|y-x|\geqslant\mu,
			\end{equation}
			then $u$ is constant.
		\end{enumerate}
	\end{lemma}
	\begin{proof}
		Part~(i) follows from Li and Zhang~\cite[Lemma~11.1]{MR2001065},
		using $u>0$ and rewriting the constants. Part~(ii) follows from
		\cite[Lemma~11.2]{MR2001065} with exponent $\nu$ and the exterior
		comparison~\eqref{msfundamentallemma2}.
	\end{proof}
	
	We first show that, for each fixed center $x$, the comparison
	$u_{x,\lambda}\leqslant u$ holds outside the sphere for all sufficiently small radii $\lambda$.
	\begin{lemma}\label{lemmaA.1}
		Let $n\geqslant3$. If $u\in \mathcal{C}^2(\mathbb R^n)$ is a positive
		pointwise solution of \eqref{ourlimitPDEnonsing}, then for every
		$x\in\mathbb R^n$ there exists a radius $\lambda_0(x)>0$ such that
		$u_{x,\lambda}(y)\leqslant u(y)$ whenever
		$0<\lambda<\lambda_0(x)$ and $|y-x|\geqslant\lambda$.
	\end{lemma}
	\begin{proof}
		By translation, we may assume that $x=0$. Since $u$ is positive and
		of class $\mathcal C^2$, the quantity
		$C_0:=\max_{\overline B_1}|\nabla\log u|$ is finite.
		Choose $r_0\in(0,1)$ such that $r_0C_0<(n-2)/2$.
		For every $\theta\in\mathbb S^{n-1}$ and $0<r\leqslant r_0$,
		\begin{equation*}
			\frac{d}{dr}\log\bigl(r^{(n-2)/2}u(r\theta)\bigr)
			=\frac{n-2}{2r}+\partial_r\log u(r\theta)
			\geqslant\frac{n-2}{2r_0}-C_0>0.
		\end{equation*}
		Consequently, $r\mapsto r^{(n-2)/2}u(r\theta)$ is strictly
		increasing on $(0,r_0]$. If $0<\lambda<r\leqslant r_0$, then
		$0<\lambda^2/r<r\leqslant r_0$, and hence
		\begin{equation*}
			\left(\frac{\lambda^2}{r}\right)^{(n-2)/2}
			u\left(\frac{\lambda^2}{r}\theta\right)
			\leqslant r^{(n-2)/2}u(r\theta).
		\end{equation*}
		Dividing by $r^{(n-2)/2}$ yields
		$u_\lambda(r\theta)\leqslant u(r\theta)$.
		Since $u_\lambda=u$ on $\partial B_\lambda$, we obtain
		$u_\lambda(y)\leqslant u(y)$ for $\lambda\leqslant|y|\leqslant r_0$.
		
		We next consider $|y|\geqslant r_0$. Set
		\begin{equation*}
			m:=\min_{\partial B_{r_0}}u>0,
			\qquad M:=\max_{\overline B_{r_0}}u<\infty.
		\end{equation*}
		The equation implies
		$-\Delta u=(\mathcal R_\alpha\ast F(u))u^{p-1}\geqslant0$
		in $\mathbb R^n$. For $T>r_0$, define
		\begin{equation*}
			h_T(y):=m\frac{|y|^{2-n}-T^{2-n}}{r_0^{2-n}-T^{2-n}},
			\qquad r_0\leqslant|y|\leqslant T.
		\end{equation*}
		Then $h_T$ is harmonic in $B_T\setminus\overline B_{r_0}$, with
		$h_T=m$ on $\partial B_{r_0}$ and $h_T=0$ on $\partial B_T$.
		Thus $-\Delta(u-h_T)\geqslant0$ in the annulus, while
		\begin{equation*}
			u-h_T=u-m\geqslant0\quad\text{on }\partial B_{r_0},
			\qquad
			u-h_T=u>0\quad\text{on }\partial B_T.
		\end{equation*}
		By the minimum principle, $u\geqslant h_T$ throughout the annulus.
		For each fixed $y$ with $|y|>r_0$, this inequality holds for every
		$T>|y|$. Letting $T\to\infty$ and using $n\geqslant3$, we obtain
		\begin{equation}\label{prelim:superharmoniclower}
			u(y)\geqslant m\frac{|y|^{2-n}}{r_0^{2-n}}
			=mr_0^{n-2}|y|^{2-n}.
		\end{equation}
		The same bound holds on $\partial B_{r_0}$ by the definition of $m$.
		
		Choose $\lambda_0\in(0,r_0)$ such that
		$M\lambda_0^{n-2}\leqslant mr_0^{n-2}$, which is possible since
		$n-2>0$. If $0<\lambda<\lambda_0$ and $|y|\geqslant r_0$, then
		\begin{equation*}
			|\mathcal I_{0,\lambda}(y)|
			=\frac{\lambda^2}{|y|}\leqslant\frac{\lambda^2}{r_0}<r_0.
		\end{equation*}
		It follows that
		\begin{equation*}
			u_\lambda(y)
			\leqslant M\lambda^{n-2}|y|^{2-n}
			\leqslant mr_0^{n-2}|y|^{2-n}
			\leqslant u(y).
		\end{equation*}
		Combining this with the comparison on
		$\lambda\leqslant|y|\leqslant r_0$ proves the assertion for $x=0$.
		Translating back to the original center completes the proof.
	\end{proof}
	
	For a positive pointwise solution $u$ of \eqref{ourlimitPDEnonsing}
	and each $x\in\mathbb R^n$, define
	\begin{equation}\label{criticalradius}
		\bar\lambda(x):=\sup\left\{L>0:
		\begin{array}{l}
			u_{x,\mu}(y)\leqslant u(y)\text{ for every }0<\mu\leqslant L,\\
			\text{and every }y\text{ with }|y-x|\geqslant\mu
		\end{array}\right\}.
	\end{equation}
	We call $L$ admissible if it belongs to the set in
	\eqref{criticalradius}. By Lemma~\ref{lemmaA.1}, every
	$L\in(0,\lambda_0(x))$ is admissible, so $\bar\lambda(x)\in(0,\infty]$.
	Moreover, if the comparison holds for all radii $0<\mu\leqslant L$,
	it also holds for all radii $0<\mu\leqslant L'$ whenever $0<L'\leqslant L$.
	
	For such a solution, write $v:=v_{\mathbb R^n}[u]$.
	Equation~\eqref{bound of the Hartree integral} ensures that $v$ is finite everywhere in $\mathbb R^n$.
	Its Kelvin transforms are finite away from their centers by~\eqref{prelim:potentialkelvin}.
	By~\eqref{prelim:potentialcovariance}, we have, for every
	$x\in\mathbb R^n$ and $\lambda>0$,
	\begin{equation*}
		v_{x,\lambda}(y)
		=v_{\mathbb R^n}[u_{x,\lambda}](y)<\infty,\qquad y\ne x.
	\end{equation*}
	We next show that, if the moving sphere stops at a finite radius, then
	$u$ agrees with the corresponding Kelvin transform away from the center.
	\begin{lemma}\label{lemmaA.2}
		Let $u\in\mathcal C^2(\mathbb R^n)$ be a positive pointwise solution
		of \eqref{ourlimitPDEnonsing}, with $\bar\lambda$ defined by
		\eqref{criticalradius}. If $\bar\lambda(x)<\infty$, then
		$u_{x,\bar\lambda(x)}\equiv u$ on $\mathbb R^n\setminus\{x\}$.
		If $\bar\lambda(x_0)=\infty$ at one center, then
		$\bar\lambda(x)=\infty$ at every center.
	\end{lemma}
	\begin{proof}
		We first prove the equality with the Kelvin transform when the critical radius is finite.
		Fix such a center and translate it to the origin, so that $0<\bar\lambda(0)<\infty$. Write $v=v_{\mathbb R^n}[u]$ and set $d_\lambda:=u-u_\lambda$.
		By the definition of $\bar\lambda(0)$ in~\eqref{criticalradius},
		$d_\lambda\geqslant0$ for $|y|\geqslant\lambda$ whenever $0<\lambda<\bar\lambda(0)$.
		Letting $\lambda\uparrow\bar\lambda(0)$, we obtain $d_{\bar\lambda(0)}\geqslant0$
		for $|y|\geqslant\bar\lambda(0)$, with equality on $\partial B_{\bar\lambda(0)}$
		by the definition of the Kelvin transform.
		
		To compare $v$ and $v_{0,\bar\lambda(0)}$, recall from Lemma~\ref{lm:differencecomparison}
		that $K_\alpha(0,\bar\lambda(0);y,z)>0$ for $|y|,|z|>\bar\lambda(0)$ and $y\ne z$.
		For convenience, write
		\begin{equation*}
			Q_{\bar\lambda(0)}(y,z):=\left(\frac{\bar\lambda(0)}{|y|}\right)^{n-\alpha}
			|\mathcal I_{0,\bar\lambda(0)}(y)-z|^{\alpha-n}.
		\end{equation*}
		Thus $|y-z|^{\alpha-n}-Q_{\bar\lambda(0)}(y,z)=K_\alpha(0,\bar\lambda(0);y,z)$.
		Since $\alpha>0$, the singularity at $z=y$ is locally integrable.
		
		Using~\eqref{prelim:potentialcovariance} for $v_{0,\bar\lambda(0)}$, split the integrals for $v$ and $v_{0,\bar\lambda(0)}$ into the interior and exterior of $B_{\bar\lambda(0)}$.
		In each inner integral, denote the original variable by $\zeta$
		and set $\zeta=\mathcal I_{0,\bar\lambda(0)}(z)$, so that $|z|>\bar\lambda(0)$. Then
		\begin{equation*}
			\,\ud\zeta=\left(\frac{\bar\lambda(0)}{|z|}\right)^{2n}\,\ud z,
			\qquad
			|y-\mathcal I_{0,\bar\lambda(0)}(z)|=\frac{|y|}{|z|}|\mathcal I_{0,\bar\lambda(0)}(y)-z|,
		\end{equation*}
		and
		\begin{equation*}
			F(u(\mathcal I_{0,\bar\lambda(0)}(z)))=\left(\frac{|z|}{\bar\lambda(0)}\right)^{n+\alpha}F(u_{\bar\lambda(0)}(z)).
		\end{equation*}
		The last identity follows from $(n-2)p=n+\alpha$.
		\begin{samepage}
			Substituting these identities into the two integrals, we find
			\begin{equation*}
				v(y)=\int_{|z|>\bar\lambda(0)}
				\bigl[|y-z|^{\alpha-n}F(u(z))+Q_{\bar\lambda(0)}(y,z)F(u_{\bar\lambda(0)}(z))\bigr]\,\ud z,
			\end{equation*}
			and
			\begin{equation*}
				v_{0,\bar\lambda(0)}(y)=\int_{|z|>\bar\lambda(0)}
				\bigl[|y-z|^{\alpha-n}F(u_{\bar\lambda(0)}(z))+Q_{\bar\lambda(0)}(y,z)F(u(z))\bigr]\,\ud z.
			\end{equation*}
		\end{samepage}
		The second identity also uses $(u_{\bar\lambda(0)})_{\bar\lambda(0)}=u$.
		As noted before the lemma, $v$ and $v_{0,\bar\lambda(0)}$ are finite, so we may subtract these two integrals.
		Since $F(u(z))\geqslant F(u_{\bar\lambda(0)}(z))$ for $|z|>\bar\lambda(0)$,
		the positivity of $K_\alpha$ gives, for $|y|>\bar\lambda(0)$,
		\begin{equation*}
			v(y)-v_{0,\bar\lambda(0)}(y)
			=\int_{|z|>\bar\lambda(0)}K_\alpha(0,\bar\lambda(0);y,z)
			\bigl(F(u(z))-F(u_{\bar\lambda(0)}(z))\bigr)\,\ud z\geqslant0.
		\end{equation*}
		By Proposition~\ref{prop:conformalinvariance},
		$-\Delta u_{\bar\lambda(0)}=v_{0,\bar\lambda(0)} u_{\bar\lambda(0)}^{p-1}$.
		Subtracting this equation from the equation for $u$ and using the comparison above, we obtain
		\begin{equation*}
			-\Delta d_{\bar\lambda(0)}
			=v\bigl(u^{p-1}-u_{\bar\lambda(0)}^{p-1}\bigr)
			+(v-v_{0,\bar\lambda(0)})u_{\bar\lambda(0)}^{p-1}\geqslant0
			\quad\text{in }\mathbb R^n\setminus\overline B_{\bar\lambda(0)}.
		\end{equation*}
		The strong minimum principle on the connected exterior domain implies that
		either $d_{\bar\lambda(0)}\equiv0$ or $d_{\bar\lambda(0)}>0$ on
		$\mathbb R^n\setminus\overline B_{\bar\lambda(0)}$. Suppose that $d_{\bar\lambda(0)}>0$ there.
		We will show that the sphere can then move beyond $\bar\lambda(0)$, a contradiction.
		For this, we consider the region near the sphere, the region near infinity, and the annulus between them.
		
		We first consider the region near the sphere.
		At $y=\bar\lambda(0)\theta$, the outward unit normal of the exterior domain
		is $-\theta$. Hopf's boundary lemma implies
		$-\partial_r d_{\bar\lambda(0)}(\bar\lambda(0)\theta)<0$, and hence
		$\partial_r d_{\bar\lambda(0)}(\bar\lambda(0)\theta)>0$ for every
		$\theta\in\mathbb S^{n-1}$. Using the compactness of the sphere and the continuity of the radial derivative, we can choose $c_0,\delta>0$ such that
        \begin{equation*}
        \partial_r d_{\bar\lambda(0)}(r\theta)\geqslant2c_0
        \qquad\text{for }\theta\in\mathbb S^{n-1},\quad
        \bar\lambda(0)\leqslant r\leqslant\bar\lambda(0)+3\delta.
        \end{equation*}
		Since $u\in\mathcal C^2$, the map
		$(\lambda,r,\theta)\mapsto\partial_\lambda\partial_r u_\lambda(r\theta)$
		is continuous for
		\begin{equation*}
			\lambda\in[\bar\lambda(0),\bar\lambda(0)+\delta],\qquad
			r\in[\bar\lambda(0),\bar\lambda(0)+3\delta],\qquad
			\theta\in\mathbb S^{n-1}.
		\end{equation*}
		Since this parameter set is compact, there exists a constant $M<\infty$
		that bounds the absolute value of this derivative for all these parameters.
		Integrating with respect to $\lambda$, we obtain
		\begin{equation*}
			\left|\partial_r(u_\lambda-u_{\bar\lambda(0)})(r\theta)\right|
			=\left|\int_{\bar\lambda(0)}^\lambda\partial_s\partial_r u_s(r\theta)\,\ud s\right|
			\leqslant M(\lambda-\bar\lambda(0)).
		\end{equation*}
		Choose $0<\varepsilon_1<\delta$ so that $M\varepsilon_1\leqslant c_0$.
		Then $\partial_r d_\lambda\geqslant2c_0-M\varepsilon_1\geqslant c_0$
		whenever $\bar\lambda(0)\leqslant\lambda\leqslant\bar\lambda(0)+\varepsilon_1$ and
		$\lambda\leqslant r\leqslant\bar\lambda(0)+2\delta$. Since
		$d_\lambda(\lambda\theta)=0$, it follows that
		\begin{equation}\label{prelim:continuationcollar}
			d_\lambda(r\theta)
			=\int_\lambda^r\partial_t d_\lambda(t\theta)\,\ud t
			\geqslant c_0(r-\lambda)\geqslant0 \quad \text{for }\bar\lambda(0)\leqslant\lambda\leqslant\bar\lambda(0)+\varepsilon_1,
			\quad\lambda\leqslant r\leqslant\bar\lambda(0)+2\delta.
		\end{equation}
		
		For the comparison near infinity, fix $R_0>\bar\lambda(0)+2\delta$ and let
		$m_0:=\min_{|y|=R_0}d_{\bar\lambda(0)}(y)>0$. Since $d_{\bar\lambda(0)}>0$ and
		$-\Delta d_{\bar\lambda(0)}\geqslant0$ on $|y|>\bar\lambda(0)$, the annulus comparison
		used to prove~\eqref{prelim:superharmoniclower} yields
		\begin{equation*}
			d_{\bar\lambda(0)}(r\theta)\geqslant m_0R_0^{n-2}r^{2-n},\qquad r\geqslant R_0.
		\end{equation*}
		To control how $u_\lambda$ changes with $\lambda$, we differentiate the formula in
		Definition~\ref{def:kelvintransform} for $\lambda\in[\bar\lambda(0),\bar\lambda(0)+\delta]$:
		\begin{equation*}
			\partial_\lambda u_\lambda(r\theta)
			=r^{2-n}\left[(n-2)\lambda^{n-3}
			u\left(\frac{\lambda^2}{r}\theta\right)
			+\frac{2\lambda^{n-1}}r
			\nabla u\left(\frac{\lambda^2}{r}\theta\right)\cdot\theta\right].
		\end{equation*}
		For $r\geqslant R_0$, the arguments of $u$ and $\nabla u$ lie in
		the fixed compact ball $\overline B_{(\bar\lambda(0)+\delta)^2/R_0}$.
		Since also $\lambda\leqslant\bar\lambda(0)+\delta$ and $r^{-1}\leqslant R_0^{-1}$,
		the expression in brackets is bounded in absolute value by a constant
		$M_\infty<\infty$, independent of $\lambda,r,\theta$.
		Integrating with respect to $\lambda$, we obtain
		\begin{equation*}
			|u_\lambda(r\theta)-u_{\bar\lambda(0)}(r\theta)|
			\leqslant M_\infty(\lambda-\bar\lambda(0))r^{2-n}
			\quad\text{for }r\geqslant R_0,\quad \bar\lambda(0)\leqslant\lambda\leqslant\bar\lambda(0)+\delta.
		\end{equation*}
		Choose $0<\varepsilon_2\leqslant\delta$ so that
		$M_\infty\varepsilon_2\leqslant m_0R_0^{n-2}/2$.
		For $r\geqslant R_0$ and $\bar\lambda(0)\leqslant\lambda\leqslant\bar\lambda(0)+\varepsilon_2$,
		we then have
		\begin{equation}\label{prelim:continuationinfinity}
			d_\lambda(r\theta)
			\geqslant d_{\bar\lambda(0)}(r\theta)-|u_\lambda(r\theta)-u_{\bar\lambda(0)}(r\theta)|\geqslant\tfrac12m_0R_0^{n-2}r^{2-n}>0.
		\end{equation}
		
		Finally, on the remaining compact annulus
		$A:=\{y:\bar\lambda(0)+2\delta\leqslant|y|\leqslant R_0\}$, set
		$m_1:=\min_A d_{\bar\lambda(0)}>0$.
		Uniform continuity of $(\lambda,y)\mapsto u_\lambda(y)$ on
		$[\bar\lambda(0),\bar\lambda(0)+\delta]\times A$ allows us to choose $\varepsilon_3\in(0,\delta]$ such that
		$\sup_A|u_\lambda-u_{\bar\lambda(0)}|\leqslant m_1/2$ whenever
		$\bar\lambda(0)\leqslant\lambda\leqslant\bar\lambda(0)+\varepsilon_3$. Hence
		\begin{equation}\label{prelim:continuationcompact}
			d_\lambda(y)\geqslant m_1-|u_\lambda(y)-u_{\bar\lambda(0)}(y)|
			\geqslant m_1/2>0,\qquad y\in A.
		\end{equation}
		
		Take $\varepsilon:=\min\{\varepsilon_1,\varepsilon_2,\varepsilon_3\}>0$.
		For $\lambda\in[\bar\lambda(0),\bar\lambda(0)+\varepsilon)$, we have
		$\lambda<\bar\lambda(0)+\delta<\bar\lambda(0)+2\delta<R_0$.
		The three estimates \eqref{prelim:continuationcollar}, \eqref{prelim:continuationcompact}, and \eqref{prelim:continuationinfinity} apply, respectively, to
		$\lambda\leqslant|y|\leqslant\bar\lambda(0)+2\delta$,
		$\bar\lambda(0)+2\delta\leqslant|y|\leqslant R_0$, and $|y|\geqslant R_0$.
		Hence $u_\lambda(y)\leqslant u(y)$ for all $|y|\geqslant\lambda$.
		Since the comparison is already known for $0<\lambda\leqslant\bar\lambda(0)$,
		it holds for every $0<\lambda\leqslant\bar\lambda(0)+\varepsilon/2$,
		contradicting the definition of $\bar\lambda(0)$ in~\eqref{criticalradius}.
		We conclude that $u_{\bar\lambda(0)}=u$ on $|y|>\bar\lambda(0)$.
		The equality also holds on $\partial B_{\bar\lambda(0)}$.
		For $0<|z|<\bar\lambda(0)$, put $y=\mathcal I_{0,\bar\lambda(0)}(z)$; then $|y|>\bar\lambda(0)$ and $|y||z|=\bar\lambda(0)^2$.
		The exterior equality yields
		\begin{equation*}
			u_{\bar\lambda(0)}(z)
			=\left(\frac{\bar\lambda(0)}{|z|}\right)^{n-2}u(y)
			=\left(\frac{\bar\lambda(0)^2}{|z||y|}\right)^{n-2}u(z)=u(z).
		\end{equation*}
		Returning to the original center proves the first assertion.
		
		We now prove that an infinite critical radius at one center implies an infinite critical radius at every center. Suppose $\bar\lambda(x_0)=\infty$.
		By~\eqref{criticalradius}, $u_{x_0,\lambda}(y)\leqslant u(y)$ for $|y-x_0|\geqslant\lambda$
		and every $\lambda>0$.
		Multiplying by $|y|^{n-2}$, we obtain
		\begin{equation*}
			|y|^{n-2}u(y)
			\geqslant\lambda^{n-2}
			\left(\frac{|y|}{|y-x_0|}\right)^{n-2}
			u\bigl(\mathcal I_{x_0,\lambda}(y)\bigr).
		\end{equation*}
		As $|y|\to\infty$, the ratio $|y|/|y-x_0|$ tends to $1$ and
		$\mathcal I_{x_0,\lambda}(y)\to x_0$. Continuity of $u$ therefore yields
		\begin{equation*}
			\liminf_{|y|\to\infty}|y|^{n-2}u(y)
			\geqslant\lambda^{n-2}u(x_0).
		\end{equation*}
		Since this holds for every $\lambda>0$, and $u(x_0)>0$ and $n-2>0$,
		the lower limit is $+\infty$.
		If $\bar\lambda(x)<\infty$ at some center $x$, the first assertion
		implies $u=u_{x,\bar\lambda(x)}$ on $\mathbb R^n\setminus\{x\}$.
		Passing to the limit in this equality, we find
		\begin{equation*}
			\lim_{|y|\to\infty}|y|^{n-2}u(y)
			=\bar\lambda(x)^{n-2}u(x)<\infty,
		\end{equation*}
		a contradiction. Thus $\bar\lambda(x)=\infty$ at every center.
	\end{proof}
	
	\begin{proof}[Proof of Theorem~\ref{thmB}]
		If $\bar\lambda(x_0)=\infty$ at some center, Lemma~\ref{lemmaA.2}
		implies $\bar\lambda(x)=\infty$ for every $x\in\mathbb R^n$.
		For each $x$ and $\mu>0$, we may therefore choose an admissible
		$L>\mu$ in \eqref{criticalradius}, obtaining
		\begin{equation*}
			u_{x,\mu}(y)\leqslant u(y)
			\quad\text{for every }x\in\mathbb R^n,\ \mu>0,
			\ |y-x|\geqslant\mu.
		\end{equation*}
		Lemma~\ref{lm:msfundamentallemma}(ii), with $\nu=n-2$, implies that
		$u$ is a positive constant. Its Hartree potential is infinite,
		since $\int_1^\infty r^{\alpha-1}\,\ud r=\infty$, contradicting
		the finiteness of $v_{\mathbb R^n}[u]$. Thus $\bar\lambda(x)<\infty$
		at every center. Lemma~\ref{lemmaA.2} ensures Kelvin equality at each
		center, and Lemma~\ref{lm:msfundamentallemma}(i), again with $\nu=n-2$,
		allows us to write
		\begin{equation*}
			u=A\phi_\mu,\qquad
			\phi_\mu(x):=\left(\frac{\mu}{1+\mu^2|x-x_0|^2}\right)^{(n-2)/2}
		\end{equation*}
		for some $A,\mu>0$ and $x_0\in\mathbb R^n$.
		
		To determine $A$, we use the identities
		\begin{equation*}
			-\Delta\phi_\mu=n(n-2)\phi_\mu^{(n+2)/(n-2)},
			\qquad
			\mathcal R_\alpha\ast\phi_\mu^p
			=J_{n,\alpha}\phi_\mu^{(n-\alpha)/(n-2)},
		\end{equation*}
		where $J_{n,\alpha}=\pi^{n/2}\Gamma(\alpha/2)/\Gamma((n+\alpha)/2)$.
		The first follows by direct differentiation. For the second,
		translation and scaling reduce the calculation to $\mu=1$ and
		$x_0=0$. Set $a=(n-\alpha)/2$ and $b=(n+\alpha)/2$, so that
		$a,b>0$ and $a+b=n$. For $y\ne x$, the beta-parameter identity takes the form
		\begin{equation*}
			\frac{1}{|x-y|^{2a}(1+|y|^2)^b}
			=\frac{\Gamma(n)}{\Gamma(a)\Gamma(b)}
			\int_0^1\frac{t^{a-1}(1-t)^{b-1}}
			{[t|x-y|^2+(1-t)(1+|y|^2)]^n}\,\ud t.
		\end{equation*}
		Since
		\begin{equation*}
			t|x-y|^2+(1-t)(1+|y|^2)
			=|y-tx|^2+(1-t)(1+t|x|^2),
		\end{equation*}
		we integrate first with respect to $y$ to obtain
		\begin{equation*}
			\begin{aligned}
				\int_{\mathbb R^n}\frac{|x-y|^{\alpha-n}}{(1+|y|^2)^b}\,\ud y
				&=\frac{\pi^{n/2}\Gamma(n/2)}{\Gamma(a)\Gamma(b)}
				\int_0^1\frac{t^{a-1}(1-t)^{\alpha/2-1}}
				{(1+t|x|^2)^{n/2}}\,\ud t\\
				&=J_{n,\alpha}(1+|x|^2)^{-a}.
			\end{aligned}
		\end{equation*}
		By Tonelli's theorem, we may change the order of integration. The
		last equality follows from $s=t(1+|x|^2)/(1+t|x|^2)$ and the
		beta integral.
		
		Since $(n-\alpha)+(n-2)(p-1)=n+2$ and $F(t)=t^p/p$ for $t>0$,
		substitution of $u=A\phi_\mu$ into \eqref{ourlimitPDEnonsing} yields
		\begin{equation*}
			n(n-2)A=\frac{J_{n,\alpha}}{p}A^{2p-1},
			\qquad
			A=\left(\frac{pn(n-2)}{J_{n,\alpha}}\right)^{1/(2p-2)}=C_n(\alpha).
		\end{equation*}
		This determines the constant $C_n(\alpha)$ defined in~\eqref{alphaconstant} and proves the asserted formula. Conversely, the same identities
		show that $C_n(\alpha)\phi_\mu$ has finite Hartree potential and
		satisfies \eqref{ourlimitPDEnonsing} for every $x_0\in\mathbb R^n$
		and $\mu>0$, completing the classification.
	\end{proof}
	
	\subsection{An auxiliary Liouville lemma}\label{subsec:constantdefect}
	In the blow-up argument of Subsection~\ref{subsec:upperbound}, the limit of the Hartree potentials may differ from the Hartree potential of the limit by a nonnegative constant. The following nonexistence result will rule out a positive value of this constant.
	
	\begin{lemma}\label{lem:constantdefect}
		Let $n\geqslant3$, $0<\alpha<n$, $p=(n+\alpha)/(n-2)$, and $F(t)=t^p/p$ for $t>0$. For every constant $b_0>0$, the equation
		\begin{equation}\label{prelim:defectequation}
			-\Delta u(y)=\left(\int_{\mathbb R^n}|y-z|^{\alpha-n}F(u(z))\,\ud z+b_0\right)u(y)^{p-1}
			\quad\text{in }\mathbb R^n
		\end{equation}
		has no positive classical solution $u\in C^2(\mathbb R^n)$.
	\end{lemma}
	\begin{proof}
		Suppose, to the contrary, that such a solution $u$ exists, and set
		$v:=v_{\mathbb R^n}[u]+b_0$. Since $u\in C^2(\mathbb R^n)$ is positive and satisfies $-\Delta u=vu^{p-1}>0$, we may start the moving spheres as in the proof of Lemma~\ref{lemmaA.1}. Thus, for each $x\in\mathbb R^n$, we have $u_{x,\lambda}(y)\leqslant u(y)$ for $|y-x|\geqslant\lambda$ and all sufficiently small $\lambda>0$.
		In particular, the critical radius $\bar\lambda(x)$ defined in
		\eqref{criticalradius} is positive.
		
		We claim that $\bar\lambda(x)=\infty$ for every $x\in\mathbb R^n$.
		Since the equation is invariant under translations, it suffices to
		prove the claim for $x=0$. As in the proof of Lemma~\ref{lemmaA.2}, write $u_\lambda:=u_{0,\lambda}$ and $d_\lambda:=u-u_\lambda$, and let $v_{0,\lambda}$ be the transform in \eqref{prelim:potentialkelvin} with center $0$, applied to the whole function $v$, including $b_0$. It follows from \eqref{prelim:potentialcovariance} that
		\begin{equation*}
			v_{0,\lambda}(y)=\int_{\mathbb R^n}|y-z|^{\alpha-n}F(u_\lambda(z))\,\ud z
			+b_0\left(\frac{\lambda}{|y|}\right)^{n-\alpha}.
		\end{equation*}
		Moreover, since $(n-\alpha)+(n-2)(p-1)=n+2$, formula \eqref{prelim:laplacekelvin} gives
		\begin{equation*}
			-\Delta u_\lambda=v_{0,\lambda} u_\lambda^{p-1}\quad\text{in }\mathbb R^n\setminus\{0\}.
		\end{equation*}
		By the same calculation as in the proof of Lemma~\ref{lemmaA.2}, we have, for $|y|>\lambda$,
		\begin{equation}\label{prelim:defectdifference}
			\begin{aligned}
				v(y)-v_{0,\lambda}(y)
				&=\int_{|z|>\lambda}K_\alpha(0,\lambda;y,z)
				[F(u(z))-F(u_\lambda(z))]\,\ud z\\
				&\quad+b_0\left[1-\left(\frac{\lambda}{|y|}\right)^{n-\alpha}\right].
			\end{aligned}
		\end{equation}
		Suppose now that $\bar\lambda:=\bar\lambda(0)<\infty$. For each fixed $y$ with $|y|\geqslant\bar\lambda$, continuity of the Kelvin transform allows us to let $\lambda\uparrow\bar\lambda$ in $u_\lambda(y)\leqslant u(y)$, so that $u_{\bar\lambda}(y)\leqslant u(y)$ for $|y|\geqslant \bar{\lambda}$. Together with the positivity of $K_\alpha$ in Lemma~\ref{lm:differencecomparison} and the monotonicity of $F$, this implies, by \eqref{prelim:defectdifference},
		\begin{equation*}
			v(y)-v_{0,\bar\lambda}(y)
			\geqslant b_0\left[1-\left(\frac{\bar\lambda}{|y|}\right)^{n-\alpha}\right]>0
			\quad\text{for }|y|>\bar\lambda.
		\end{equation*}
		Subtracting the differential equations for $u$ and $u_{\bar\lambda}$, we therefore have, on $|y|>\bar\lambda$,
		\begin{equation*}
			\begin{aligned}
				-\Delta d_{\bar\lambda}
				&=v\bigl(u^{p-1}-u_{\bar\lambda}^{p-1}\bigr)
				+(v-v_{0,\bar\lambda})u_{\bar\lambda}^{p-1}>0.
			\end{aligned}
		\end{equation*}
		Since $d_{\bar\lambda}\geqslant0$, a point $y_0$ with $|y_0|>\bar\lambda$ and $d_{\bar\lambda}(y_0)=0$ would be a local minimum. At this minimum, $\Delta d_{\bar\lambda}(y_0)\geqslant0$, contrary to $\Delta d_{\bar\lambda}<0$. Hence $d_{\bar\lambda}>0$ for $|y|>\bar\lambda$. By definition of the Kelvin transform, $d_{\bar\lambda}=0$ on $\partial B_{\bar\lambda}$.
		
		By Hopf's boundary lemma and compactness of $\mathbb S^{n-1}$,
		\begin{equation*}
			m:=\min_{\theta\in\mathbb S^{n-1}}
			\partial_r d_{\bar\lambda}(\bar\lambda\theta)>0.
		\end{equation*}
		By continuity, we can choose $\delta>0$ such that, for every
		$\bar\lambda\leqslant\lambda\leqslant\bar\lambda+\delta$
		and $\theta\in\mathbb S^{n-1}$,
		\begin{equation*}
			\partial_r d_\lambda(r\theta)\geqslant\frac m2
			\quad\text{for }\lambda\leqslant r\leqslant\bar\lambda+2\delta.
		\end{equation*}
		Since $d_\lambda(\lambda\theta)=0$, it follows that
		\begin{equation*}
			d_\lambda(r\theta)
			=\int_\lambda^r\partial_s d_\lambda(s\theta)\,\ud s
			\geqslant\frac m2(r-\lambda)\geqslant0
			\quad\text{for }\lambda\leqslant r\leqslant\bar\lambda+2\delta.
		\end{equation*}
		
		Fix $R>\bar\lambda+2\delta$ and set
		$m_R:=\min_{|y|=R}d_{\bar\lambda}(y)>0$.
		Applying the annulus comparison used in the proof of
		\eqref{prelim:superharmoniclower}, we find
		\begin{equation*}
			d_{\bar\lambda}(y)\geqslant m_RR^{n-2}|y|^{2-n}
			\quad\text{for }|y|\geqslant R.
		\end{equation*}
		The estimate for $u_\lambda$ in the proof of
		Lemma~\ref{lemmaA.2} also applies here: for $r\geqslant R$,
		$\theta\in\mathbb S^{n-1}$ and
		$\bar\lambda\leqslant\lambda\leqslant\bar\lambda+\delta$,
		\begin{equation*}
			|u_\lambda(r\theta)-u_{\bar\lambda}(r\theta)|
			\leqslant C(\lambda-\bar\lambda)r^{2-n},
		\end{equation*}
		where $C$ is independent of $r,\theta,\lambda$ in these ranges. Combining these estimates and using $d_\lambda=d_{\bar\lambda}-(u_\lambda-u_{\bar\lambda})$, we see that, if $\varepsilon\in(0,\delta]$ is chosen so that $C\varepsilon\leqslant m_RR^{n-2}/2$, then
		\begin{equation*}
			d_\lambda(y)\geqslant\tfrac12m_RR^{n-2}|y|^{2-n}>0
			\quad\text{for }|y|\geqslant R,
			\quad\bar\lambda\leqslant\lambda\leqslant\bar\lambda+\varepsilon.
		\end{equation*}
		
		Finally, on the remaining compact annulus $\bar\lambda+2\delta\leqslant|y|\leqslant R$, the function $d_{\bar\lambda}$ has a positive minimum. Since $d_\lambda\to d_{\bar\lambda}$ uniformly on this annulus as $\lambda\downarrow\bar\lambda$, we may decrease $\varepsilon>0$, if necessary, so that $d_\lambda>0$ there for every $\bar\lambda\leqslant\lambda\leqslant\bar\lambda+\varepsilon$.
		
		The comparisons near the sphere, near infinity, and on this annulus together show that
		\begin{equation*}
			u_\lambda(y)\leqslant u(y)
			\quad\text{for }|y|\geqslant\lambda,\quad
			\bar\lambda\leqslant\lambda\leqslant\bar\lambda+\varepsilon.
		\end{equation*}
		Combined with the comparisons already known for $0<\lambda\leqslant\bar\lambda$, this contradicts the definition of $\bar\lambda$.
		
		Since the center was arbitrary, we conclude that $\bar\lambda(x)=\infty$ for every $x$. Lemma~\ref{lm:msfundamentallemma}(ii), with $\nu=n-2$, therefore implies that $u$ is a positive constant, say $u\equiv a>0$. Substituting this into \eqref{prelim:defectequation}, however, gives $0=va^{p-1}\geqslant b_0a^{p-1}>0$. This contradiction completes the proof.
	\end{proof}
	
	\section{Radial symmetry of singular solutions}\label{sec:existence}
	This section proves Theorem~\ref{thm1}. We first establish the integral representation \eqref{ourlimitPDEdual} for positive solutions of \eqref{ourlimitPDE} in $\mathcal C^2(\mathbb R^n\setminus\{0\})$. We then use this representation to prove that  positive solutions satisfying \eqref{eq:blowup} are radially symmetric about the origin and are nonincreasing in the radial variable.
	
	\subsection{Integral representation}
	Throughout this subsection, let $n\geqslant3$ and $0<\alpha<n$, and let $u\in C^2(\mathbb R^n\setminus\{0\})$ be a positive classical solution of \eqref{ourlimitPDE}. We write
	\begin{equation*}
		v=v_{\mathbb R^n}[u],\qquad g=vf(u)=-\Delta u.
	\end{equation*}
	Equation~\eqref{bound of the Hartree integral} gives $v(x)<\infty$ for $x\ne0$, while $g=-\Delta u$ is continuous away from the origin. The results in this subsection apply to all positive classical solutions of \eqref{ourlimitPDE}, without assuming \eqref{eq:blowup}.
	
	Following \cite{MR4420104,MR4304557}, we use the weighted space
	\begin{equation*}
		L_0(\mathbb R^n):=\left\{u\in L^1_{\loc}(\mathbb R^n):
		\int_{\mathbb R^n}\frac{|u(x)|}{1+|x|^n}\,\ud x<\infty\right\}.
	\end{equation*}
	
	\begin{definition}\label{def:distributionalsolutions}
		Let $n\geqslant 3$  and $\alpha\in(0,n)$.
		A function $u\in\mathcal C^0(\mathbb R^n\setminus\{0\})\cap L^1_{\rm loc}(\mathbb R^n)$ is a distributional solution of \eqref{ourlimitPDE} on $\mathbb R^n$ if $(\mathcal R_\alpha*F(u))f(u)\in L^1_{\rm loc}(\mathbb R^n)$ and
		\begin{flalign*}
			-\int_{\mathbb R^n}u\Delta\varphi\,\ud x
			=\int_{\mathbb R^n}(\mathcal R_\alpha*F(u))f(u)\varphi\,\ud x,
		\end{flalign*}
		for every $\varphi\in\mathcal C_c^\infty(\mathbb R^n)$.
	\end{definition}
	
	\begin{lemma}\label{globaldistributional}
		We have
		\begin{equation*}
			\int_{\mathbb R^n}(1+|z|)^{\alpha-n}F(u(z))\,\ud z<\infty,
			\qquad u\in L^p_{\loc}(\mathbb R^n)\cap L_0(\mathbb R^n),\qquad g\in L^1_{\loc}(\mathbb R^n).
		\end{equation*}
		Moreover,
		\begin{equation}\label{global:global-distribution}
			-\Delta u=g\quad\text{in }\mathcal D'(\mathbb R^n).
		\end{equation}
	\end{lemma}
	
	\begin{proof}
		Fix a unit vector $e\in\mathbb R^n$. Since $\alpha<n$ and $|e-z|\leqslant1+|z|$,
		\begin{equation*}
			\int_{\mathbb R^n}(1+|z|)^{\alpha-n}F(u(z))\,\ud z
			\leqslant\int_{\mathbb R^n}|e-z|^{\alpha-n}F(u(z))\,\ud z
			=v(e)<\infty.
		\end{equation*}
		In particular, for every $R>0$,
		\begin{equation*}
			\int_{B_R}u^p\,\ud x\leqslant p(1+R)^{n-\alpha}v(e)<\infty.
		\end{equation*}
		Thus $u\in L^p_{\loc}(\mathbb R^n)\subset L^1_{\loc}(\mathbb R^n)$.
		Moreover, H\"older's inequality yields
		\begin{align*}
			\int_{\mathbb R^n}\frac{u(x)}{(1+|x|)^n}\,\ud x
			&\leqslant
			\left(\int_{\mathbb R^n}(1+|x|)^{\alpha-n}u(x)^p\,\ud x\right)^{1/p}\\
			&\quad\times
			\left(\int_{\mathbb R^n}(1+|x|)^{-n-\alpha/(p-1)}\,\ud x\right)^{(p-1)/p}
			<\infty,
		\end{align*}
		because $\alpha>0$ and $p>1$. The comparability of $(1+|x|)^n$ and $1+|x|^n$ then implies $u\in L_0(\mathbb R^n)$.
		
		Let $\chi\in C^\infty(\mathbb R)$ satisfy $0\leqslant\chi\leqslant1$, $\chi(t)=0$ for $t\leqslant1$, and $\chi(t)=1$ for $t\geqslant2$. For any $b>0$ and $0<\varepsilon<\min\{b/4,1\}$, set
		\begin{equation*}
			\eta_\varepsilon(x)=\chi(|x|/\varepsilon),\qquad
			\zeta_b(x)=1-\chi(2|x|/b-1).
		\end{equation*}
		Then $\zeta_b\in C_c^\infty(B_{2b})$, $\zeta_b=1$ on $B_b$, and $0\leqslant\zeta_b\leqslant1$. The derivatives of $\eta_\varepsilon$ are supported in $\overline{B_{2\varepsilon}}\setminus B_\varepsilon$, with
		\begin{equation*}
			|\nabla\eta_\varepsilon|\leqslant C\varepsilon^{-1},\qquad
			|\Delta\eta_\varepsilon|\leqslant C\varepsilon^{-2}.
		\end{equation*}
		Here and below, $C$ is independent of $\varepsilon$. By H\"older's inequality,
		\begin{align}
			\varepsilon^{-2}\int_{B_{2\varepsilon}\setminus B_\varepsilon}u\,\ud x
			&\leqslant C\varepsilon^{n(1-\frac{1}{p})-2}\|u\|_{L^p(B_{2b})}\notag\\
			&=C\varepsilon^{\frac{\alpha(n-2)}{(n+\alpha)}}\|u\|_{L^p(B_{2b})} \to
			0\quad\text{as }\varepsilon \downarrow 0.
			\label{global:annulus}
		\end{align}
		
		Testing \eqref{ourlimitPDE} with $\zeta_b\eta_\varepsilon$ yields
		\begin{align*}
			0\leqslant\int_{\mathbb R^n}g\zeta_b\eta_\varepsilon\,\ud x
			&=-\int_{\mathbb R^n}u\eta_\varepsilon\Delta\zeta_b\,\ud x
			-\int_{\mathbb R^n}u\zeta_b\Delta\eta_\varepsilon\,\ud x\\
			&\leqslant\int_{B_{2b}}u|\Delta\zeta_b|\,\ud x
			+C\varepsilon^{-2}\int_{B_{2\varepsilon}\setminus B_\varepsilon}u\,\ud x,
		\end{align*}
		since $\zeta_b\eta_\varepsilon\in C_c^\infty(\mathbb R^n\setminus\{0\})$ and $\nabla\zeta_b\cdot\nabla\eta_\varepsilon=0$. As $g\geqslant0$ and $\eta_\varepsilon\to1$ away from the origin, Fatou's lemma and \eqref{global:annulus} imply
		\begin{equation*}
			\int_{B_b}g\,\ud x
			\leqslant\int_{\mathbb R^n}g\zeta_b\,\ud x
			\leqslant\liminf_{\varepsilon\downarrow0}\int_{\mathbb R^n}g\zeta_b\eta_\varepsilon\,\ud x
			\leqslant\int_{B_{2b}}u|\Delta\zeta_b|\,\ud x<\infty.
		\end{equation*}
		This implies that $g\in L^1_{\loc}(\mathbb R^n)$.
		
		Finally, let $\varphi\in C_c^\infty(\mathbb R^n)$ and choose $b>0$ so that $\supp\varphi\subset B_b$. Testing with $\varphi\eta_\varepsilon$, we obtain
		\begin{align*}
			\int_{\mathbb R^n}g\varphi\eta_\varepsilon\,\ud x
			={}&-\int_{\mathbb R^n}u\eta_\varepsilon\Delta\varphi\,\ud x
			-2\int_{\mathbb R^n}u\nabla\varphi\cdot\nabla\eta_\varepsilon\,\ud x\\
			&-\int_{\mathbb R^n}u\varphi\Delta\eta_\varepsilon\,\ud x.
		\end{align*}
		The last two terms tend to zero by \eqref{global:annulus}. Using $u,g\in L^1_{\loc}(\mathbb R^n)$, we apply dominated convergence to the remaining terms and obtain \eqref{global:global-distribution}.
	\end{proof}
	
	\begin{lemma}\label{uandmL0}
		Define
		\begin{equation*}
			m_u(x):=\int_{\mathbb R^n}\Phi_n(x-z)g(z)\,\ud z,
			\qquad \text{ for }x\ne0.
		\end{equation*}
		Then $m_u$ is finite and satisfies
		\begin{equation}\label{global:mleu}
			0\leqslant m_u(x)\leqslant u(x)\qquad  \text{ for }x\ne0.
		\end{equation}
		Moreover, $m_u\in C(\mathbb R^n\setminus\{0\})\cap L_0(\mathbb R^n)$,
		after assigning any finite value at the origin.
	\end{lemma}
	
	\begin{proof}
		Fix $b>0$. Let $G_b$ be the Dirichlet Green kernel
		for $-\Delta$ on $B_b$, and let $H_b$ be the corresponding
		Poisson kernel. Explicitly, for $x,z\in B_b$ with $x\ne z$ and
		$\xi\in\partial B_b$
		\begin{equation}\label{Gb(x,z) of mu}
			G_b(x,z)
			:=\Phi_n(x-z)
			-c_n\left(
			b^2-2x\cdot z+\frac{|x|^2|z|^2}{b^2}
			\right)^{\frac{2-n}{2}}
		\end{equation}
		and
		\begin{equation}\label{global:Poissonkernel}
			H_b(x,\xi)
			:=\frac{b^2-|x|^2}
			{|\mathbb S^{n-1}|\,b\,|x-\xi|^n}.
		\end{equation}
		
		Set
		\begin{equation}\label{global:Greenpotential}
			V_b(x):=\int_{B_b}G_b(x,z)g(z)\,\ud z,
			\qquad x\in B_b.
		\end{equation}
		Lemma~\ref{globaldistributional} implies $g\in L^1(B_b)$. Since $0\leqslant G_b(x,z)\leqslant\Phi_n(x-z)$ and $B_b-z\subset B_{2b}$ for $z\in B_b$,
		\begin{equation*}
			\int_{B_b}G_b(x,z)\,\ud x
			\leqslant\int_{B_{2b}}\Phi_n(y)\,\ud y
			=\frac{2b^2}{n-2},
		\end{equation*}
		which, combined with Tonelli's theorem, yields
		\begin{equation*}
			\int_{B_b}V_b(x)\,\ud x
			\leqslant\frac{2b^2}{n-2}\|g\|_{L^1(B_b)}<\infty.
		\end{equation*}
		For $\psi\in C_c^\infty(B_b)$, the same estimate gives
		\begin{equation*}
			\int_{B_b}\int_{B_b}G_b(x,z)g(z)|\Delta\psi(x)|\,\ud z\,\ud x
			\leqslant\frac{2b^2}{n-2}\|\Delta\psi\|_\infty\|g\|_{L^1(B_b)}<\infty.
		\end{equation*}
		Using Fubini's theorem and the defining distributional identity for the Green kernel, we obtain that
		\begin{align}
			\int_{B_b}V_b(x)(-\Delta\psi(x))\,\ud x
			&=\int_{B_b}g(z)\left[\int_{B_b}G_b(x,z)(-\Delta\psi(x))\,\ud x\right]\,\ud z\notag\\
			&=\int_{B_b}g(z)\psi(z)\,\ud z.
			\label{global:Vdistribution}
		\end{align}
		Thus $-\Delta V_b=g$ in $\mathcal D'(B_b)$.
		
		We next show that $V_b$ is finite and continuous on $B_{b}\setminus \{0\}$. Fix $x_0 \in B_b\setminus\{0\}$ and choose $0<\rho<\frac14\min\{|x_0|,b-|x_0|\}$. The function $g$ is bounded on $\overline{B_{2\rho}(x_0)}$. For $0<\delta<\rho$ and $|x-x_0|<\delta$,
		\begin{equation}\label{global:Vnear}
			0\leqslant\int_{B_{2\delta}(x_0)}G_b(x,z)g(z)\,\ud z
			\leqslant\left(\sup_{\overline{B_{2\rho}(x_0)}}g\right)
			\int_{B_{3\delta}}\Phi_n(y)\,\ud y
			\leqslant C\delta^2,
		\end{equation}
		where $C$ is independent of $x$ and $\delta$. On $B_b\setminus B_{2\delta}(x_0)$,
		\begin{equation*}
			0\leqslant G_b(x,z)g(z)\leqslant c_n\delta^{2-n}g(z).
		\end{equation*}
		The latter bound proves finiteness of the far integral and permits dominated convergence as $x\to x_0$, with $\delta$ fixed. Consequently, by the triangle inequality,
		\begin{equation*}
			\limsup_{x\to x_0}|V_b(x)-V_b(x_0)|\leqslant2C\delta^2.
		\end{equation*}
		Letting $\delta\downarrow0$, we obtain $V_b\in C(B_b\setminus\{0\})$.
		
		We also claim that $V_b \to 0$ at every point of $\partial B_b$. Fix $\xi\in\partial B_b$ and $0<\delta<b/4$.
		We split the integral defining $V_b(x)$ over
		$B_b\cap B_\delta(\xi)$ and $B_b\setminus B_\delta(\xi)$.
		For $x\in B_b$ with $|x-\xi|<\delta/2$, the boundedness
		of $g$ near $\partial B_b$ gives
		\begin{equation*}
			0\leqslant
			\int_{B_b\cap B_\delta(\xi)}G_b(x,z)g(z)\,\ud z
			\leqslant C\int_{B_{3\delta/2}}\Phi_n(y)\,\ud y
			\leqslant C\delta^2,
		\end{equation*}
		where $C$ is independent of $x$ and $\delta$. If $z\in B_b\setminus B_\delta(\xi)$, then $|x-z|\geqslant\delta/2$, and hence
        \begin{equation*}
        0\leqslant G_b(x,z)g(z)\leqslant c_n(\delta/2)^{2-n}g(z).
        \end{equation*}
        With $\delta$ fixed, the right-hand side is integrable and independent of $x$, while $G_b(x,z)\to0$ as $x\to\xi$ for each $z\in B_b$. Then dominated convergence yields
		\begin{equation*}
			\lim_{\substack{x\to\xi\\x\in B_b}}
			\int_{B_b\setminus B_\delta(\xi)}
			G_b(x,z)g(z)\,\ud z=0.
		\end{equation*}
		Thus $\limsup_{x\to\xi,\,x\in B_b}V_b(x)\leqslant C\delta^2$.
		Letting $\delta\downarrow0$ and using $V_b\geqslant0$, we obtain
		\begin{equation}\label{global:Vzero}
			\lim_{\substack{x\to\xi\\x\in B_b}}V_b(x)=0.
		\end{equation}
		
		Define
		\begin{equation}\label{global:Poissonextension}
			\mathcal H_b(x)
			:=\int_{\partial B_b}H_b(x,\xi)u(\xi)\,\ud\sigma(\xi),
			\qquad x\in B_b.
		\end{equation}
		The continuity of $u$ on $\partial B_b$ implies that
		$\mathcal H_b\in C^\infty(B_b)\cap C(\overline{B_b})$
		is harmonic in $B_b$ and equals $u$ on $\partial B_b$.
		By Lemma~\ref{globaldistributional} and
		\eqref{global:Vdistribution},
		\begin{equation*}
			-\Delta(u-V_b-\mathcal H_b)=0
			\quad\text{in }\mathcal D'(B_b).
		\end{equation*}
		Then, using Weyl's lemma and continuity away from the origin, we obtain that there exists a smooth harmonic function $W$ in $B_b$ such that
		\begin{equation*}
			W(x)=u(x)-V_b(x)-\mathcal H_b(x),
			\qquad 0<|x|<b.
		\end{equation*}
		Since $\mathcal H_b=u$ on $\partial B_b$ and $V_b$ tends to zero at the boundary according to \eqref{global:Vzero}, $W$ extends continuously to $\overline{B_b}$ with zero boundary values. The maximum principle therefore yields $W\equiv0$. Thus, 
		\begin{equation*}
			u(x)=\int_{B_b}G_b(x,z)g(z)\,\ud z
			+\int_{\partial B_b}H_b(x,\xi)u(\xi)\,\ud\sigma(\xi),
			\qquad 0<|x|<b.
		\end{equation*}
		
		We now let $b\to\infty$ to pass from the Green potential
		on $B_b$ to the Newton potential on $\mathbb R^n$. Fix $x\ne0$. The Poisson term is nonnegative, so
		\begin{equation}\label{global:Vleu}
			0\leqslant\int_{B_b}G_b(x,z)g(z)\,\ud z\leqslant u(x),\qquad \text{ for }b>|x|.
		\end{equation}
		For fixed $z\ne x$, the explicit formula \eqref{Gb(x,z) of mu} shows that $G_b(x,z)\to\Phi_n(x-z)$ as $b\to\infty$. After extending $G_b(x,\cdot)$ by zero outside $B_b$, we apply Fatou's lemma together with \eqref{global:Vleu} to obtain
        \begin{equation*}
        0\leqslant m_u(x)\leqslant\liminf_{b\to\infty}\int_{B_b}G_b(x,z)g(z)\,\ud z \leqslant u(x)<\infty,
        \end{equation*}
        which proves \eqref{global:mleu}.
		
		To prove continuity of $m_u$, fix $x_0\ne0$ and choose $0<\rho<|x_0|/4$. As in \eqref{global:Vnear}, for $0<\delta<\rho$ and $|x-x_0|<\delta$,
		\begin{equation}\label{global:mnear}
			0\leqslant\int_{B_{2\delta}(x_0)}\Phi_n(x-z)g(z)\,\ud z\leqslant C\delta^2,
		\end{equation}
		where $C$ depends on the bound of $g$ on $\overline{B_{2\rho}(x_0)}$, but not on $x$ or $\delta$. If $|z-x_0|\geqslant2\delta$, we have $|x-z|\geqslant\frac12|x_0-z|$, and thus
		\begin{equation*}
			0\leqslant\Phi_n(x-z)g(z)\leqslant2^{n-2}\Phi_n(x_0-z)g(z).
		\end{equation*}
		The right-hand side is integrable since $m_u(x_0)<\infty$. For fixed $\delta$, dominated convergence therefore gives
		\begin{equation*}
			\int_{\mathbb R^n\setminus B_{2\delta}(x_0)}\Phi_n(x-z)g(z)\,\ud z
			\longrightarrow
			\int_{\mathbb R^n\setminus B_{2\delta}(x_0)}\Phi_n(x_0-z)g(z)\,\ud z.
		\end{equation*}
		Combining this limit with \eqref{global:mnear}, applied at both $x$ and $x_0$, we find
		\begin{equation*}
			\limsup_{x\to x_0}|m_u(x)-m_u(x_0)|\leqslant2C\delta^2.
		\end{equation*}
		Letting $\delta\downarrow0$ proves continuity at $x_0$.
		Finally, $0\leqslant m_u\leqslant u$ and $u\in L_0(\mathbb R^n)$ imply
		$m_u\in L_0(\mathbb R^n)\subset L^1_{\loc}(\mathbb R^n)$.
	\end{proof}
	
	\begin{proposition}\label{prop:radialsymmetry}
		For every $x\ne0$,
		\begin{equation}\label{global:representation}
			u(x)=\int_{\mathbb R^n}\Phi_n(x-z)v(z)u(z)^{p-1}\,\ud z,
		\end{equation}
		and the integral is finite.
	\end{proposition}
	
	\begin{proof}
		By Lemma~\ref{uandmL0}, $m_u\in L^1_{\loc}(\mathbb R^n)\cap C(\mathbb R^n\setminus\{0\})$ and $0\leqslant m_u\leqslant u$ away from the origin. To identify its distributional Laplacian, we fix $\varphi\in C_c^\infty(\mathbb R^n)$ and choose $R>0$ such that $\supp\varphi\subset B_R$. Applying Tonelli's theorem and using the local integrability of $u$ from Lemma~\ref{globaldistributional}, we obtain that
		\begin{align*}
			\int_{\mathbb R^n}\int_{\mathbb R^n}\Phi_n(x-z)g(z)|\Delta\varphi(x)|\,\ud z\,\ud x
			&=\int_{\mathbb R^n}m_u(x)|\Delta\varphi(x)|\,\ud x\\
			&\leqslant\|\Delta\varphi\|_\infty\int_{B_R}u(x)\,\ud x<\infty.
		\end{align*}
		Thus, by Fubini's theorem and $-\Delta\Phi_n=\delta_0$, we get 
		\begin{align}
			\int_{\mathbb R^n}m_u(x)(-\Delta\varphi(x))\,\ud x
			&=\int_{\mathbb R^n}g(z)\left[\int_{\mathbb R^n}\Phi_n(x-z)(-\Delta\varphi(x))\,\ud x\right]\,\ud z\notag\\
			&=\int_{\mathbb R^n}g(z)\varphi(z)\,\ud z,
			\label{global:m_u-distribution}
		\end{align}
		which implies that $-\Delta m_u=g$ in $\mathcal D'(\mathbb R^n)$.
		
		Write $w=u-m_u$. By Lemmas~\ref{globaldistributional}--\ref{uandmL0}, $w\in L_0(\mathbb R^n)\subset L^1_{\loc}(\mathbb R^n)$. Subtracting \eqref{global:m_u-distribution} from the distributional identity in Lemma~\ref{globaldistributional} yields
		\begin{equation*}
			-\Delta w=0\qquad\text{in }\mathcal D'(\mathbb R^n).
		\end{equation*}
		Weyl's lemma then provides a smooth harmonic representative of $w$ on $\mathbb R^n$. Since $m_u\leqslant u$, the harmonic representative of $w$ is nonnegative. The classical Liouville theorem therefore implies that $w$ is constant almost everywhere. Since $w\in L_0(\mathbb R^n)$ and no nonzero constant
		belongs to this space, we conclude that $w=0$ almost everywhere. The continuity of $u$ and $m_u$ on $\mathbb R^n\setminus\{0\}$ therefore gives $u(x)=m_u(x)$ for every $x\ne0$. Finally, substituting $g=vu^{p-1}$ into the definition of $m_u$ proves \eqref{global:representation}.
	\end{proof}
	
	\subsection{Radial symmetry}\label{subsec:radialsymmetry}
	We now prove radial symmetry and radial monotonicity for the positive solutions satisfying \eqref{eq:blowup} in Theorem~\ref{thm1}. We use the integral moving-spheres method from \cite{MR3694645,MR2055032,arxiv:1901.01678,MR4304557}, comparing both $u$ and its Hartree potential under inversion.
	
	Write $v=v_{\mathbb R^n}[u]$, as in Section~\ref{sec:preliminaries}. By Proposition~\ref{prop:radialsymmetry}, we have the integral system
	\begin{flalign}\tag{$\mathcal S_{n,\alpha,\infty}^\prime$}\label{ourlimitsystemdual}
		\begin{cases}
			\begin{split}
				u&  = \Phi_n\ast\left[vf(u)\right]\\
				v&  = \mathcal{R}_\alpha\ast F(u)
			\end{split} \quad {\rm in} \quad \mathbb{R}^n\setminus\{0\}.
		\end{cases}
	\end{flalign}
	Both convolutions are over $\mathbb R^n$. We use the Kelvin transforms $u_{x,\mu}$ and $v_{x,\mu}$ from Definition~\ref{def:kelvintransform} and~\eqref{prelim:potentialkelvin}, respectively.
	For $x\ne0$ and $\mu>0$, we apply the change of variables from Proposition~\ref{prop:conformalinvariance} to each integral to obtain
	\begin{flalign*}
		\begin{cases}
			\begin{split}
				u_{x, \mu}(y)&  = c_n\int_{\mathbb R^n}\frac{1}{|z-y|^{n-2}}\left[v_{x, \mu}(z)f(u_{x, \mu}(z))\right]\,\ud z\\
				v_{x, \mu}(y)&  = \int_{\mathbb R^n}\frac{1}{|z-y|^{n-\alpha}} F(u_{x, \mu}(z))\,\ud z.
			\end{split}
		\end{cases}
	\end{flalign*}
	These identities hold on $\mathbb R^n\setminus\{x,\mathcal I_{x,\mu}(0)\}$. In what follows, $x\ne0$ and $0<\mu<|x|$, so $\mathcal I_{x,\mu}(0)\in B_\mu(x)$. Splitting the integrals at this sphere and inverting the interior part yields, for $|y-x|>\mu$ and $y\ne0$,
	\begin{equation}\label{u-uxmu}
		u(y)-u_{x,\mu}(y)=c_n\int_{|z-x|\geqslant\mu}K_2(x,\mu;y,z)\left[f(u(z))v(z)-f(u_{x,\mu}(z))v_{x,\mu}(z)\right]\,\ud z,
	\end{equation}
	and 
	\begin{equation}\label{v-vxmu}
		v(y)-v_{x,\mu}(y)=\int_{|z-x|\geqslant\mu}K_\alpha(x,\mu;y,z)\left[F(u(z))-F(u_{x,\mu}(z))\right]\,\ud z,
	\end{equation}
	where $K_2$ and $K_\alpha$ are the kernels in
	Lemma~\ref{lm:differencecomparison}.
	They are positive when $|y-x|,|z-x|>\mu$ and $y\ne z$.
	The constant $c_n$ in \eqref{u-uxmu} comes from
	$\Phi_n(z)=c_n|z|^{2-n}$. Both difference integrals converge absolutely by the kernel bounds and the finite integral representations of $u$, $v$, and their Kelvin transforms.

	\begin{proposition}\label{prop:symmetry}
		Let $n\geqslant 3$  and $\alpha\in(0,n)$.
		If $u\in\mathcal C^2(\mathbb R^n\setminus\{0\})$ is a positive classical solution of \eqref{ourlimitPDE} satisfying \eqref{eq:blowup}, then $u$ is radially symmetric about the origin and nonincreasing as a function of $|x|$.
	\end{proposition}
	
	\begin{proof}
		Proposition~\ref{prop:radialsymmetry} shows that $(u,v)$ satisfies \eqref{ourlimitsystemdual}.
		\noindent{\bf Claim 1:} For every $x \in \mathbb{R}^n \setminus \{0\}$, there exists $\mu_{2} \in \left(0,|x|\right)$ such that, for every $0<\mu<\mu_{2}$,
		\begin{equation*}
			u_{x,\mu}(y) \leqslant u(y) \quad {\rm for} \quad |y-x|\geqslant \mu \quad {\rm and} \quad y\in \mathbb{R}^{n} \setminus \{0\}.
		\end{equation*}
		
		Fix $x\ne0$. Since $u$ is positive and belongs to $C^2(\mathbb{R}^n\setminus\{0\})$, there exists $C_{\log}>0$ such that 
		$|\nabla\ln u|\leqslant C_{\log}$ in $B_{|x|/2}(x)$. Differentiating $r^{\frac{n-2}{2}}u(x+r\theta)$ with respect to $r$ shows that it is increasing for every $\theta\in\mathbb S^{n-1}$ and $0<r<\mu_1:=\min\{\frac{n-2}{2C_{\log}},\frac{|x|}{2}\}$. Since $\mu^2/r\leqslant r$ for $r\geqslant\mu$, this monotonicity and the Kelvin formula give, for $0<\mu<\mu_1$,
		\begin{equation*}
			u_{x,\mu}(y)\leqslant u(y)\quad {\rm for} \quad \mu\leqslant |y-x|<\mu_1.
		\end{equation*}

		For $|y-x|\geqslant\mu_1$, we use a lower bound for $u$. Applying Fatou's lemma to $u=\Phi_n\ast[vf(u)]$ in \eqref{ourlimitsystemdual}, we obtain
		\begin{equation*}
			\liminf_{|y|\to\infty}|y|^{n-2}u(y)>0.
		\end{equation*}
		Restricting the integral for $u$ to a fixed ball away from the origin provides a positive lower bound for $u$ on each bounded subset of $\mathbb R^n\setminus\{0\}$. Combined with the estimate at infinity, this yields a constant $C>0$ such that
		\begin{equation*}
			u(y)\geqslant \frac{C}{|y-x|^{\,n-2}}\quad {\rm for} \quad |y-x|\geqslant \mu_1 .
		\end{equation*}
		Since $u$ is bounded on $B_{\mu_1}(x)$, we may choose $\mu_2\in(0,\mu_1)$ so small that, for $0<\mu<\mu_2$,
		\begin{equation*}
			u_{x,\mu}(y)
			=\Big(\frac{\mu}{|y-x|}\Big)^{n-2}
			u\!\left(x+\frac{\mu^2(y-x)}{|y-x|^2}\right)
			\leqslant \frac{1}{|y-x|^{n-2}} \left(\mu^{n-2}\sup_{\substack{B_{\mu_{1}}(x)}}u\right)
			\leqslant u(y)
		\end{equation*}
		for $|y-x|\geqslant\mu_1$. Since $u_{x,\mu}\leqslant u$ also holds on $\mu\leqslant|y-x|<\mu_1$, Claim 1 follows.

		Define the maximal radius
		\begin{equation}\label{barmu}
			\bar{\mu}(x):=\sup\left\{L\in(0,|x|]:
			\begin{array}{l}
				u_{x,\mu}(y)\leqslant u(y)\ \text{for every }0<\mu<L,\\
				y\ne0,\quad |y-x|\geqslant\mu
			\end{array}\right\}.
		\end{equation}
		For convenience, write $\bar\mu=\bar\mu(x)$. By Claim 1, $0<\bar\mu\leqslant|x|$.
		
		\noindent{\bf Claim 2:} $\bar{\mu}=|x|$ for any $x \in \mathbb{R}^n\setminus \{0\}$.
		
		Suppose, for a contradiction, that $\bar\mu<|x|$. Fix
		\begin{equation*}
			0<\eta<\min\{1/4,(|x|-\bar\mu)/4\}.
		\end{equation*}
		Then $\overline B_{\bar\mu+3\eta}(x)$ does not contain the origin. We only consider radii $\mu\in[\bar\mu,\bar\mu+\eta]$.
		We will find $0<\varepsilon<\eta$ such that, for every $\mu\in[\bar\mu,\bar\mu+\varepsilon)$,
		\begin{equation*}
			u_{x,\mu}(y) \leqslant u(y) \quad \text{for all} \quad |y-x|\geqslant \mu \quad \text{and} \quad y\in \mathbb{R}^{n} \setminus \{0\}.
		\end{equation*} 
		
		For fixed $y\ne0$ with $|y-x|\geqslant\bar\mu$, passing to the limit as $\mu\uparrow\bar\mu$ in the comparison defining \eqref{barmu} yields, by continuity,
		\begin{equation}\label{barmulemma}
			u_{x,\bar{\mu}}(y) \leqslant u(y) \quad \text{for all} \quad |y-x|\geqslant \bar\mu \quad \text{and} \quad y\in \mathbb{R}^{n} \setminus \{0\}.
		\end{equation} 
		Equality holds on $\partial B_{\bar\mu}(x)$. 
		
		Since $\bar\mu<|x|$, the closed ball $\overline B_{\bar\mu}(x)$ avoids the origin, so $u$ is bounded there. For $|y-x|\geqslant\bar\mu$, the Kelvin formula gives
		\begin{equation*}
			u_{x,\bar\mu}(y)
			=\left(\frac{\bar\mu}{|y-x|}\right)^{n-2}
			u\left(x+\frac{\bar\mu^2(y-x)}{|y-x|^2}\right)
			\leqslant\max_{\overline B_{\bar\mu}(x)}u=:M<\infty.
		\end{equation*}
		By \eqref{eq:blowup}, there exists $y_0\ne0$ sufficiently close to the origin such that $|y_0-x|>\bar\mu$ and $u(y_0)>M$. Hence $u(y_0)>u_{x,\bar\mu}(y_0)$.
		
		Also, \eqref{barmulemma} implies $F(u(z))\geqslant F(u_{x,\bar\mu}(z))$ for $|z-x|\geqslant\bar\mu$ and $z\ne0$. Thus, by \eqref{v-vxmu}, for $|y-x|>\bar\mu$ and $y\ne0$,
		\begin{equation*}
			v(y)-v_{x,\bar{\mu}}(y)=\int_{|z-x|\geqslant\bar{\mu}}K_\alpha(x,\bar{\mu};y,z)\left[F(u(z))-F(u_{x,\bar{\mu}}(z))\right]\,\ud z \geqslant 0.
		\end{equation*} 
		Using $v\geqslant v_{x,\bar\mu}$ in \eqref{u-uxmu}, we obtain 
		\begin{align}
			u(y)-u_{x,\bar{\mu}}(y)&=c_n\int_{|z-x|\geqslant\bar{\mu}}K_2(x,\bar{\mu};y,z)\left[f(u(z))v(z)-f(u_{x,\bar{\mu}}(z))v_{x,\bar{\mu}}(z)\right]\,\ud z \nonumber \\
			&\geqslant c_n\int_{|z-x|\geqslant\bar{\mu}}K_2(x,\bar{\mu};y,z)v_{x,\bar{\mu}}(z)\left[f(u(z))-f(u_{x,\bar{\mu}}(z))\right]\,\ud z .\label{u-uxbarmu-lower}
		\end{align} 
		By continuity, $u>u_{x,\bar\mu}$ on a neighborhood of $y_0$ contained in $\{|z-x|>\bar\mu\}\setminus\{0\}$. The integrand on the right-hand side of \eqref{u-uxbarmu-lower} is positive on this neighborhood and nonnegative elsewhere. Hence
		\begin{equation}\label{uxbarmu<u}
			u_{x,\bar{\mu}}(y) < u(y) \quad \text{for all} \quad |y-x|> \bar{\mu} \quad \text{and} \quad y\in \mathbb{R}^{n} \setminus \{0\}.
		\end{equation}
		The strict monotonicity of $F$ and \eqref{v-vxmu} also imply
		\begin{equation*}
			v_{x,\bar{\mu}}(y) < v(y) \quad \text{for all} \quad |y-x|> \bar{\mu} \quad \text{and} \quad y\in \mathbb{R}^{n} \setminus \{0\}.
		\end{equation*}
		Multiplying \eqref{u-uxbarmu-lower} by $|y-x|^{n-2}$ and applying Fatou's lemma as $|y|\to\infty$, we obtain
		\begin{align*}
			&\liminf_{|y| \to \infty} |y-x|^{n-2} \left(u(y)-u_{x,\bar{\mu}}(y)\right)\\
			&\geqslant \liminf_{|y| \to \infty} c_n\int_{|z-x|\geqslant\bar{\mu}}|y-x|^{n-2}K_2(x,\bar{\mu};y,z)v_{x,\bar{\mu}}(z)\left[f(u(z))-f(u_{x,\bar{\mu}}(z))\right]\,\ud z \\
			&\geqslant c_n\int_{|z-x|\geqslant\bar{\mu}}\left[1-\left(\frac{\bar{\mu}}{|z-x|}\right)^{n-2}\right]v_{x,\bar{\mu}}(z)\left[f(u(z))-f(u_{x,\bar{\mu}}(z))\right]\,\ud z >0.
		\end{align*}
		Hence there exist $c_\infty>0$ and $R_2>\bar\mu+1$ such that 
		\begin{equation}\label{u>R2case}
			u(y)-u_{x,\bar{\mu}}(y) \geqslant \frac{c_\infty}{|y-x|^{n-2}} \quad \text{for}\quad |y-x|\geqslant R_{2} \quad \text{and} \quad y\in \mathbb{R}^{n} \setminus \{0\}.
		\end{equation}
		For the remaining region $\bar\mu+1\leqslant|y-x|\leqslant R_2$, we use the following consequence of Lemma~\ref{lm:differencecomparison}:
		\begin{equation*}
			\frac{\delta_{0}}{|y-z|^{n-2}} \leqslant K_2(x,\mu;y,z) \leqslant \frac{1}{|y-z|^{n-2}},
		\end{equation*}
		for $|y-x|\geqslant\bar\mu+1$, $|z-x|\geqslant\bar\mu+2$, $y\ne z$, and $\mu\in[\bar\mu,\bar\mu+\eta]$, where $0<\delta_0<1$ is independent of $\mu,y,z$. Thus, by \eqref{u-uxbarmu-lower} and \eqref{uxbarmu<u}, for $\bar\mu+1\leqslant|y-x|\leqslant R_2$ and $y\ne0$,
		\begin{align*}
			&|y-x|^{n-2}\left(u(y)-u_{x,\bar{\mu}}(y)\right)\\
			&\geqslant c_n\int_{\bar{\mu}+2\leqslant|z-x|\leqslant\bar{\mu}+4}|y-x|^{n-2}K_2(x,\bar{\mu};y,z)v_{x,\bar{\mu}}(z)\left[f(u(z))-f(u_{x,\bar{\mu}}(z))\right]\,\ud z \\
			&\geqslant c_n\int_{\bar{\mu}+2\leqslant|z-x|\leqslant\bar{\mu}+4}\frac{\delta_{0}|y-x|^{n-2}}{|y-z|^{n-2}}v_{x,\bar{\mu}}(z)\left[f(u(z))-f(u_{x,\bar{\mu}}(z))\right]\,\ud z \\
			&\geqslant c_{2}\int_{\bar{\mu}+2\leqslant|z-x|\leqslant\bar{\mu}+4}v_{x,\bar{\mu}}(z)\left[f(u(z))-f(u_{x,\bar{\mu}}(z))\right]\,\ud z >0,
		\end{align*}
		for some $c_2>0$ independent of $y$. Together with \eqref{u>R2case}, this yields
		\begin{equation*}
			u(y)-u_{x,\bar{\mu}}(y) \geqslant \frac{\varepsilon_{1}}{|y-x|^{n-2}} \quad \text{for} \quad |y-x|\geqslant \bar{\mu}+1 \quad \text{and} \quad y\in \mathbb{R}^{n} \setminus \{0\},
		\end{equation*}
		for some $\varepsilon_1\in(0,1)$. For $|y-x|\geqslant\bar\mu+1$ and $\mu\in[\bar\mu,\bar\mu+\eta]$, differentiating the Kelvin formula and using the boundedness of $u$ and its first derivatives on $\overline B_{\bar\mu+\eta}(x)$ gives $|\partial_\mu u_{x,\mu}(y)|\leqslant C|y-x|^{2-n}$. The preceding lower bound therefore allows us to choose $0<\varepsilon_2<\min\{\varepsilon_1,\eta\}$ so small that, for $\bar\mu\leqslant\mu\leqslant\bar\mu+\varepsilon_2$,
		\begin{align}\label{u-uxmugeq}
			u(y)-u_{x,\mu}(y) &\geqslant \frac{\varepsilon_{1}}{|y-x|^{n-2}} +\left(u_{x,\bar{\mu}}-u_{x,\mu}\right)(y) \nonumber\\
			&\geqslant \frac{\varepsilon_{1}}{2|y-x|^{n-2}} \quad {\rm  for} \quad |y-x|\geqslant \bar{\mu}+1  \quad {\rm  and} \quad y\in \mathbb{R}^{n} \setminus \{0\}.
		\end{align}
		To compare $u$ and $u_{x,\mu}$ in $\mu<|y-x|<\bar\mu+1$, we first prove $v\geqslant v_{x,\mu}$ outside the sphere, then use \eqref{u-uxmu}. Choose $0<\varepsilon_3<\varepsilon_2/2$, to be reduced below. For $\mu\in[\bar\mu,\bar\mu+\varepsilon_3)$, identity~\eqref{v-vxmu} yields
		\begin{align}\label{v-vxmu3}
			v(y)-v_{x,\mu}(y)=&\int_{|z-x|\geqslant\mu}K_\alpha(x,\mu;y,z)\left[F(u(z))-F(u_{x,\mu}(z))\right]\,\ud z \nonumber\\
			\geqslant&\int_{\mu\leqslant|z-x|\leqslant\mu+\varepsilon_{3}}K_\alpha(x,\mu;y,z)\left[F(u(z))-F(u_{x,\mu}(z))\right]\,\ud z\nonumber\\
			&+\int_{\mu+\varepsilon_{3}<|z-x|\leqslant\bar{\mu}+1}K_\alpha(x,\mu;y,z)\left[F(u_{x,\bar{\mu}}(z))-F(u_{x,\mu}(z))\right]\,\ud z\nonumber\\
			&+\int_{\bar{\mu}+2\leqslant|z-x|\leqslant\bar{\mu}+3}K_\alpha(x,\mu;y,z)\left[F(u(z))-F(u_{x,\mu}(z))\right]\,\ud z.
		\end{align}
		The omitted regions are $\bar\mu+1<|z-x|<\bar\mu+2$ and $|z-x|>\bar\mu+3$, where the integrals are nonnegative by \eqref{u-uxmugeq} and $K_\alpha\geqslant0$. In the middle integral, we used $F(u(z))\geqslant F(u_{x,\bar\mu}(z))$ from \eqref{barmulemma}.
		
		The Kelvin formula, the $C^1$ bound for $u$ on $\overline B_{\bar\mu+3\eta}(x)$, and the equality $u=u_{x,\mu}$ on $|z-x|=\mu$ yield
		\begin{equation*}
			\left|F(u(z))-F(u_{x,\mu}(z))\right| \leqslant C\left(|z-x|-\mu\right)  \quad {\rm  for} \quad \mu\leqslant|z-x|\leqslant\mu+\varepsilon_{3},
		\end{equation*}
		and
		\begin{equation*}
			\left|F(u_{x,\bar{\mu}}(z))-F(u_{x,\mu}(z))\right| \leqslant C\left(\mu-\bar{\mu}\right)\leqslant C\varepsilon_{3}  \quad {\rm  for} \quad \mu\leqslant|z-x|\leqslant\bar{\mu}+1.
		\end{equation*}
		For the third integral, \eqref{u-uxmugeq} provides a positive lower bound for $u-u_{x,\mu}$ on the fixed annulus $\bar\mu+2\leqslant|z-x|\leqslant\bar\mu+3$. The values of $u_{x,\mu}$ there lie between two fixed positive constants. Since $F$ is strictly increasing, there exists $\delta_1>0$, independent of $\varepsilon_3$, such that, for $\mu\in[\bar\mu,\bar\mu+\varepsilon_3)$,
		\begin{equation*}
			F(u(z))-F(u_{x,\mu}(z))\geqslant \delta_{1} \quad {\rm  for} \quad \bar{\mu}+2\leqslant|z-x|\leqslant\bar{\mu}+3.
		\end{equation*}
		Applying these three estimates for $F$ to \eqref{v-vxmu3}, we obtain
		\begin{align}\label{Kalphaestimate2}
			v(y)-v_{x,\mu}(y)
			&\geqslant-C\int_{\mu\leqslant|z-x|\leqslant\mu+\varepsilon_{3}}K_\alpha(x,\mu;y,z)\left(|z-x|-\mu\right)\,\ud z\nonumber\\
			&\quad-C\varepsilon_{3}\int_{\mu+\varepsilon_{3}<|z-x|\leqslant\bar{\mu}+1}K_\alpha(x,\mu;y,z)\,\ud z\nonumber\\
			&\quad+\delta_{1}\int_{\bar{\mu}+2\leqslant|z-x|\leqslant\bar{\mu}+3}K_\alpha(x,\mu;y,z)\,\ud z\nonumber\\
			&=-C\int_{\mu\leqslant|z|\leqslant\mu+\varepsilon_{3}}K_\alpha(0,\mu;y-x,z)\left(|z|-\mu\right)\,\ud z\nonumber\\
			&\quad-C\varepsilon_{3}\int_{\mu+\varepsilon_{3}<|z|\leqslant\bar{\mu}+1}K_\alpha(0,\mu;y-x,z)\,\ud z\nonumber\\
			&\quad+\delta_{1}\int_{\bar{\mu}+2\leqslant|z|\leqslant\bar{\mu}+3}K_\alpha(0,\mu;y-x,z)\,\ud z.
		\end{align}
		In the last equality, we have renamed $z-x$ as $z$. For $|z|\geqslant\bar\mu+2$, the kernel vanishes on $|y-x|=\mu$. Using \eqref{prelim:kerneldistance} in the definition of $K_\alpha$ and differentiating in the radial direction yields
		\begin{equation*}
			(y-x) \cdot \nabla_{y}  K_\alpha(0,\mu;y-x,z) \Big|_{|y-x|=\mu}=(n-\alpha)|y-x-z|^{\alpha-n-2}(|z|^2-\mu^2)>0,
		\end{equation*}
		and estimate~\eqref{prelim:kernelestimate} yields, for $\mu\leqslant|y-x|<\bar\mu+1$ and $|z|\geqslant\bar\mu+2$,
		\begin{equation}\label{KalphaestimateA}
			\frac{\delta_{2}}{|y-x-z|^{n-\alpha}}\left(|y-x|-\mu\right) \leqslant K_\alpha(0,\mu;y-x,z)\leqslant\frac{\delta_{3}}{|y-x-z|^{n-\alpha}}\left(|y-x|-\mu\right), 
		\end{equation}
		where $0<\delta_2\leqslant\delta_3<\infty$ are independent of $\mu,y,z$ in these ranges. Indeed, the distance $|y-x-z|$ is bounded above and below by positive constant multiples of $|z|$. Thus the parameter $s_{\mathcal K}$ in \eqref{prelim:kernelparameter} is bounded above and below by positive constant multiples of $|y-x|-\mu\in[0,1]$.
		
		For $|y-x|\geqslant\bar\mu+1$ and $|z|\geqslant\bar\mu+2$, the parameter $s_{\mathcal K}=s_{\mathcal K}(0,\mu;y-x,z)$ in \eqref{prelim:kernelparameter} has a positive lower bound independent of $\mu,y,z$. Lemma~\ref{lm:differencecomparison} therefore gives, for some $0<\delta_4\leqslant1$,
		\begin{equation}\label{KalphaestimateB}
			\frac{\delta_{4}}{|y-x-z|^{n-\alpha}} \leqslant K_\alpha(0,\mu;y-x,z)\leqslant\frac{1}{|y-x-z|^{n-\alpha}}, 
		\end{equation}
		for $z\ne y-x$.
		
		We estimate \eqref{Kalphaestimate2} in three regions, with $y\ne0$ and $\mu\in[\bar\mu,\bar\mu+\varepsilon_3)$. First, let $\bar\mu+1\leqslant|y-x|\leqslant2(\bar\mu+3)$. By \eqref{KalphaestimateB} and $K_\alpha(0,\mu;y-x,z)\leqslant|y-x-z|^{\alpha-n}$,
		\begin{align}\label{v-vxmu1.0}
			\nonumber
			v(y)-v_{x,\mu}(y)
			\geqslant
			&-C\varepsilon_{3}\int_{\mu\leqslant|z|\leqslant\bar{\mu}+1}\frac{1}{|y-x-z|^{n-\alpha}}\,\ud z+\delta_{1}\delta_{4}\int_{\bar{\mu}+2\leqslant|z|\leqslant\bar{\mu}+3}\frac{1}{|y-x-z|^{n-\alpha}}\,\ud z\\
			\geqslant&\left(C_{1}\delta_{1}\delta_{4}-C_{2}\varepsilon_{3}\right)\geqslant0,
		\end{align}
		for sufficiently small $\varepsilon_3$, where $C_1,C_2>0$ are independent of $\varepsilon_3$.
		
		Second, let $|y-x|\geqslant2(\bar\mu+3)$. On both integration regions below, $|z|\leqslant\bar\mu+3\leqslant|y-x|/2$. Using \eqref{KalphaestimateB} and the triangle inequality in \eqref{Kalphaestimate2}, we obtain
		\begin{align}\label{v-vxmu2.0}
			\nonumber
			v(y)-v_{x,\mu}(y)
			\geqslant
			&-C\varepsilon_{3}\int_{\mu\leqslant|z|\leqslant\bar{\mu}+1}\frac{1}{|y-x-z|^{n-\alpha}}\,\ud z+\delta_{1}\delta_{4}\int_{\bar{\mu}+2\leqslant|z|\leqslant\bar{\mu}+3}\frac{1}{|y-x-z|^{n-\alpha}}\,\ud z\\\nonumber
			\geqslant&-C\varepsilon_{3}\int_{\mu\leqslant|z|\leqslant\bar{\mu}+1}\frac{1}{(|y-x|-|z|)^{n-\alpha}}\,\ud z+\delta_{1}\delta_{4}\int_{\bar{\mu}+2\leqslant|z|\leqslant\bar{\mu}+3}\frac{1}{(|y-x|+|z|)^{n-\alpha}}\,\ud z\\\nonumber
			\geqslant&-C\varepsilon_{3}\int_{\mu\leqslant|z|\leqslant\bar{\mu}+1}\left(\frac{2}{|y-x|}\right)^{n-\alpha}\,\ud z+\delta_{1}\delta_{4}\int_{\bar{\mu}+2\leqslant|z|\leqslant\bar{\mu}+3}\left(\frac{1}{2|y-x|}\right)^{n-\alpha}\,\ud z\\
			\geqslant&\left(C_{3}\delta_{1}\delta_{4}-C_{4}\varepsilon_{3}\right)\frac{1}{|y-x|^{n-\alpha}}\geqslant0,
		\end{align}
		for sufficiently small $\varepsilon_3$, where $C_3,C_4>0$ are independent of $\varepsilon_3$.
		Finally, let $\mu<|y-x|<\bar\mu+1$. To estimate the two negative terms in \eqref{Kalphaestimate2}, we use the following annular kernel bounds:
		\begin{align}\label{Kalphaestimate3} \nonumber
			\int_{\mu\leqslant|z|\leqslant\mu+\varepsilon_{3}}K_\alpha(0,\mu;y-x,z)\left(|z|-\mu\right)\,\ud z 
			&\leqslant C\left(|y-x|-\mu\right)\varepsilon_{3}^{\frac{\alpha}{n}}+C\varepsilon_{3}\left(|y-x|-\mu\right)\nonumber\\ 
			&\leqslant C\left(|y-x|-\mu\right)\varepsilon_{3}^{\frac{\alpha}{n}}, 
		\end{align}
		and 
		\begin{equation}\label{Kalphaestimate4}
			\int_{\mu+\varepsilon_{3}\leqslant|z|\leqslant\bar{\mu}+1}K_\alpha(0,\mu;y-x,z)\,\ud z \leqslant C\left(\varepsilon_{3}^{\alpha-1}+|\ln\varepsilon_{3}|+1\right)\left(|y-x|-\mu\right).
		\end{equation}
		Estimates~\eqref{Kalphaestimate3}--\eqref{Kalphaestimate4} follow from the kernel estimates in Jin and Xiong~\cite[proof of Proposition~3.2]{arxiv:1901.01678}, with $2\sigma=\alpha$, $\lambda=\mu$, $\bar\lambda=\bar\mu$, $\varepsilon=\varepsilon_3$, and the comparison point replaced by $y-x$. The constants are independent of $\mu$, $y$, and $\varepsilon_3$ in the above ranges.
		
		By \eqref{KalphaestimateA}, for $\mu\leqslant|y-x|<\bar\mu+1$,
		\begingroup\small
		\begin{align*}
			\int_{\bar{\mu}+2\leqslant|z|\leqslant\bar{\mu}+3}K_\alpha(0,\mu;y-x,z)\,\ud z
			&\geqslant \delta_{2}\left(|y-x|-\mu\right)\int_{\bar{\mu}+2\leqslant|z|\leqslant\bar{\mu}+3} \frac{1}{|y-x-z|^{n-\alpha}} \,\ud z
			\geqslant C\left(|y-x|-\mu\right).
		\end{align*}
		\endgroup
		Using \eqref{Kalphaestimate3}, \eqref{Kalphaestimate4}, and this lower bound in \eqref{Kalphaestimate2}, we obtain, for $\mu\leqslant|y-x|<\bar\mu+1$ and sufficiently small $\varepsilon_3>0$,
		\begin{equation*}
			v(y)-v_{x,\mu}(y)\geqslant (|y-x|-\mu)\left(c\delta_{1}-C\varepsilon_{3}^{\frac{\alpha}{n}}-C\varepsilon_{3}\left(\varepsilon_{3}^{\alpha-1}+|\ln\varepsilon_{3}|+1\right)\right) \geqslant 0.
		\end{equation*}
		Since $\alpha>0$, the subtracted terms tend to zero as $\varepsilon_3\to0$, while $c\delta_1>0$ is fixed. Choose $\varepsilon_3$ so that this lower bound and those in \eqref{v-vxmu1.0} and \eqref{v-vxmu2.0} are all nonnegative. These three estimates cover $|y-x|>\mu$; equality holds by the Kelvin formula on $|y-x|=\mu$. Hence
		\begin{equation}\label{v-vxmulast}
			v(y)-v_{x,\mu}(y)\geqslant 0 \quad {\rm  for} \quad |y-x| \geqslant \mu \quad {\rm  and} \quad y \in \mathbb{R}^n \setminus \{0\},
		\end{equation}
		for every $\mu\in[\bar\mu,\bar\mu+\varepsilon_3)$.
		
		We now use $v\geqslant v_{x,\mu}$ from \eqref{v-vxmulast} in \eqref{u-uxmu}. Choose $0<\varepsilon<\varepsilon_3/2$ and split the integral as in \eqref{v-vxmu3}, with $\varepsilon$ in place of $\varepsilon_3$. The omitted integrals are nonnegative by \eqref{u-uxmugeq}; in the middle integral, we use $f(u)\geqslant f(u_{x,\bar\mu})$ from \eqref{barmulemma}. For $\mu\in[\bar\mu,\bar\mu+\varepsilon)$, $\mu<|y-x|<\bar\mu+1$, and $y\ne0$, we obtain
		\begin{align}\label{u-uxmu3}
			u(y)-u_{x,\mu}(y)=&c_n\int_{|z-x|\geqslant\mu}K_2(x,\mu;y,z)\left[f(u(z))v(z)-f(u_{x,\mu}(z))v_{x,\mu}(z)\right]\,\ud z \nonumber\\	
			\geqslant& c_n\int_{|z-x|\geqslant\mu}K_2(x,\mu;y,z)v_{x,\mu}(z)\left[f(u(z))-f(u_{x,\mu}(z))\right]\,\ud z\nonumber\\
			\geqslant& c_n\int_{\mu\leqslant|z-x|\leqslant\mu+\varepsilon}K_2(x,\mu;y,z)v_{x,\mu}(z)\left[f(u(z))-f(u_{x,\mu}(z))\right]\,\ud z\nonumber\\
			&+c_n\int_{\mu+\varepsilon<|z-x|\leqslant\bar{\mu}+1}K_2(x,\mu;y,z)v_{x,\mu}(z)\left[f(u_{x,\bar{\mu}}(z))-f(u_{x,\mu}(z))\right]\,\ud z\nonumber\\
			&+c_n\int_{\bar{\mu}+2\leqslant|z-x|\leqslant\bar{\mu}+3}K_2(x,\mu;y,z)v_{x,\mu}(z)\left[f(u(z))-f(u_{x,\mu}(z))\right]\,\ud z.
		\end{align}
		The two absolute-value estimates for $F$ used in \eqref{Kalphaestimate2} also hold for $f$. Indeed, the function values in these two differences have uniform positive lower and upper bounds, so $f'$ is bounded even when $p-1<1$. Thus
		\begin{equation*}
			\left|f(u(z))-f(u_{x,\mu}(z))\right| \leqslant C_f\left(|z-x|-\mu\right) \quad {\rm  for} \quad \mu\leqslant|z-x|\leqslant\mu+\varepsilon,
		\end{equation*}
		and
		\begin{equation*}
			\left|f(u_{x,\bar{\mu}}(z))-f(u_{x,\mu}(z))\right| \leqslant C_f\left(\mu-\bar{\mu}\right)\leqslant C_f\varepsilon \quad {\rm  for} \quad \mu\leqslant|z-x|\leqslant\bar{\mu}+1,
		\end{equation*}
		for some $C_f>0$ independent of $\mu,z,\varepsilon$. On the fixed annulus below, \eqref{u-uxmugeq} bounds $u-u_{x,\mu}$ below by a positive constant. By the fixed positive bounds for $u_{x,\mu}$ and strict monotonicity of $f$, there exists $\delta_f>0$, independent of $\varepsilon$, such that, for $\mu\in[\bar\mu,\bar\mu+\varepsilon)$,
		\begin{equation*}
			f(u(z))-f(u_{x,\mu}(z))\geqslant \delta_f\quad {\rm  for} \quad \bar{\mu}+2\leqslant|z-x|\leqslant\bar{\mu}+3.
		\end{equation*}
		Identity~\eqref{bound of the Hartree integral} and the integral formula for $v$ imply that $v$ is continuous and strictly positive on $\overline B_{\bar\mu+\eta}(x)$. Its Kelvin transform therefore satisfies $0<c\leqslant v_{x,\mu}(z)\leqslant C$ for $\mu\leqslant|z-x|\leqslant\bar\mu+3$, with $c,C$ independent of $\mu,z,\varepsilon$.
		Applying these estimates for $f$ and $v_{x,\mu}$ to \eqref{u-uxmu3} and replacing $z-x$ by the new variable $z$, we obtain
		\begingroup\small
		\begin{align*}
			u(y)-u_{x,\mu}(y)
			\geqslant&-C\int_{\mu\leqslant|z|\leqslant\mu+\varepsilon}K_2(0,\mu;y-x,z)\left(|z|-\mu\right)\,\ud z-C\varepsilon\int_{\mu+\varepsilon<|z|\leqslant\bar{\mu}+1}K_2(0,\mu;y-x,z)\,\ud z\\
			&+c\delta_f\int_{\bar{\mu}+2\leqslant|z|\leqslant\bar{\mu}+3}K_2(0,\mu;y-x,z)\,\ud z.
		\end{align*}
		\endgroup
		Estimates~\eqref{Kalphaestimate3}, \eqref{Kalphaestimate4}, and \eqref{KalphaestimateA}, with kernel order $2$ and $\varepsilon_3=\varepsilon$, give, for $\mu<|y-x|<\bar\mu+1$,
		\begin{equation*}
			u(y)-u_{x,\mu}(y)\geqslant (|y-x|-\mu)\left(c\delta_f-C\varepsilon^{\frac{2}{n}}-C\varepsilon\left(\varepsilon+|\ln\varepsilon|+1\right)\right) \geqslant 0, 
		\end{equation*}
		for sufficiently small $\varepsilon$, since the subtracted terms tend to zero and $c\delta_f>0$ is fixed. Equality on the sphere and \eqref{u-uxmugeq} for $|y-x|\geqslant\bar\mu+1$ complete the comparison:
		\begin{equation*}
			u_{x,\mu}(y) \leqslant u(y) \quad {\rm  for} \quad |y-x|\geqslant \mu \quad {\rm  and} \quad y\in \mathbb{R}^{n} \setminus \{0\}.
		\end{equation*} 
		Together with \eqref{barmu}, this proves the comparison for every $0<\mu<\bar\mu+\varepsilon$. Hence $L=\bar\mu+\varepsilon/2<|x|$ belongs to the set in \eqref{barmu}, but $L>\bar\mu$, contradicting the definition of $\bar\mu$. Thus $\bar\mu=|x|$.
		
		\noindent{\bf Claim 3:} $u$ is radially symmetric about the origin and nonincreasing in the radial variable.
		\par\nopagebreak[4]
		
		By Claim 2 and \eqref{barmu}, for every $x\ne0$,
		\begin{equation*}
			u_{x,\mu}(y) \leqslant u(y) \quad {\rm  for} \quad |y-x|\geqslant \mu \quad {\rm  and} \quad y\in \mathbb{R}^{n} \setminus \{0\}\quad {\rm  with} \quad 0<\mu<|x|.
		\end{equation*} 
		We use large spheres to obtain a comparison across a plane. Fix a unit vector $e$, a number $a>0$, and a point $y\ne0$ with $(y-ae)\cdot e<0$. For $R>a$, set $x=Re$ and $\mu=R-a$. Then $0<\mu<|x|$, and $y$ lies outside the sphere because
		\begin{equation*}
			|y-x|^{2} = |y-Re|^{2} = |y-ae|^2+(R-a)^2-2(R-a)(y-ae)\cdot e > \mu^{2}. 
		\end{equation*}
		By the comparison in Claim 2,
		\begin{equation*}
			u(y) \geqslant u_{x,\mu}(y) =\left(\frac{\mu}{|y-x|}\right)^{n-2}u\left(x+\frac{\mu^2(y-x)}{|y-x|^2}\right).
		\end{equation*}
		As $R\to\infty$, we have
		\begin{equation*}
			\frac{\mu^2}{|y-x|^2}=\frac{(R-a)^2}{\left|y-Re\right|^2} =1+\mathrm{o}(1) \quad {\rm as} \quad R\to\infty,
		\end{equation*}
		and
		\begin{align*}
			x+\frac{\mu^2(y-x)}{|y-x|^2}
			&=\frac{\left(R-a\right)^2y}{\left|y-Re\right|^2}+\frac{\left(|y|^2-2R(y \cdot e -a)-a^2\right)Re}{\left|y-Re\right|^2} \\
			&=y-2(y\cdot e -a)e+\mathrm{o}(1) \quad {\rm as} \quad R\to\infty.
		\end{align*}
		Passing to the limit and using the continuity of $u$, we obtain
		\begin{equation*}
			u(y) \geqslant u(y-2(y\cdot e -a)e).
		\end{equation*}
		For $y\cdot e<0$, let $a\downarrow0$ to obtain $u(y)\geqslant u(y-2(y\cdot e)e)$. Apply this inequality at $y-2(y\cdot e)e$ with direction $-e$ to get $u(y-2(y\cdot e)e)\geqslant u(y)$. Thus $u$ is symmetric across every hyperplane through the origin and is radial.
		
		Finally, for $0<r_1<r_2$ and a unit vector $e$, take $y=r_1e$ and $a=(r_1+r_2)/2$ in $u(y)\geqslant u(y-2(y\cdot e-a)e)$. The reflected point is $r_2e$, so $u(r_1e)\geqslant u(r_2e)$. This proves radial monotonicity and completes the proof.
	\end{proof}
	
	\section{Asymptotic radial symmetry}\label{sec:asymptotics}
	In this section, we prove Theorem~\ref{thm2} for $0<\alpha<n$. We first derive a local integral representation and then establish a pointwise upper bound near the origin by a blow-up argument using Lemma~\ref{lem:constantdefect} and the classification in Theorem~\ref{thmB}. Integral representations and blow-up arguments for related conformally invariant equations can be found in~\cite{MR3694645,arxiv:1901.01678}. Here the Hartree integral remains over the original ball, while the local Newton representation is taken over a smaller ball. To prove asymptotic radial symmetry, we compare the solution and the full Hartree potential with their respective Kelvin transforms. We then use the differential equation and Hopf's boundary lemma to continue these comparisons as the sphere expands.
	
	\subsection{Integral representation}
	Throughout this subsection, let $n\geqslant3$ and $0<\alpha<n$. We use $p,F,f$ and $v_R[u]$ from Section~\ref{section1}, with $R=2$. For a positive function $u\in C^2(B_2\setminus\{0\})$, define
	\begin{equation*}
		g=v_2[u]f(u)=v_2[u]u^{p-1}.
	\end{equation*}
	We consider positive classical solutions of
	\begin{equation}\tag{$\mathcal P_{n,\alpha,2}$}\label{ourlocalPDER=2}
		-\Delta u=g\quad\hbox{in }B_2\setminus\{0\}.
	\end{equation}
	For such a solution, $g=-\Delta u\geqslant0$ is continuous away from the origin. The Green representations below are taken on interior balls $B_b$, with $b<2$, while the Hartree integral defining $v_2[u]$ remains over $B_2$. The results of this subsection also do not require \eqref{eq:blowup}.
	
	\begin{proposition}\label{togetlocaldis}
		Suppose that $u\in C^2(B_2\setminus\{0\})$ is a positive classical solution of \eqref{ourlocalPDER=2}. Then $u\in L^p(B_2)$, $g\in L^1_{\mathrm{loc}}(B_2)$, and $-\Delta u=g$ in the sense of distributions on $B_2$: for every $\varphi\in C_c^\infty(B_2)$,
		\begin{equation}\tag{$\mathcal D_{n,\alpha,2}$}\label{distributionalsenselocalR=2}
			-\int_{B_2}u\Delta\varphi\,\ud y=\int_{B_2}g\varphi\,\ud y.
		\end{equation}
	\end{proposition}
	\begin{proof}
		Fix $y_0\in B_2\setminus\{0\}$. Since $|y_0-z|\leqslant4$ for $z\in B_2$ and $\alpha<n$,
		\begin{equation*}
			\frac{4^{\alpha-n}}p\int_{B_2}u(z)^p\,\ud z
			\leqslant v_2[u](y_0)
			=\frac{-\Delta u(y_0)}{u(y_0)^{p-1}}<\infty.
		\end{equation*}
		The finiteness follows from \eqref{bound of the Hartree integral}. Hence $u\in L^p(B_2)\subset L^1(B_2)$. Moreover, H\"older's inequality yields
		\begin{equation}\label{local:cutofferror}
			\varepsilon^{-2}\int_{B_{2\varepsilon}\setminus B_\varepsilon}u\,\ud y
			\leqslant C\|u\|_{L^p(B_2)}
			\varepsilon^{n(1-1/p)-2}
			=C\|u\|_{L^p(B_2)}
			\varepsilon^{\frac{\alpha(n-2)}{n+\alpha}}
			\longrightarrow0.
		\end{equation}
		Let $\eta_\varepsilon$ be as in the proof of Lemma~\ref{globaldistributional}. For $0<b<2$, choose $\zeta\in C_c^\infty(B_2)$ with $0\leqslant\zeta\leqslant1$ and $\zeta=1$ near $\overline B_b$. The same cutoff argument and \eqref{local:cutofferror} imply
		\begin{equation*}
			\int_{B_b}g\,\ud y
			\leqslant\liminf_{\varepsilon\downarrow0}\int_{B_2}g\zeta\eta_\varepsilon\,\ud y
			\leqslant\int_{B_2}u|\Delta\zeta|\,\ud y<\infty.
		\end{equation*}
		Thus $g\in L^1_{\mathrm{loc}}(B_2)$. For $\varphi\in C_c^\infty(B_2)$, \eqref{local:cutofferror} also yields
		\begin{equation*}
			\left|\int_{B_2}u\bigl(2\nabla\varphi\cdot\nabla\eta_\varepsilon
			+\varphi\Delta\eta_\varepsilon\bigr)\,\ud y\right|
			\leqslant C_\varphi\varepsilon^{-2}
			\int_{B_{2\varepsilon}\setminus B_\varepsilon}u\,\ud y\longrightarrow0.
		\end{equation*}
		Testing with $\varphi\eta_\varepsilon$ and passing to the limit by dominated convergence proves \eqref{distributionalsenselocalR=2}.
	\end{proof}

	For $0<b<2$, let $G_b$ and $H_b$ be the kernels in \eqref{Gb(x,z) of mu} and \eqref{global:Poissonkernel}, respectively. We write $G_b(y,z)=\Phi_n(y-z)+A_b(y,z)$ for $y,z\in B_b$, $y\ne z$, where
	\begin{equation*}
		A_b(y,z):=-c_n\left(b^2-2y\cdot z+\frac{|y|^2|z|^2}{b^2}\right)^{(2-n)/2},
		\qquad y,z\in B_b.
	\end{equation*}
	The regular part $A_b$ is smooth on $B_b\times B_b$ and harmonic in $y$ for each fixed $z\in B_b$. Recall also that $0\leqslant G_b\leqslant\Phi_n$, $H_b\geqslant0$, and $\int_{\partial B_b}H_b(y,z)\,\ud\sigma(z)=1$ for $y\in B_b$.

	\begin{proposition}\label{localgreenrepre}
		Under the hypotheses of Proposition~\ref{togetlocaldis}, for every fixed $0<b<2$,
		\begin{equation}\label{Greenrepres}
			u(y)=\int_{B_b}G_b(y,z)g(z)\,\ud z
			+\int_{\partial B_b}H_b(y,z)u(z)\,\ud \sigma(z)
			\quad\text{in }B_b\setminus\{0\}.
		\end{equation}
	\end{proposition}
	\begin{proof}
		By Proposition~\ref{togetlocaldis}, $u,g\in L^1(B_b)$ and $-\Delta u=g$ in $\mathcal D'(B_b)$. Define $V_b$ and $\mathcal H_b$ by \eqref{global:Greenpotential} and \eqref{global:Poissonextension}, respectively, for the present $u$ and $g$.
		Since $b<2$, $g$ is bounded near every point of $\overline B_b\setminus\{0\}$. The fixed-ball arguments in the proof of Lemma~\ref{uandmL0} therefore apply: using $g\in L^1(B_b)$ and $0\leqslant G_b\leqslant\Phi_n$, they show that
		\begin{equation*}
			V_b\in L^1(B_b)\cap C(B_b\setminus\{0\}),\qquad
			-\Delta V_b=g\quad\text{in }\mathcal D'(B_b),\qquad
			\lim_{\substack{y\to\xi\\y\in B_b}}V_b(y)=0
			\quad(\xi\in\partial B_b).
		\end{equation*}
		The Poisson extension $\mathcal H_b$ is harmonic in $B_b$, continuous on $\overline B_b$, and equal to $u$ on $\partial B_b$. Thus $u-V_b-\mathcal H_b\in L^1(B_b)$ is distributionally harmonic. Weyl's lemma and continuity away from the origin provide a smooth harmonic function $W$ on $B_b$ satisfying
		\begin{equation*}
			W(y)=u(y)-V_b(y)-\mathcal H_b(y),\qquad 0<|y|<b.
		\end{equation*}
		As $y\to\xi\in\partial B_b$, we have $u(y)-\mathcal H_b(y)\to0$ and $V_b(y)\to0$. Hence $W$ extends continuously to $\overline B_b$ with zero boundary values. The maximum principle implies $W\equiv0$, proving \eqref{Greenrepres}.
	\end{proof}
	
	\begin{proposition}\label{localrespresentation}
		Under the hypotheses of Proposition~\ref{togetlocaldis}, there exist $0<\tau<2$ and a positive harmonic function $h\in C^\infty(B_\tau)$ such that
		\begin{equation}\label{local:newtonrepresentation}
			u(y)=c_n\int_{B_\tau}|y-z|^{2-n}v_2[u](z)u(z)^{p-1}\,\ud z+h(y)
			\quad\text{in }B_\tau\setminus\{0\}.
		\end{equation}
	\end{proposition}
	\begin{proof}
		Fix $0<b<2$ and let
		\begin{equation*}
			m_b=\min_{\partial B_b}u>0,
			\qquad C_b=\max_{\overline B_{b/2}\times\overline B_{b/2}}|A_b|<\infty.
		\end{equation*}
		Since $g\in L^1(B_b)$, the integral $\int_{B_\tau}g$ tends to zero as $\tau\to0$. We may therefore choose $0<\tau<b/2$ so small that
		\begin{equation*}
			C_b\int_{B_\tau}g(z)\,\ud z\leqslant m_b/2.
		\end{equation*}
		
		For $y\in B_\tau\setminus\{0\}$, the identity $\Phi_n(y-z)=G_b(y,z)-A_b(y,z)$ and \eqref{Greenrepres} yield
		\begin{equation*}
            0\leqslant\int_{B_\tau}\Phi_n(y-z)g(z)\,\ud z \leqslant\int_{B_b}G_b(y,z)g(z)\,\ud z+C_b\int_{B_\tau}g(z)\,\ud z<\infty.
		\end{equation*}
		Define, for $y\in B_\tau$,
		\begin{equation*}
			h(y)={}\int_{B_\tau}A_b(y,z)g(z)\,\ud z
			+\int_{B_b\setminus B_\tau}G_b(y,z)g(z)\,\ud z
			+\int_{\partial B_b}H_b(y,z)u(z)\,\ud \sigma(z).
		\end{equation*}
		We next show that $h$ is smooth and harmonic in $B_\tau$, including at the origin. Fix a compact set $K\Subset B_\tau$. For $y\in B_\tau$ and $z\in B_b$, we have
		\begin{equation*}
			b^2-2y\cdot z+\frac{|y|^2|z|^2}{b^2}
			\geqslant\left(b-\frac{|y||z|}{b}\right)^2
			\geqslant(b-\tau)^2>0.
		\end{equation*}
		The formula for $A_b$ therefore bounds all its
		$y$-derivatives uniformly. Moreover,
		$|y-z|\geqslant\tau-\max_{w\in K}|w|>0$ for
		$y\in K$ and $z\in B_b\setminus B_\tau$.
		Using $G_b=\Phi_n+A_b$ and these bounds, we find, for every
		multi-index $\kappa$,
		\begin{equation*}
			\sup_{\substack{y\in K\\z\in B_\tau}}
			|\partial_y^\kappa A_b(y,z)|
			+
			\sup_{\substack{y\in K\\z\in B_b\setminus B_\tau}}
			|\partial_y^\kappa G_b(y,z)|
			\leqslant C_{K,\kappa}<\infty.
		\end{equation*}
		Since $g\in L^1(B_b)$, the bound $C_{K,\kappa}|g(z)|$ allows
		us to move derivatives of any order with respect to $y$
		inside the first two integrals. These integrals are harmonic, since $\Delta_y A_b=0$ and $\Delta_y G_b(y,z)=0$ for $y\in B_\tau$ and
		$z\in B_b\setminus B_\tau$.
		The Poisson integral is also smooth and harmonic.
		Hence $h\in C^\infty(B_\tau)$ and $\Delta h=0$ in $B_\tau$,
		including at the origin.
		
		Splitting \eqref{Greenrepres} over $B_\tau$ and $B_b\setminus B_\tau$ establishes \eqref{local:newtonrepresentation}. It remains to check that $h>0$. The second integral defining $h$ is nonnegative because $G_b,g\geqslant0$. Since $u\geqslant m_b$ on $\partial B_b$ and $H_b\geqslant0$, the Poisson term satisfies
		\begin{equation*}
			\int_{\partial B_b}H_b(y,z)u(z)\,\ud \sigma(z)
			\geqslant m_b\int_{\partial B_b}H_b(y,z)\,\ud \sigma(z)=m_b.
		\end{equation*}
		Only the first integral can be negative. The bound $|A_b|\leqslant C_b$ and our choice of $\tau$ ensure that its absolute value is at most $m_b/2$. Thus, for every $y\in B_\tau$,
		\begin{equation*}
			h(y)\geqslant-C_b\int_{B_\tau}g(z)\,\ud z+m_b
			\geqslant m_b/2>0.
		\end{equation*}
	\end{proof}
	
	\subsection{Upper bound estimate}\label{subsec:upperbound}
	For convenience, we set
    \begin{equation*}
    \beta:=\frac{n-2}{2}, \qquad q:=p-1=\frac{\alpha+2}{n-2}.
    \end{equation*}
	
	\begin{proposition}\label{prop:upperbound}
		Let $n\geqslant 3 $ and $0<\alpha<n$. Suppose that
		$u\in C^2(B_2\setminus\{0\})$ is a positive classical solution of
		\begin{equation}\label{p45:originalPDE}
			-\Delta u=v_2[u]u^q\quad\text{in }B_2\setminus\{0\},
			\qquad \lim_{X\to0}u(X)=+\infty.
		\end{equation}
		Then
		\begin{equation}\label{p45:target}
			\limsup_{X\to0}|X|^{(n-2)/2}u(X)<\infty.
		\end{equation}
	\end{proposition}
	
	\begin{proof}
		By Proposition~\ref{togetlocaldis}, $u\in L^p(B_2)$ and $v_2[u]u^q\in L^1_{\mathrm{loc}}(B_2)$. Fix $0<\tau<2$ and a positive harmonic function $h\in C^\infty(B_\tau)$ as in Proposition~\ref{localrespresentation}, so that the representation \eqref{local:newtonrepresentation} holds.
		
		We argue by contradiction. In Claim A, we rescale about suitable points and use Theorem~\ref{thmB} to identify the limit as a bubble. The Kelvin comparison in Claim B, together with the critical-radius and calculus lemmas, will force this limit to be constant.
		
		\pffproofstep{Selection of blow-up points and rescaling.}
		Suppose that \eqref{p45:target} fails. There exists a sequence of nonzero points $x_k\to0$ such that
		\begin{equation*}
			|x_k|^\beta u(x_k)\longrightarrow\infty\qquad\text{as }k\to\infty.
		\end{equation*}
		For large $k$, the ball
		$\overline B_{|x_k|/2}(x_k)$ lies in $B_\tau\setminus\{0\}$.
		Let $\bar x_k$ be a maximum point of
		\begin{equation*}
			\widehat u_k(X):=\bigl(|x_k|/2-|X-x_k|\bigr)^\beta u(X)
			\quad\text{on }\overline B_{|x_k|/2}(x_k).
		\end{equation*}
		Its positive maximum is attained in the interior. Define
		\begin{equation*}
			2t_k:=|x_k|/2-|\bar x_k-x_k|,\qquad M_k:=u(\bar x_k).
		\end{equation*}
		If $|X-\bar x_k|\leqslant t_k$, then
		the triangle inequality gives
		\begin{equation*}
			|x_k|/2-|X-x_k|
			\geqslant |x_k|/2-|\bar x_k-x_k|-|X-\bar x_k|
			\geqslant2t_k-t_k=t_k.
		\end{equation*}
		By maximality of $\widehat u_k$,
		\begin{equation*}
			t_k^\beta u(X)\leqslant\widehat u_k(X)
			\leqslant\widehat u_k(\bar x_k)=(2t_k)^\beta M_k,
			\qquad
			\widehat u_k(\bar x_k)\geqslant\widehat u_k(x_k)
			=2^{-\beta}|x_k|^\beta u(x_k).
		\end{equation*}
		Consequently,
		\begin{equation}\label{p45:peakest}
			u(X)\leqslant C_*M_k\quad\text{on }\overline B_{t_k}(\bar x_k),\qquad
			(2t_k)^\beta M_k\geqslant2^{-\beta}|x_k|^\beta u(x_k)\to\infty,
		\end{equation}
		where $C_*:=2^{\beta}$. Moreover $t_k\leqslant |x_k|/4$ and
		$|x_k|/2<|\bar x_k|<3|x_k|/2$. Thus $\bar x_k\to0$ and $M_k\to\infty$ as $k\to\infty$.
		Define
		\begin{equation}\label{p45:scaling}
			\lambda_k:=M_k^{-1/\beta},\quad
			w_k(x):=\lambda_k^\beta u(\bar x_k+\lambda_kx),\quad
			R_k:=t_k/\lambda_k,\quad a_k:=-\bar x_k/\lambda_k.
		\end{equation}
		Since $\lambda_k^\beta M_k=1$, we have, as $k\to\infty$,
		\begin{equation*}
			R_k^\beta=t_k^\beta M_k\longrightarrow\infty,
			\qquad
			\frac{|a_k|}{R_k}=\frac{|\bar x_k|}{t_k}\geqslant2.
		\end{equation*}
		Together with \eqref{p45:peakest}, these relations imply
		\begin{equation}\label{p45:normalization}
			R_k\to\infty,\qquad |a_k|\geqslant2R_k\to\infty,\qquad
			w_k(0)=1,\qquad0<w_k\leqslant C_*\ \text{on }B_{R_k}.
		\end{equation}
		The original singular point $0$ is mapped to $a_k$. The Newton representation and the Hartree integral are taken over $B_\tau$ and $B_2$, respectively. We also fix $0<\tau'<\tau$ and introduce the smaller ball $B_{\tau'}$ for the Kelvin comparison in Claim B. The corresponding rescaled domains are
		\begin{equation}\label{p45:domains}
			\pffOo:=\{x:\bar x_k+\lambda_kx\in B_\tau\},\quad
			\pffOi:=\{x:\bar x_k+\lambda_kx\in B_2\},\quad
			\pffOt:=\{x:\bar x_k+\lambda_kx\in B_{\tau'}\},
		\end{equation}
		The superscripts ``in'' and ``out'' refer to the inner Hartree integral and the outer Newton integral, respectively. These domains satisfy
		$\overline\pffOt\subset\pffOo\subset\pffOi$.
		Since $\bar x_k\to0$, $\lambda_k\to0$ and $|a_k|\to\infty$,
		every fixed closed ball $\overline B_L$ is contained in
		$\pffOt\setminus\{a_k\}$ for all sufficiently large $k$.
		
		Put
		\begin{equation}\label{p45:scaledobjects}
			P_k(y):=\int_{\pffOi}|y-z|^{\alpha-n}F(w_k(z))\,\ud z,
			\qquad h_k(x):=\lambda_k^\beta h(\bar x_k+\lambda_kx).
		\end{equation}
		Thus $P_k=v_{\pffOi}[w_k]$ in the notation of Section~\ref{sec:preliminaries}. The change of variables $Z=\bar x_k+\lambda_kz$ and the identity $\beta p=(n+\alpha)/2$ give
		\begin{equation}\label{p45:mass}
			\int_{\pffOi}F(w_k)=\lambda_k^{\beta p-n}\int_{B_2}F(u)
			=\lambda_k^{(\alpha-n)/2}\int_{B_2}F(u)<\infty.
		\end{equation}
		Performing the same change of variables in the integral defining $v_2[u]$, and writing $Y=\bar x_k+\lambda_ky$, we obtain
		\begin{equation*}
			P_k(y)=\lambda_k^{\beta p-\alpha}v_2[u](Y)
			=\lambda_k^{(n-\alpha)/2}v_2[u](Y)<\infty
			\quad\text{for }y\in\pffOi\setminus\{a_k\}.
		\end{equation*}
		Here $v_2[u](Y)<\infty$ follows from \eqref{bound of the Hartree integral}, since $Y\in B_2\setminus\{0\}$. Since $p+q-1=(\alpha+2)/\beta$,
		\begin{equation*}
			\lambda_k^{\alpha+2-\beta(p+q-1)}=1.
		\end{equation*}
		Thus the powers of $\lambda_k$ cancel in both the integral representation and the differential equation.
		Substituting $X=\bar x_k+\lambda_kx$ and $Y=\bar x_k+\lambda_ky$ into \eqref{local:newtonrepresentation} and rescaling \eqref{p45:originalPDE}, respectively, we obtain
		\begin{equation}\label{p45:scaledintegral}
			w_k(x)=c_n\int_{\pffOo}|x-y|^{2-n}P_k(y)w_k(y)^q\,\ud y+h_k(x),
			\qquad x\in\pffOo\setminus\{a_k\},
		\end{equation}
		and
		\begin{equation}\label{p45:scaledPDE}
			-\Delta w_k=P_kw_k^q
			\quad\text{in }\pffOi\setminus\{a_k\}.
		\end{equation}
		Fix a compact set $K\subset\pffR^n$ and $0<r<\tau$. As $K$ is bounded and both $\bar x_k$ and $\lambda_k$ tend to $0$, we have $\bar x_k+\lambda_kK\subset B_r$ for all $k$ large enough. The chain rule and the smoothness of $h$ on $B_\tau$ yield, for every integer $j\geqslant0$,
		\begin{equation*}
			\|D^jh_k\|_{L^\infty(K)}\leqslant C_j\lambda_k^{\beta+j}\to0\qquad\text{as }k\to\infty,
		\end{equation*}
		where $C_j$ bounds $D^jh$ on $\overline B_r$ and is independent of $k$.
		
		\begin{pffclaima}
			For each fixed $0<\gamma<\min\{1,\alpha\}$, there exists a subsequence such that, as $k\to\infty$,
			\begin{equation*}
				w_k\longrightarrow w_0\quad\text{in }C^{2,\gamma'}_{\loc}(\pffR^n),
				\qquad P_k\longrightarrow P_0\quad\text{in }C^{0,\gamma'}_{\loc}(\pffR^n)
				\quad\text{for }0<\gamma'<\gamma.
			\end{equation*}
			The limits satisfy $0<w_0\leqslant C_*$, $w_0(0)=1$, and
			\begin{equation*}
				P_0(x)=\int_{\pffR^n}|x-z|^{\alpha-n}F(w_0(z))\,\ud z<\infty,
				\qquad -\Delta w_0=P_0w_0^q\quad\text{in }\pffR^n.
			\end{equation*}
			Moreover, there exist constants $A,a>0$ and a point $z_0\in\pffR^n$ such that
			$w_0(x)=A(a^2+|x-z_0|^2)^{-\beta}$.
		\end{pffclaima}
		
		\begin{proof}[Proof of Claim A]
			\pffclaimpart{A1}{Local uniform boundedness of the potentials}
			We first bound $P_k$ on each fixed ball, uniformly for large $k$. The argument reduces this estimate to a bound for $P_k(0)$.
			\par\smallskip\phantomsection\label{p45:step:A1.1}
			Fix $L>0$ and take $k$ large enough that $B_{2L}\subset B_{R_k}\subset\pffOi$.
			By \eqref{p45:normalization}, $F(w_k)\leqslant C_*^p/p$ on $B_{2L}$.
			For $y\in B_L$, we have $B_{2L}-y\subset B_{3L}$ and
			$|y-z|\geqslant |z|/2$ whenever $|z|\geqslant2L$.
			Splitting the integral over $B_{2L}$ and its complement, we obtain
			\begin{align}\label{p45:Preduction}
				\sup_{B_L}P_k
				&\leqslant\frac{C_*^p}{p}\int_{B_{3L}}|\zeta|^{\alpha-n}\,\ud\zeta
				+2^{n-\alpha}\int_{\pffOi\setminus B_{2L}}|z|^{\alpha-n}F(w_k(z))\,\ud z \notag\\
				&\leqslant\frac{C_*^p|\pffSph|}{p\alpha}(3L)^\alpha
				+2^{n-\alpha}P_k(0).
			\end{align}
			\par\smallskip\phantomsection\label{p45:step:A1.2}
			It remains to bound $P_k(0)$. Suppose, to the contrary, that along a subsequence $A_k:=P_k(0)\to\infty$ as $k\to\infty$. Set
			\begin{equation*}
				\sigma_k:=A_k^{-1/2},\qquad
				Q_k(\xi):=A_k^{-1}P_k(\sigma_k\xi),\qquad
				\psi_k(\xi):=w_k(\sigma_k\xi).
			\end{equation*}
			For fixed $S,R>0$, large $k$ satisfy $\sigma_kS<R/2$ and $B_{2R}\subset B_{R_k}$.
			For $|\xi|\leqslant S$, splitting the integral at $B_R$ gives
			\begin{equation*}
				\begin{aligned}
					|P_k(\sigma_k\xi)-P_k(0)|
					&\leqslant\int_{B_R}\bigl||z-\sigma_k\xi|^{\alpha-n}-|z|^{\alpha-n}\bigr|F(w_k(z))\,\ud z\\
					&\quad+\int_{\pffOi\setminus B_R}\bigl||z-\sigma_k\xi|^{\alpha-n}-|z|^{\alpha-n}\bigr|F(w_k(z))\,\ud z.
				\end{aligned}
			\end{equation*}
			Since $F(w_k)\leqslant C_*^p/p$ on $B_R$ and
			$B_R-\sigma_k\xi\subset B_{3R/2}$, local integrability of the
			kernel bounds the integral over $B_R$ by $C_R$, independently
			of $k$ and $|\xi|\leqslant S$.
			
			For $|z|\geqslant R$, the mean value theorem and
			$|z-t\sigma_k\xi|\geqslant|z|/2$ for $0\leqslant t\leqslant1$
			imply
			\begin{equation*}
				\bigl||z-\sigma_k\xi|^{\alpha-n}-|z|^{\alpha-n}\bigr|
				\leqslant C\sigma_kS|z|^{\alpha-n-1}
				\leqslant C\frac{\sigma_kS}{R}|z|^{\alpha-n},
			\end{equation*}
			where $C$ depends only on $n,\alpha$.
			Since $P_k(0)=A_k$, the absolute value of the second part is therefore at most
			$C\sigma_kSA_k/R$. Dividing the two bounds by $A_k$ and using
			the definition of $Q_k$, we obtain
			\begin{equation}\label{p45:Qconvergence}
				\sup_{B_S}|Q_k-1|
				\leqslant C_R/A_k+C\sigma_kS/R\to0
				\quad\text{as }k\to\infty.
			\end{equation}
			\par\smallskip\phantomsection\label{p45:step:A1.3}
			Since $\sigma_k^2A_k=1$, $|a_k|/\sigma_k\to\infty$, and $R_k/\sigma_k\to\infty$, the rescaled equation and normalization give
			\begin{equation*}
				-\Delta\psi_k=Q_k\psi_k^q,\qquad \psi_k(0)=1,\qquad0<\psi_k\leqslant C_*
			\end{equation*}
			on each fixed ball for large $k$. Fix an exponent $s>n$. To obtain local $C^1$ compactness of $\psi_k$, we apply the interior estimate for the Laplacian:
        \begin{equation}\label{p45:elliptic}
            \|\psi_k\|_{W^{2,s}(B_L)}
            \leqslant C(n,s,L)
            \bigl(\|\psi_k\|_{L^s(B_{2L})}
             +\|\Delta\psi_k\|_{L^s(B_{2L})}\bigr).
             \end{equation} 
             For large $k$, $0\leqslant Q_k\leqslant2$ on $B_{2L}$ by \eqref{p45:Qconvergence}. Thus $|\Delta\psi_k|\leqslant2C_*^q$ there. Together with $0<\psi_k\leqslant C_*$, estimate \eqref{p45:elliptic} gives a bound for $\psi_k$ in $W^{2,s}(B_L)$ that is independent of $k$. Applying this estimate on larger balls and using the compact embedding into $C^1$, we extract a diagonal subsequence converging to $\psi$ in $C^1_{\loc}(\pffR^n)$. In particular, $0\leqslant\psi\leqslant C_*$ and $\psi(0)=1$.
			For $t_1,t_2\in[0,C_*]$, we have
			\begin{equation}\label{p45:powercontinuity}
				|t_1^q-t_2^q|\leqslant
				\begin{cases}
					|t_1-t_2|^q,&0<q\leqslant1,\\
					qC_*^{q-1}|t_1-t_2|,&q>1.
				\end{cases}
			\end{equation}
			This estimate and \eqref{p45:Qconvergence} imply $Q_k\psi_k^q\to\psi^q$ locally uniformly as $k\to\infty$. For $\varphi\in C_c^\infty(\pffR^n)$, choose a fixed ball $B_L$ containing its support. Since $R_k/\sigma_k\to\infty$, the equation for $\psi_k$ holds throughout $B_L$ for large $k$. Testing against $\varphi$ and passing to the limit, we obtain
			\begin{equation*}
				\int_{\pffR^n}\psi(-\Delta\varphi)
				=\lim_{k\to\infty}\int_{B_L}Q_k\psi_k^q\varphi
				=\int_{\pffR^n}\psi^q\varphi.
			\end{equation*}
			Hence $-\Delta\psi=\psi^q$ in distributions on $\pffR^n$. Since $\psi$ is locally Lipschitz, \eqref{p45:powercontinuity} shows that $\psi^q$ is locally H\"older continuous, also at points where $\psi=0$. By interior Schauder regularity, $\psi\in C^{2,\delta}_{\loc}$ for every $0<\delta<\min\{q,1\}$. The strong minimum principle, applied to $\Delta\psi=-\psi^q\leqslant0$, and the normalization $\psi(0)=1$ imply $\psi>0$ throughout $\pffR^n$.
			
			\par\smallskip\phantomsection\label{p45:step:A1.4}
			However, such a function $\psi$ cannot exist. If $q>1$, then $q<(n+2)/(n-2)$ because $\alpha<n$, and the Gidas--Spruck nonexistence theorem~\cite[Theorem~1.1]{GS1981} applies.
			If $0<q\leqslant1$, the inequality $0<\psi\leqslant C_*$ implies $\psi^{q-1}\geqslant C_*^{q-1}$, so $-\Delta\psi\geqslant C_*^{q-1}\psi$.
			Let $\phi_R>0$ be the first Dirichlet eigenfunction on $B_R$, with $-\Delta\phi_R=\lambda_1(B_R)\phi_R$ and $\phi_R=0$ on $\partial B_R$. The Hopf boundary lemma implies $\partial_\nu\phi_R<0$ for the outward normal. By Green's identity,
			\begin{equation*}
				C_*^{q-1}\int_{B_R}\psi\phi_R
				\leqslant\int_{B_R}(-\Delta\psi)\phi_R
				=\lambda_1(B_R)\int_{B_R}\psi\phi_R
				+\int_{\partial B_R}\psi\partial_\nu\phi_R
				\leqslant\lambda_1(B_R)\int_{B_R}\psi\phi_R.
			\end{equation*}
			Since $0<\int_{B_R}\psi\phi_R<\infty$, division yields
			\begin{equation*}
				0<C_*^{q-1}\leqslant\lambda_1(B_R)=R^{-2}\lambda_1(B_1)
				\qquad\text{for every }R>0.
			\end{equation*}
			Letting $R\to\infty$ gives a contradiction. Therefore $P_k(0)$ is bounded, and \eqref{p45:Preduction} yields $\sup_{B_L}P_k\leqslant C_L$ for each fixed $L>0$ and all sufficiently large $k$, with $C_L$ independent of $k$.
			
			\pffclaimpart{A2}{Local compactness and the limiting equation}
			We next prove local H\"older bounds for $P_k$ and use the equation to obtain compactness of $w_k$.
			\par\smallskip\phantomsection\label{p45:step:A2.1}
			Fix $L>0$ and take $k$ large enough that $\overline B_{4L}\subset B_{R_k}$. Then $F(w_k)\leqslant C_*^p/p$ on $B_{4L}$. For $x,y\in\overline B_L$ with $0<d:=|x-y|\leqslant L/4$, define $E:=B_{2d}(x)\cup B_{2d}(y)\subset B_{4L}$ and write
			\begin{equation*}
				\begin{aligned}
					P_k(x)-P_k(y)
					&=\left(\int_E+\int_{B_{4L}\setminus E}+\int_{\pffOi\setminus B_{4L}}\right)
					\bigl(|x-z|^{\alpha-n}-|y-z|^{\alpha-n}\bigr)F(w_k(z))\,\ud z\\
					&=:J_1+J_2+J_3.
				\end{aligned}
			\end{equation*}
			Since $E\subset B_{3d}(x)\cap B_{3d}(y)$, we have
			\begin{equation*}
				|J_1|\leqslant\frac{C_*^p}{p}
				\left(\int_{B_{3d}(x)}|x-z|^{\alpha-n}\,\ud z
				+\int_{B_{3d}(y)}|y-z|^{\alpha-n}\,\ud z\right)
				=\frac{2C_*^p|\pffSph|}{p\alpha}(3d)^\alpha.
			\end{equation*}
			For $z\in B_{4L}\setminus E$ and $\xi_t=x+t(y-x)$, $0\leqslant t\leqslant1$, we have
			$|z-\xi_t|\geqslant|z-x|-d\geqslant|z-x|/2$. Since
			$|\nabla_\xi|z-\xi|^{\alpha-n}|=(n-\alpha)|z-\xi|^{\alpha-n-1}$, integration along the segment yields
			\begin{equation*}
				\bigl||x-z|^{\alpha-n}-|y-z|^{\alpha-n}\bigr|
				\leqslant(n-\alpha)d\int_0^1|z-\xi_t|^{\alpha-n-1}\,\ud t
				\leqslant Cd|z-x|^{\alpha-n-1}.
			\end{equation*}
			Here $2d\leqslant|z-x|\leqslant5L$, so polar integration centred at $x$ yields
			\begin{equation*}
				|J_2|\leqslant Cd\int_{2d}^{5L}r^{\alpha-2}\,\ud r
				\leqslant
				\begin{cases}
					C_\alpha d^\alpha,&0<\alpha<1,\\
					Cd\log(5L/(2d)),&\alpha=1,\\
					C_{L,\alpha}d,&1<\alpha<n.
				\end{cases}
			\end{equation*}
			For $|z|\geqslant4L$, the same segment estimate and the bound for $P_k(0)$ give
			\begin{equation*}
				\begin{aligned}
					|J_3|&\leqslant Cd\int_{\pffOi\setminus B_{4L}}|z|^{\alpha-n-1}F(w_k(z))\,\ud z\\
					&\leqslant\frac{Cd}{4L}\int_{\pffOi\setminus B_{4L}}|z|^{\alpha-n}F(w_k(z))\,\ud z
					\leqslant C_LdP_k(0)\leqslant C_Ld.
				\end{aligned}
			\end{equation*}
			For $0<\gamma<\min\{1,\alpha\}$, the bounds for $J_1,J_2,J_3$ imply
			\begin{equation*}
				|P_k(x)-P_k(y)|\leqslant C_{L,\gamma}|x-y|^\gamma
				\quad\text{for }0<|x-y|\leqslant L/4.
			\end{equation*}
			Here we used $d\log(5L/(2d))\leqslant C_{L,\gamma}d^\gamma$ when $\alpha=1$. For $|x-y|>L/4$, the same conclusion follows from the uniform bound for $P_k$ on $B_L$. Consequently,
			\begin{equation}\label{p45:Holder}
				\|P_k\|_{C^{0,\gamma}(\overline B_L)}\leqslant C_{L,\gamma}.
			\end{equation}
			\par\smallskip\phantomsection\label{p45:step:A2.2}
			For $w_k$, the equation \eqref{p45:scaledPDE} and the bounds on $B_{2L}$ give
			\begin{equation*}
				|\Delta w_k|=P_kw_k^q\leqslant C_{2L}C_*^q,
				\qquad
				\|w_k\|_{W^{2,s}(B_L)}
				\leqslant C(n,s,L)(C_*+C_{2L}C_*^q)|B_{2L}|^{1/s}
				\quad\text{for }s>n,
			\end{equation*}
			where we used \eqref{p45:elliptic}. The estimate on larger balls implies compactness of $w_k$ in $C^1$ on every fixed closed ball. Using \eqref{p45:Holder} and the compact embedding of $C^{0,\gamma}(\overline B_L)$ into $C^{0,\gamma'}(\overline B_L)$ for $0<\gamma'<\gamma$, we can then extract a diagonal subsequence for the pair $(w_k,P_k)$ satisfying, as $k\to\infty$,
			\begin{equation}\label{p45:limits}
				w_k\to w_0\ \text{in }C^1_{\loc}(\pffR^n),\qquad
				P_k\to P_0\ \text{in }C^{0,\gamma'}_{\loc}(\pffR^n)
				\quad\text{for }0<\gamma'<\gamma,
			\end{equation}
			where $0\leqslant w_0\leqslant C_*$, $w_0(0)=1$ and $P_0\geqslant0$.
			
			\par\smallskip\phantomsection\label{p45:step:A2.3}
			
			On each compact set $K\subset\pffR^n$,
			\eqref{p45:powercontinuity} and \eqref{p45:limits} imply,
			as $k\to\infty$,
			\begin{equation*}
				\begin{aligned}
					\|P_kw_k^q-P_0w_0^q\|_{L^\infty(K)}
					&\leqslant C_*^q\|P_k-P_0\|_{L^\infty(K)}
					+\|P_0\|_{L^\infty(K)}\|w_k^q-w_0^q\|_{L^\infty(K)}
					\longrightarrow0.
				\end{aligned}
			\end{equation*}
			Fix $\varphi\in C_c^\infty(\pffR^n)$ and choose a fixed ball
			$B_L$ containing its support. For large $k$,
			$B_L\subset\pffOi\setminus\{a_k\}$.
			Testing \eqref{p45:scaledPDE} against $\varphi$, we obtain
			\begin{equation*}
				\int_{\pffR^n}w_0(-\Delta\varphi)
				=\lim_{k\to\infty}\int_{B_L}w_k(-\Delta\varphi)
				=\lim_{k\to\infty}\int_{B_L}P_kw_k^q\varphi
				=\int_{\pffR^n}P_0w_0^q\varphi.
			\end{equation*}
			Thus $-\Delta w_0=P_0w_0^q$ in the sense of distributions
			on $\pffR^n$.
			
			\par\smallskip\phantomsection\label{p45:step:A2.4}
			
			The local uniform convergence in \eqref{p45:limits} and estimate~\eqref{p45:Holder} imply $P_0\in C^{0,\gamma}_{\loc}(\pffR^n)$. Since $w_0$ is locally Lipschitz, \eqref{p45:powercontinuity} shows that $P_0w_0^q\in C^{0,\delta}_{\loc}$ for every $0<\delta<\min\{\gamma,q,1\}$. Consequently, $w_0\in C^{2,\delta}_{\loc}$ by interior Schauder regularity, and the limit equation holds pointwise. Since $\Delta w_0=-P_0w_0^q\leqslant0$ and $w_0(0)=1$, the strong minimum principle yields
			\begin{equation}\label{p45:limitPDE}
				-\Delta w_0=P_0w_0^q\quad\text{and}\quad w_0>0
				\qquad\text{in }\pffR^n.
			\end{equation}
			
			\par\smallskip\phantomsection\label{p45:step:A2.5}
			
			To improve the convergence of $w_k$, fix $L>0$ and let $m_{2L}:=\min_{\overline B_{2L}}w_0>0$. By uniform convergence, for large $k$,
			\begin{equation*}
				\|w_k-w_0\|_{L^\infty(B_{2L})}<m_{2L}/2,
				\qquad m_{2L}/2\leqslant w_k\leqslant C_*
				\quad\text{on }B_{2L}.
			\end{equation*}
			Hence the derivative of $t^q$ is bounded on $[m_{2L}/2,C_*]$, even when $q<1$. The uniform local $C^1$ bound and the mean value theorem imply
			\begin{equation*}
				|w_k(x)^q-w_k(y)^q|\leqslant C_L|w_k(x)-w_k(y)|
				\leqslant C_L|x-y|
				\quad\text{for }x,y\in B_{2L},
			\end{equation*}
			with $C_L$ independent of $k$.
			Here $[\cdot]_{\gamma;E}$ denotes the H\"older seminorm on $E$. The product estimate
			\begin{equation*}
				[P_kw_k^q]_{\gamma;B_{2L}}
				\leqslant\|w_k^q\|_{L^\infty(B_{2L})}[P_k]_{\gamma;B_{2L}}
				+\|P_k\|_{L^\infty(B_{2L})}[w_k^q]_{\gamma;B_{2L}}
			\end{equation*}
			and \eqref{p45:Holder} bound the right-hand side uniformly in $C^{0,\gamma}$. Applying interior Schauder estimates, we find
			\begin{equation*}
				\|w_k\|_{C^{2,\gamma}(\overline B_L)}
				\leqslant C(n,\gamma,L)
				\bigl(\|w_k\|_{C^0(\overline B_{2L})}
				+\|P_kw_k^q\|_{C^{0,\gamma}(\overline B_{2L})}\bigr)
				\leqslant C_{L,\gamma}.
			\end{equation*}
			Passing to a further diagonal subsequence, we obtain $w_k\to w_0$ in $C^{2,\gamma'}_{\loc}$ as $k\to\infty$, for every $0<\gamma'<\gamma$.
			
			\pffclaimpart{A3}{Passage to the full Hartree integral}
			Local convergence does not by itself justify passing to the limit over the expanding domains. Therefore, we first establish
			\begin{equation}\label{p45:bdefect}
				P_0(y)=\int_{\pffR^n}|y-z|^{\alpha-n}F(w_0(z))\,\ud z+b_0
				\qquad\text{for }y\in \mathbb {R}^n,
			\end{equation}
			for some constant $b_0\geqslant0$, and then show that $b_0=0$.
			Fix $L>0$ and $R>2L$. For large $k$, $\overline B_R\subset\pffOo\setminus\{a_k\}$ by \eqref{p45:domains}. We first let $k\to\infty$ with $L,R$ fixed, and then let $R\to\infty$. Set
			\begin{equation*}
				I_{k,R}(y):=\int_{B_R}|y-z|^{\alpha-n}F(w_k(z))\,\ud z,
				\quad\text{and}\quad T_{k,R}(y):=\int_{\pffOi\setminus B_R}|y-z|^{\alpha-n}F(w_k(z))\,\ud z.
			\end{equation*}
			Then $P_k=I_{k,R}+T_{k,R}$, with both terms nonnegative. Let $I_{0,R}$ be defined by replacing $w_k$ with $w_0$ in the definition of $I_{k,R}$. For $|y|\leqslant L$,
			\begin{equation*}
				\int_{B_R}|y-z|^{\alpha-n}\,\ud z
				\leqslant\int_{B_{R+L}}|\zeta|^{\alpha-n}\,\ud\zeta
				=\frac{|\pffSph|}{\alpha}(R+L)^\alpha.
			\end{equation*}
			Since $F(w_k)\to F(w_0)$ uniformly on $\overline B_R$ as $k\to\infty$, we have
			\begin{equation}\label{p45:Ilimit}
				\sup_{B_L}|I_{k,R}-I_{0,R}|
				\leqslant \frac{|\pffSph|}{\alpha}(R+L)^\alpha
				\|F(w_k)-F(w_0)\|_{L^\infty(B_R)}\to0.
			\end{equation}
			To show that the tail contributes a constant in the limit, we compare $T_{k,R}(y)$ with $T_{k,R}(0)$. For $|y|\leqslant L$, $|z|\geqslant R$ and $0\leqslant t\leqslant1$, we have $|z-ty|\geqslant |z|-L\geqslant |z|/2$. Integration along this segment yields
			\begin{equation*}
				\begin{aligned}
					\bigl||y-z|^{\alpha-n}-|z|^{\alpha-n}\bigr|
					&\leqslant(n-\alpha)|y|\int_0^1|z-ty|^{\alpha-n-1}\,\ud t\\
					&\leqslant C(n,\alpha)L|z|^{\alpha-n-1}
					\leqslant C(n,\alpha)\frac LR|z|^{\alpha-n}.
				\end{aligned}
			\end{equation*}
			Integrating against $F(w_k)$ over $\pffOi\setminus B_R$ and using the uniform bound for $P_k(0)$ from Step A1, we obtain
			\begin{equation}\label{p45:Toscillation}
				|T_{k,R}(y)-T_{k,R}(0)|\leqslant C(L/R)T_{k,R}(0)\leqslant CL/R.
			\end{equation}
			For fixed $R$, the identity $T_{k,R}(0)=P_k(0)-I_{k,R}(0)$ and \eqref{p45:Ilimit} imply, as $k\to\infty$,
			\begin{equation*}
				T_{k,R}(0)\longrightarrow b_R:=P_0(0)-I_{0,R}(0),
			\end{equation*}
			where $0\leqslant b_R\leqslant P_0(0)$.
			Since $F(w_0)\geqslant0$, $I_{0,R}(0)$ increases with $R$. Thus $b_R\downarrow b_0$ as $R\to\infty$ for some $b_0\geqslant0$.
			Passing to the limit $k\to\infty$ in \eqref{p45:Toscillation} and using $T_{k,R}(y)=P_k(y)-I_{k,R}(y)$, we find
			\begin{equation*}
				|P_0(y)-I_{0,R}(y)-b_R|\leqslant CL/R
				\qquad\text{for }|y|\leqslant L.
			\end{equation*}
			Here $C$ is independent of $R,L$. Moreover, $I_{0,R}(y)\leqslant P_0(y)$ follows by letting $k\to\infty$ in $I_{k,R}\leqslant P_k$. As $R\to\infty$, monotone convergence gives
			\begin{equation*}
				I_{0,R}(y)\uparrow\int_{\pffR^n}|y-z|^{\alpha-n}F(w_0(z))\,\ud z
				\leqslant P_0(y)<\infty.
			\end{equation*}
			Letting $R\to\infty$ with $L$ fixed proves \eqref{p45:bdefect} on $B_L$. The constant $b_0=\lim_{R\to\infty}b_R$ is independent of $y$ and $L$, so the identity holds throughout $\pffR^n$.
			
			To show that $b_0=0$, recall from Step A2 that $w_0\in C^2(\pffR^n)$, $w_0>0$, and $-\Delta w_0=P_0w_0^q$ in $\pffR^n$. Since $q=p-1$, \eqref{p45:bdefect} rewrites this equation as
			\begin{equation*}
				-\Delta w_0(y)=\left(\int_{\pffR^n}|y-z|^{\alpha-n}F(w_0(z))\,\ud z+b_0\right)w_0(y)^{p-1}
				\quad\text{in }\pffR^n.
			\end{equation*}
			The Hartree integral is finite at every point by the preceding estimate. If $b_0>0$, the existence of $w_0$ contradicts Lemma~\ref{lem:constantdefect}. Since $b_0\geqslant0$, we conclude that $b_0=0$.
			
			\pffclaimpart{A4}{Identification of the limit}
			With $b_0=0$, \eqref{p45:bdefect} identifies $P_0$ as the Hartree potential of $w_0$, finite at every point. Combined with \eqref{p45:limitPDE}, this representation shows
			that $w_0$ satisfies the hypotheses of Theorem~\ref{thmB}. Therefore
			\begin{equation}\label{p45:bubble}
				w_0(x)=A(a^2+|x-z_0|^2)^{-\beta},\qquad
				A=(a^2+|z_0|^2)^\beta,
			\end{equation}
			where the second equality follows from $w_0(0)=1$.
			This proves Claim A.
		\end{proof}
		
		In what follows, we work with the subsequence obtained in Claim A. We next compare $w_k$ with its Kelvin transform on the domain $\pffOt$ defined in~\eqref{p45:domains} and pass this comparison to the limit. Following Definition~\ref{def:kelvintransform}, write
		\begin{equation}\label{p45:Kelvin}
			x^\mu:=\mathcal I_{0,\mu}(x)=\frac{\mu^2x}{|x|^2},\qquad
			w_{k,\mu}(x):=(w_k)_{0,\mu}(x)=\left(\frac\mu{|x|}\right)^{n-2}w_k(x^\mu),
		\end{equation}
		where $\mu>0$, $x\ne0$, and the second expression is defined whenever $x^\mu\in\pffOi\setminus\{a_k\}$. For each fixed $\mu_0>0$, we have $\overline B_{\mu_0}\subset\pffOi\setminus\{a_k\}$ for large $k$, by \eqref{p45:domains}. For $0<\mu\leqslant\mu_0$ and $|x|\geqslant\mu$, the bound $|x^\mu|\leqslant\mu\leqslant\mu_0$ therefore ensures that all the transforms below are defined.
		
		\begin{pffclaimb}
			For every $\mu>0$, the limit in Claim A satisfies
			\begin{equation*}
				(w_0)_{0,\mu}(y)\leqslant w_0(y)\qquad\text{for }|y|\geqslant\mu.
			\end{equation*}
		\end{pffclaimb}
		\begin{proof}[Proof of Claim B]
			Since $\pffOt\setminus\{a_k\}$ eventually contains every fixed closed ball, it suffices to prove that, for each fixed $\mu_0>0$, there exists an index $k_0(\mu_0)$ such that
			\begin{equation*}
				w_{k,\mu}\leqslant w_k\quad\text{on }\pffOt\setminus(B_\mu\cup\{a_k\})
				\qquad\text{for }0<\mu\leqslant\mu_0,\quad k\geqslant k_0(\mu_0).
			\end{equation*}
			We establish this comparison in Steps B1 and B2, and then pass to the limit in Step B3.
			\pffclaimpart{B1}{Initialization at small radii}
			There exist a radius $\mu_1>0$ and an index $k_1$ such that, for $k\geqslant k_1$ and $0<\mu\leqslant\mu_1$,
			\begin{equation}\label{p45:claim1}
				w_{k,\mu}\leqslant w_k\quad\text{on }\pffOt\setminus(B_\mu\cup\{a_k\}).
			\end{equation}
			
			For large $k$, $\overline B_1\subset\pffOt\setminus\{a_k\}$. Since $w_0>0$, the local $C^1$ convergence in \eqref{p45:limits} provides constants $m_*,M_*>0$, independent of $k$, such that
			$w_k\geqslant m_*$ and $|\nabla w_k|\leqslant M_*$ on $\overline B_1$ for large $k$.
			Choose $0<r_0<1/2$ so that $r_0M_*\leqslant\beta m_*/2$.
			For every fixed $\theta\in\pffSph$, we have $\partial_rw_k(r\theta)=\nabla w_k(r\theta)\cdot\theta\geqslant-M_*$. Hence, for $0<r\leqslant r_0$,
			\begin{equation*}
				\frac{\ud}{\ud r}\bigl(r^\beta w_k(r\theta)\bigr)
				=r^{\beta-1}\bigl(\beta w_k(r\theta)+r\partial_rw_k(r\theta)\bigr)
				\geqslant\tfrac12\beta m_*r^{\beta-1}>0.
			\end{equation*}
			For $0<\mu<r\leqslant r_0$, we have $0<\mu^2/r<r\leqslant r_0$. Applying this monotonicity at the radii $\mu^2/r$ and $r$ and using $2\beta=n-2$, we find
			\begin{equation*}
				w_k(r\theta)>
				\left(\frac\mu r\right)^{n-2}
				w_k\left(\frac{\mu^2}{r}\theta\right)=w_{k,\mu}(r\theta)
				\quad\text{for }0<\mu<r\leqslant r_0.
			\end{equation*}
			
			To compare the functions outside $B_{r_0}$, we use the integral representation. Fix $0<\rho<r_0$. For large $k$, $\overline B_\rho\subset\pffOo\cap\pffOi$ and $a_k\notin\overline B_\rho$. Since $h_k>0$ and $w_k\geqslant m_*$ on $B_\rho$, \eqref{p45:scaledintegral} and \eqref{p45:scaledobjects} give, for every $y\in\pffOo\setminus\{a_k\}$,
			\begin{align*}
				w_k(y)&\geqslant c_n\int_{B_\rho}|y-z|^{2-n}P_k(z)w_k(z)^q\,\ud z\\
				&\geqslant\frac{c_n}{p}\int_{B_\rho}|y-z|^{2-n}w_k(z)^q
				\int_{B_\rho}|z-\xi|^{\alpha-n}w_k(\xi)^p\,\ud\xi\,\ud z\\
				&\geqslant \frac{c_n}{p}(2\rho)^{\alpha-n}|B_\rho|^2m_*^{p+q}(|y|+\rho)^{2-n}
				\geqslant c_*(1+|y|)^{2-n},
			\end{align*}
			where $c_*:=(c_n/p)(2\rho)^{\alpha-n}|B_\rho|^2m_*^{p+q}>0$. Here we used $|z-\xi|\leqslant2\rho$ and $|y-z|\leqslant|y|+\rho$, together with $\rho <1$.
			
			Since $R_k\to\infty$, we also have $w_k\leqslant C_*$ on $\overline B_{r_0}$ for large $k$, by~\eqref{p45:normalization}.
			For $|y|\geqslant r_0$ and $0<\mu<r_0$, $|y^\mu|=\mu^2/|y|\leqslant\mu^2/r_0<r_0$. Thus the reflected point lies in the fixed ball $B_{r_0}$, away from $a_k$, and
			$w_{k,\mu}(y)\leqslant C_*\mu^{n-2}|y|^{2-n}$.
			Choose
			\begin{equation*}
				\mu_1:=\min\left\{\frac{r_0}{2},
				\left(\frac{c_*}{2C_*}\right)^{1/(n-2)}\frac{r_0}{1+r_0}\right\}.
			\end{equation*}
			For $0<\mu\leqslant\mu_1$ and $r=|y|\geqslant r_0$, we have $r_0/(1+r_0)\leqslant r/(1+r)$, so
			\begin{equation*}
				\begin{aligned}
					w_{k,\mu}(y)&\leqslant C_*\mu^{n-2}r^{2-n}\\
					&\leqslant\frac{c_*}{2}\left(\frac{r_0}{1+r_0}\right)^{n-2}r^{2-n}
					\leqslant\frac{c_*}{2}(1+r)^{2-n}
					\leqslant\frac12w_k(y)
				\end{aligned}
			\end{equation*}
			for $y\in\pffOo\setminus\{a_k\}$ with $|y|\geqslant r_0$. Together with the comparison for $\mu<|y|\leqslant r_0$ and equality on $|y|=\mu$, this proves \eqref{p45:claim1}, since $\pffOt\subset\pffOo$. The constants were chosen independently of $k$, so a single index $k_1$ ensures all the bounds and inclusions above for every $k\geqslant k_1$ and $0<\mu\leqslant\mu_1$.
			\par\medskip
			If $\mu_0\leqslant\mu_1$, Step B1 gives the required comparison with $k_0(\mu_0)=k_1$. We therefore fix $\mu_0>\mu_1$. For all sufficiently large $k$, define
			\begin{equation}\label{p45:criticalradius}
				\mu_k^*:=\sup\left\{\ell\in(0,\mu_0]:
				w_{k,\mu}\leqslant w_k\text{ on }\pffOt\setminus(B_\mu\cup\{a_k\})
				\text{ for every }0<\mu\leqslant\ell\right\}.
			\end{equation}
			By Step B1, the defining set is nonempty and $\mu_1\leqslant\mu_k^*\leqslant\mu_0$.
			
			\pffclaimpart{B2}{Continuation to any fixed radius}
			We show that $\mu_k^*=\mu_0$ for all sufficiently large $k$.
			
			Suppose, to the contrary, that $\mu_k^*<\mu_0$ for arbitrarily large $k$. For these $k$, we will show that the comparison extends beyond $\mu_k^*$ once $k$ is sufficiently large, contradicting the definition of $\mu_k^*$.
			Following~\eqref{prelim:potentialkelvin}, define
			\begin{equation}\label{p45:potentialKelvinDefinition}
				P_{k,\mu}(y):=\left(\frac\mu{|y|}\right)^{n-\alpha}P_k(y^\mu).
			\end{equation}
			Here $P_k$ is the potential defined in~\eqref{p45:scaledobjects}, with the full integration domain $\pffOi$. It is finite on $\pffOi\setminus\{a_k\}$, as shown after~\eqref{p45:mass}. We also write
			\begin{equation*}
				W_{k,\mu}(x):=w_k(x)-w_{k,\mu}(x)
			\end{equation*}
			wherever both terms are defined.
			
			\pffproofsubstep{B2.1}{Comparison at the critical radius and on the outer annulus}
			By \eqref{p45:scaling}--\eqref{p45:domains}, for all sufficiently large $k$,
			\begin{equation*}
				|\bar x_k|<\min\{1,\tau'/2\},\qquad R_k>\mu_0+2,\qquad
				|a_k|>2(\mu_0+2),\qquad
				\overline B_{\mu_0+2}\subset\pffOt\setminus\{a_k\}.
			\end{equation*}
			For each $0<\mu<\mu_k^*$, the set in \eqref{p45:criticalradius} contains an $\ell>\mu$, so the comparison holds at $\mu$.
			Fix $x\in\pffOt\setminus\{a_k\}$ with $|x|>\mu_k^*$. For $0<\mu\leqslant\mu_k^*$, the reflected points satisfy $|x^\mu|=\mu^2/|x|<\mu_k^*\leqslant\mu_0$ and remain in a fixed nonsingular ball. Continuity in \eqref{p45:Kelvin} therefore allows $\mu\uparrow\mu_k^*$ in the comparison at $x$. On $|x|=\mu_k^*$, inversion fixes $x$, so equality holds. Thus
			\begin{equation}\label{p45:criticalcomparison}
				W_{k,\mu_k^*}\geqslant0
				\quad\text{on }\pffOt\setminus(\overline B_{\mu_k^*}\cup\{a_k\}),
				\qquad W_{k,\mu_k^*}=0\text{ on }\partial B_{\mu_k^*}.
			\end{equation}
			
			We also need comparison on $\pffOo\setminus\pffOt$, which is not part of the domain in \eqref{p45:criticalradius}. Set
			\begin{equation*}
				m_o:=\min_{\tau'\leqslant|X|\leqslant\tau}u(X)>0.
			\end{equation*}
			For $0<\mu\leqslant\mu_0$ and $x\in\pffOo\setminus\pffOt$, the point $X=\bar x_k+\lambda_kx$ satisfies $\tau'\leqslant|X|<\tau$.
			By the definition of $m_o$,
			\begin{equation*}
				w_k(x)=\lambda_k^\beta u(X)\geqslant m_o\lambda_k^\beta.
			\end{equation*}
			Here $m_o$ is independent of $k$ and $\mu$. Moreover,
			\begin{equation*}
				|x|=\lambda_k^{-1}|X-\bar x_k|
				\geqslant\lambda_k^{-1}(\tau'-|\bar x_k|)
				\geqslant\frac{\tau'}{2\lambda_k}.
			\end{equation*}
			Since $\lambda_k\to0$ and $R_k\to\infty$, for all sufficiently large $k$,
			\begin{equation*}
				|x^\mu|=\frac{\mu^2}{|x|}\leqslant\frac{2\mu_0^2\lambda_k}{\tau'}<1,
				\qquad B_1\subset B_{R_k}.
			\end{equation*}
			At the reflected point, the peak bound \eqref{p45:normalization} reads $w_k(x^\mu)\leqslant C_*$. By \eqref{p45:Kelvin},
			\begin{equation*}
				w_{k,\mu}(x)=\left(\frac{\mu}{|x|}\right)^{n-2}w_k(x^\mu)
				\leqslant C_*\left(\frac{2\mu_0\lambda_k}{\tau'}\right)^{n-2}.
			\end{equation*}
			Since $n-2=2\beta$, these bounds imply
			\begin{equation*}
				\frac{w_{k,\mu}(x)}{w_k(x)}
				\leqslant\frac{C_*}{m_o}\left(\frac{2\mu_0}{\tau'}\right)^{n-2}
				\lambda_k^\beta\longrightarrow0\qquad\text{as }k\to\infty
			\end{equation*}
			uniformly for $0<\mu\leqslant\mu_0$ and $x\in\pffOo\setminus\pffOt$. The ratio $w_{k,\mu}(x)/w_k(x)$ is therefore at most $1/2$ for all sufficiently large $k$. Taking $\mu=\mu_k^*$ and combining this bound with \eqref{p45:criticalcomparison} extends the comparison to $\pffOo$. Moreover, $\partial\pffOt\subset\pffOo\setminus\pffOt$, so $W_{k,\mu_k^*}\geqslant m_o\lambda_k^\beta/2>0$ on this boundary. Hence
			\begin{equation}\label{p45:outercomparison}
				W_{k,\mu_k^*}\geqslant0\quad\text{on }\pffOo\setminus(B_{\mu_k^*}\cup\{a_k\}),
				\qquad W_{k,\mu_k^*}>0\quad\text{on }\partial\pffOt.
			\end{equation}
			
			\pffproofsubstep{B2.2}{Potential difference identity and kernel estimates}
			The comparison \eqref{p45:outercomparison} controls $w_k$ only on $\pffOo\setminus B_{\mu_k^*}$, whereas $P_k$ also includes density outside $\pffOo$. We therefore separate those contributions before comparing the potentials. Throughout this step,
			$y\in\pffOo\setminus(\overline B_{\mu_k^*}\cup\{a_k\})$.
			
			For the following integral calculations, define the zero extension of $F(w_k)$ by
			\begin{equation*}
				\widetilde F_k(z):=
				\begin{cases}
					F(w_k(z)),&z\in\pffOi\setminus\{a_k\},\\
					0,&\text{otherwise}.
				\end{cases}
			\end{equation*}
			By \eqref{p45:mass}, $\widetilde F_k\in L^1(\pffR^n)$, and
			\begin{equation*}
				P_k(y)=\int_{\pffR^n}|y-z|^{\alpha-n}\widetilde F_k(z)\,\ud z.
			\end{equation*}
			For $0<\mu\leqslant\mu_0$, define
			\begin{equation*}
				\widetilde F_{k,\mu}(z):=\left(\frac\mu{|z|}\right)^{n+\alpha}\widetilde F_k(z^\mu)
				\qquad\text{for }z\ne0.
			\end{equation*}
			For $|z|\geqslant\mu_k^*$, the inversion point satisfies
			\begin{equation*}
				|z^{\mu_k^*}|\leqslant\mu_k^*,\qquad
				\overline B_{\mu_k^*}\subset\pffOi\setminus\{a_k\}.
			\end{equation*}
			Thus, by the definition of $w_{k,\mu_k^*}$ in \eqref{p45:Kelvin} and the identity $(n-2)p=n+\alpha$,
			\begin{equation}\label{p45:densitycovariance}
				\widetilde F_{k,\mu_k^*}(z)=
				\left(\frac{\mu_k^*}{|z|}\right)^{n+\alpha}F(w_k(z^{\mu_k^*}))
				=F(w_{k,\mu_k^*}(z))\qquad\text{for }|z|\geqslant\mu_k^*.
			\end{equation}
			As in Proposition~\ref{prop:conformalinvariance}, the substitution
			$\xi=z^{\mu_k^*}$ in \eqref{p45:potentialKelvinDefinition} yields
			\begin{equation*}
				P_{k,\mu_k^*}(y)
				=\int_{\pffR^n}|y-z|^{\alpha-n}
				\widetilde F_{k,\mu_k^*}(z)\,\ud z.
			\end{equation*}
			Both potentials are finite, since $y$ and $y^{\mu_k^*}$ are regular points. Thus
			\begin{equation*}
				\int_{\pffR^n}|y-z|^{\alpha-n}
				\left|\widetilde F_k(z)-\widetilde F_{k,\mu_k^*}(z)\right|\,\ud z\leqslant P_k(y)+P_{k,\mu_k^*}(y)<\infty,
			\end{equation*}
			which justifies subtracting the two integrals and changing variables in their difference.
			
			The same inversion calculation as in \eqref{v-vxmu}, applied to $\widetilde F_k$ and using \eqref{p45:densitycovariance}, gives
			\begin{align}
				P_k(y)-P_{k,\mu_k^*}(y)
				&=\int_{\pffR^n\setminus B_{\mu_k^*}}K_\alpha(0,\mu_k^*;y,z)
				[\widetilde F_k(z)-F(w_{k,\mu_k^*}(z))]\,\ud z\notag\\
				&=\int_{\pffOo\setminus B_{\mu_k^*}}K_\alpha(0,\mu_k^*;y,z)
				[F(w_k(z))-F(w_{k,\mu_k^*}(z))]\,\ud z\notag\\
				&\quad+\mathcal A_{k,\mu_k^*}(y)-\mathcal N_{k,\mu_k^*}(y),\label{p45:potentialdifference}
			\end{align}
			where
			\begin{equation}\label{p45:Adefinition}
				\mathcal A_{k,\mu_k^*}(y):=\int_{\pffOi\setminus\pffOo}K_\alpha(0,\mu_k^*;y,z)F(w_k(z))\,\ud z,
			\end{equation}
			and
			\begin{equation}\label{p45:Ndefinition}
				\mathcal N_{k,\mu_k^*}(y):=\int_{\pffR^n\setminus\pffOo}K_\alpha(0,\mu_k^*;y,z)F(w_{k,\mu_k^*}(z))\,\ud z.
			\end{equation}
			Here $\widetilde F_k$ is zero outside $\pffOi$, while $\widetilde F_{k,\mu_k^*}$ may be nonzero there.
			
			For the kernel $K_\alpha$ of Lemma~\ref{lm:differencecomparison}, write
			\begin{equation*}
				s_{\mathcal K}:=s_{\mathcal K}(0,\mu_k^*;y,z)=\frac{(|y|^2-(\mu_k^*)^2)(|z|^2-(\mu_k^*)^2)}{(\mu_k^*)^2|y-z|^2}>0
				\quad\text{for }|y|,|z|>\mu_k^*,\ y\ne z.
			\end{equation*}
			By~\eqref{prelim:kernelestimate}, the kernel is positive and satisfies
			\begin{equation*}
				c|y-z|^{\alpha-n}\min\{s_{\mathcal K},1\}
				\leqslant K_\alpha(0,\mu_k^*;y,z)
				\leqslant C|y-z|^{\alpha-n}\min\{s_{\mathcal K},1\},
			\end{equation*}
			with constants depending only on $n,\alpha$, independently of $k$ and $\mu_k^*$. The bound $0\leqslant K_\alpha\leqslant|y-z|^{\alpha-n}$ gives
            \begin{equation*}
            0\leqslant\mathcal A_{k,\mu_k^*}(y)\leqslant P_k(y)<\infty,
            \qquad
            0\leqslant\mathcal N_{k,\mu_k^*}(y)\leqslant P_{k,\mu_k^*}(y)<\infty.
            \end{equation*}
            Moreover, \eqref{p45:outercomparison} and the monotonicity of $F$ ensure that the integral over $\pffOo\setminus B_{\mu_k^*}$ in \eqref{p45:potentialdifference} is nonnegative. Hence, to establish the strict potential comparison, it suffices to prove
            \begin{equation*} 
            \mathcal A_{k,\mu_k^*}(y)>\mathcal N_{k,\mu_k^*}(y).
            \end{equation*}
			\pffproofsubstep{B2.3}{Exterior estimates and strict potential comparison}
			We estimate $\mathcal A_{k,\mu_k^*}$ from below and $\mathcal N_{k,\mu_k^*}$ from above to prove the comparison. Both estimates will contain the factor $\vartheta_{\mu_k^*}(y):=\min\{|y|-\mu_k^*,1\}$.
			For the lower bound, choose $\tau<r_-<r_+<2$ and set
			\begin{equation*}
				D:=\{Z:r_-\leqslant|Z|\leqslant r_+\},\qquad
				E_k:=\{z:\bar x_k+\lambda_kz\in D\}\subset\pffOi\setminus\pffOo.
			\end{equation*}
			Since $\bar x_k\to0$ and $\lambda_k\to0$, we have $|\bar x_k|\leqslant r_-/2$ and $\mu_0^2\leqslant r_-^2\lambda_k^{-2}/8$ for all sufficiently large $k$.
			For $y\in\pffOo\setminus(\overline B_{\mu_k^*}\cup\{a_k\})$ and $z\in E_k$, put $Y=\bar x_k+\lambda_ky$ and $Z=\bar x_k+\lambda_kz$. Since $|Y|<\tau$ and $r_-\leqslant|Z|\leqslant r_+$, the triangle inequality gives
			\begin{equation}\label{p45:annulusgeometry}
				(r_--\tau)\lambda_k^{-1}\leqslant|y-z|\leqslant(r_++\tau)\lambda_k^{-1},
				\qquad |z|\geqslant\tfrac12r_-\lambda_k^{-1}.
			\end{equation}
			Consequently,
			\begin{equation*}
				|z|^2-(\mu_k^*)^2\geqslant\frac{r_-^2}{4}\lambda_k^{-2}-\mu_0^2
				\geqslant\frac{r_-^2}{8}\lambda_k^{-2},\qquad
				\frac{|z|^2-(\mu_k^*)^2}{|y-z|^2}\geqslant
				\frac{r_-^2}{8(r_++\tau)^2}=:c_0>0.
			\end{equation*}
			It follows from the definition of $s_{\mathcal K}$ that
			\begin{equation*}
				s_{\mathcal K}\geqslant c_0\frac{(|y|-\mu_k^*)(|y|+\mu_k^*)}{(\mu_k^*)^2}
				\geqslant\frac{2c_0}{\mu_0}(|y|-\mu_k^*),
				\qquad
				\min\{s_{\mathcal K},1\}\geqslant\min\left\{\frac{2c_0}{\mu_0},1\right\}\vartheta_{\mu_k^*}(y).
			\end{equation*}
			Using the preceding estimate and \eqref{p45:annulusgeometry}, we obtain from the lower kernel bound in Lemma~\ref{lm:differencecomparison} that
            \begin{equation}\label{p45:Akernel}
            K_\alpha(0,\mu_k^*;y,z)\geqslant c\vartheta_{\mu_k^*}(y)\lambda_k^{n-\alpha} \quad\text{for }z\in E_k.
            \end{equation}
			Since $u$ is positive on the compact annulus $D$, we have $\min_D F(u)>0$.
			We restrict the nonnegative integral in \eqref{p45:Adefinition} to $E_k$ and apply \eqref{p45:Akernel}. Changing variables by $Z=\bar x_k+\lambda_kz$ and using $F(w_k(z))=\lambda_k^{(n+\alpha)/2}F(u(Z))$, we obtain that
            \begin{align}\label{p45:Alower}
            \mathcal A_{k,\mu_k^*}(y)
            &\geqslant c\vartheta_{\mu_k^*}(y)\lambda_k^{n-\alpha} \lambda_k^{(n+\alpha)/2}\lambda_k^{-n}\bigl(\min_D F(u)\bigr)|D|\notag\\
            &=c_{\mathcal A}\vartheta_{\mu_k^*}(y)\lambda_k^{(n-\alpha)/2},
            \end{align}
            where $c_{\mathcal A}>0$ is independent of $k$, $y$, and $\mu_k^*$.
			
			We next estimate $\mathcal N_{k,\mu_k^*}$ from above.
			By \hyperref[p45:step:B2.1]{Step B2.1}, $|\bar x_k|<\tau/2$. For $z\notin\pffOo$, we have $|\bar x_k+\lambda_kz|\geqslant\tau$. Using the triangle inequality and the inversion formula, we estimate
			\begin{equation*}
				|z|\geqslant\frac{\tau-|\bar x_k|}{\lambda_k}\geqslant\frac{\tau}{2\lambda_k},
				\qquad |z^{\mu_k^*}|=\frac{(\mu_k^*)^2}{|z|}\leqslant\frac{2\mu_0^2\lambda_k}{\tau}.
			\end{equation*}
			Since $\lambda_k\to0$, the reflected point lies in $B_1$ for all sufficiently large $k$. As $B_1\subset B_{R_k}$, we have $w_k(z^{\mu_k^*})\leqslant C_*$ by the normalization \eqref{p45:normalization}. Therefore, using \eqref{p45:densitycovariance} and $\mu_k^*\leqslant\mu_0$, we obtain
			\begin{equation}\label{p45:negdensity}
				F(w_{k,\mu_k^*}(z))=\frac1p\left(\frac{\mu_k^*}{|z|}\right)^{n+\alpha}w_k(z^{\mu_k^*})^p
				\leqslant C|z|^{-n-\alpha}.
			\end{equation}
			For the kernel, we retain the factor $\vartheta_{\mu_k^*}(y)$ by distinguishing two cases.
			If $|y|-\mu_k^*\geqslant1$, then $\vartheta_{\mu_k^*}(y)=1$ and $K_\alpha\leqslant|y-z|^{\alpha-n}$.
			If $0<|y|-\mu_k^*<1$, then $|y|\leqslant\mu_0+1$. For all sufficiently large $k$, the bound $4(\mu_0+1)\lambda_k\leqslant\tau$ ensures that
			\begin{equation*}
				|z|\geqslant2(\mu_0+1)\geqslant2|y|,\qquad
				|y-z|\geqslant|z|-|y|\geqslant|z|/2.
			\end{equation*}
			Consequently, $(|z|^2-(\mu_k^*)^2)/|y-z|^2\leqslant4$. Using $|y|+\mu_k^*\leqslant2\mu_0+1$ and $\mu_k^*\geqslant\mu_1$ in the definition of $s_{\mathcal K}$, we obtain
			\begin{equation*}
				s_{\mathcal K}=\frac{(|y|-\mu_k^*)(|y|+\mu_k^*)}{(\mu_k^*)^2}
				\frac{|z|^2-(\mu_k^*)^2}{|y-z|^2}
				\leqslant\frac{4(2\mu_0+1)}{\mu_1^2}(|y|-\mu_k^*).
			\end{equation*}
			Since $\vartheta_{\mu_k^*}(y)=|y|-\mu_k^*$ in this case, the upper kernel bound in Lemma~\ref{lm:differencecomparison} and $\min\{s_{\mathcal K},1\}\leqslant s_{\mathcal K}$ yield the required factor. Combining the two cases, we obtain
			\begin{equation}\label{p45:Nkernel}
				K_\alpha(0,\mu_k^*;y,z)\leqslant C\vartheta_{\mu_k^*}(y)|y-z|^{\alpha-n}
				\quad\text{for }z\notin\pffOo.
			\end{equation}
			
			Combining \eqref{p45:Ndefinition}, \eqref{p45:negdensity}, and \eqref{p45:Nkernel}, and changing variables by $\eta=\lambda_ky$ and $\zeta=\lambda_kz$, we obtain
			\begin{align*}
				\mathcal N_{k,\mu_k^*}(y)
				&\leqslant C\vartheta_{\mu_k^*}(y)
				\int_{\pffR^n\setminus\pffOo}|y-z|^{\alpha-n}|z|^{-n-\alpha}\,\ud z\notag\\
				&=C\vartheta_{\mu_k^*}(y)\lambda_k^n
				\int_{\{|\bar x_k+\zeta|\geqslant\tau\}}|\eta-\zeta|^{\alpha-n}|\zeta|^{-n-\alpha}\,\ud\zeta.
			\end{align*}
			By \hyperref[p45:step:B2.1]{Step B2.1}, $|\bar x_k|\leqslant\min\{1,\tau/2\}$, so $|\eta|\leqslant\tau+1$ and $|\zeta|\geqslant\tau/2$ on the integration set.
			Since $\alpha>0$, the kernel $|\eta-\zeta|^{\alpha-n}$ is locally integrable, so the last integral is uniformly bounded when restricted to $|\zeta|\leqslant2(\tau+1)$.
			For $|\zeta|>2(\tau+1)$, we have $|\eta-\zeta|\geqslant|\zeta|/2$, so the integrand is at most $C|\zeta|^{-2n}$, which is integrable at infinity.
			Consequently,
			\begin{equation}\label{p45:Nupper}
				\mathcal N_{k,\mu_k^*}(y)\leqslant C_{\mathcal N}\vartheta_{\mu_k^*}(y)\lambda_k^n,
			\end{equation}
			where $C_{\mathcal N}>0$ is independent of $k$, $y$, and $\mu_k^*\in[\mu_1,\mu_0]$.
			
			Combining these estimates gives the desired strict comparison at the critical radius.
			For $y\in\pffOo\setminus(\overline B_{\mu_k^*}\cup\{a_k\})$, subtract the upper bound \eqref{p45:Nupper} from the lower bound \eqref{p45:Alower} to obtain
			\begin{align*}
				\mathcal A_{k,\mu_k^*}(y)-\mathcal N_{k,\mu_k^*}(y)
				&\geqslant\vartheta_{\mu_k^*}(y)
				\bigl(c_{\mathcal A}\lambda_k^{(n-\alpha)/2}-C_{\mathcal N}\lambda_k^n\bigr)\\
				&=\vartheta_{\mu_k^*}(y)\lambda_k^{(n-\alpha)/2}
				\bigl(c_{\mathcal A}-C_{\mathcal N}\lambda_k^{(n+\alpha)/2}\bigr).
			\end{align*}
			Since $\lambda_k\to0$, we have $C_{\mathcal N}\lambda_k^{(n+\alpha)/2}\leqslant c_{\mathcal A}/2$ for all sufficiently large $k$. Equations \eqref{p45:potentialdifference} and \eqref{p45:outercomparison} therefore imply
			\begin{equation}\label{p45:strictpotential}
				P_k(y)-P_{k,\mu_k^*}(y)
				\geqslant\frac{c_{\mathcal A}}{2}\vartheta_{\mu_k^*}(y)\lambda_k^{(n-\alpha)/2}>0
				\quad\text{on }\pffOo\setminus(\overline B_{\mu_k^*}\cup\{a_k\}).
			\end{equation}

             \smallskip
             
			For the rest of the contradiction argument, we fix an index $k$ as above with $\mu_k^*<\mu_0$. The parameters used to continue the comparison may depend on this fixed $k$.
			
			\pffproofsubstep{B2.4}{Strict interior positivity of $W_{k,\mu_k^*}$}
			Let $x\in\pffOt\setminus(\overline B_{\mu_k^*}\cup\{a_k\})$.
			We have $|x^{\mu_k^*}|<\mu_k^*\leqslant\mu_0$. It follows from \hyperref[p45:step:B2.1]{Step B2.1} that $\overline B_{\mu_0}\subset\pffOi\setminus\{a_k\}$. Hence $w_k$ is $C^2$ near both $x$ and $x^{\mu_k^*}$, and the Kelvin formula shows that $w_k$ and $w_{k,\mu_k^*}$ are $C^2$ on $\pffOt\setminus(\overline B_{\mu_k^*}\cup\{a_k\})$.
			Applying the Kelvin--Laplacian identity \eqref{prelim:laplacekelvin} to equation \eqref{p45:scaledPDE}, we find
			\begin{equation*}
				-\Delta w_{k,\mu_k^*}=P_{k,\mu_k^*}w_{k,\mu_k^*}^q.
			\end{equation*}
			We subtract the equation for $w_{k,\mu_k^*}$ from \eqref{p45:scaledPDE}. Since $t^q$ is increasing for $q>0$, \eqref{p45:outercomparison} and \eqref{p45:strictpotential} give
			\begin{equation}\label{p45:strictPDE}
				-\Delta W_{k,\mu_k^*}=P_k(w_k^q-w_{k,\mu_k^*}^q)
				+(P_k-P_{k,\mu_k^*})w_{k,\mu_k^*}^q>0
				\quad\text{in }\pffOt\setminus(\overline B_{\mu_k^*}\cup\{a_k\}).
			\end{equation}
			The comparison in \eqref{p45:outercomparison} implies that $W_{k,\mu_k^*}\geqslant0$. If $W_{k,\mu_k^*}$ vanished at an interior point, it would attain a local minimum there, and hence $\Delta W_{k,\mu_k^*}\geqslant0$ at that point, contradicting \eqref{p45:strictPDE}. Therefore,
            \begin{equation*}
            W_{k,\mu_k^*}>0
            \qquad\text{in }\pffOt\setminus(\overline B_{\mu_k^*}\cup\{a_k\}).
            \end{equation*}

            \smallskip

            In the next three steps, we extend the comparison beyond the critical radius by considering three regions: a neighborhood of the inner sphere, a punctured neighborhood of the singular point, and the remaining compact region.
			\pffproofsubstep{B2.5}{Hopf's lemma and continuation near the inner sphere}
			We first prove that the radial derivative is strictly positive on the inner sphere, and then use continuity to continue the comparison in a narrow annulus near this sphere. Fix $\theta\in\pffSph$ and set $x_0=\mu_k^*\theta$.
			The ball $B_{1/4}((\mu_k^*+1/4)\theta)$ is tangent to
			$\partial B_{\mu_k^*}$ at $x_0$. For $x$ in its closure,
			\begin{equation*}
				0<\mu_k^*\leqslant|x|\leqslant\mu_k^*+\frac12<\mu_0+1,
				\qquad
				|x^{\mu_k^*}|=\frac{(\mu_k^*)^2}{|x|}
				\leqslant\mu_k^*\leqslant\mu_0.
			\end{equation*}
			By \hyperref[p45:step:B2.1]{Step B2.1}, the tangent ball lies in
			$\pffOt\setminus(\overline B_{\mu_k^*}\cup\{a_k\})$,
			and $w_k$ is $C^2$ near both $x$ and $x^{\mu_k^*}$.
			The Kelvin formula then implies
			$W_{k,\mu_k^*}\in C^2$ in a neighborhood of the closed tangent ball. By \hyperref[p45:step:B2.4]{Step B2.4}
			and \eqref{p45:criticalcomparison},
			\begin{equation*}
				W_{k,\mu_k^*}>0, \quad  -\Delta W_{k,\mu_k^*}>0 \quad\text{in }B_{1/4}((\mu_k^*+1/4)\theta),\quad \text{ and }\quad W_{k,\mu_k^*}(x_{0})=0,
			\end{equation*}
			The outward unit normal to the tangent ball at $x_0$ is
			$\nu=-\theta$. Hence Hopf's boundary lemma yields
			\begin{equation*}
				\partial_r W_{k,\mu_k^*}(\mu_k^*\theta)
				=-\partial_\nu W_{k,\mu_k^*}(x_0)>0.
			\end{equation*}
			Since $\theta$ is arbitrary, the following minimum is positive by continuity and compactness:
			\begin{equation}\label{p45:Hopf}
				m_k:=\min_{\theta\in\pffSph}
				\partial_r W_{k,\mu_k^*}(\mu_k^*\theta)>0.
			\end{equation}
			
			We now use this positive derivative to continue the comparison near the inner sphere.
			The Kelvin formula and the regularity in
			\hyperref[p45:step:B2.1]{Step B2.1} imply that
			$\partial_rW_{k,\mu}(r\theta)$ is continuous in
			$(\mu,r,\theta)$ for $\mu,r$ close to $\mu_k^*$ and
			$\theta\in\pffSph$.
			By \eqref{p45:Hopf} and compactness of $\pffSph$,
			we may therefore choose
			\begin{equation*}
				0<\delta_k<
				\min\left\{\frac14,\frac{\mu_0+1-\mu_k^*}{4}\right\},
				\qquad \varepsilon_{k,1}=\frac{\delta_k}{2},
			\end{equation*}
			so that, for every $\theta\in\pffSph$,
			\begin{equation*}
				\partial_rW_{k,\mu}(r\theta)\geqslant\frac{m_k}{2}, \quad \text{for } \;
				\mu_k^*\leqslant\mu\leqslant\mu_k^*+\varepsilon_{k,1},
				\;\; \mu\leqslant r\leqslant\mu_k^*+2\delta_k.
			\end{equation*}
			Since $W_{k,\mu}(\mu\theta)=0$, radial integration yields,
			for the same ranges of $\mu,r$ and $\theta$,
			\begin{equation}\label{p45:collar}
				W_{k,\mu}(r\theta)
				=\int_\mu^r\partial_tW_{k,\mu}(t\theta)\,\ud t
				\geqslant\frac{m_k}{2}(r-\mu)\geqslant0.
			\end{equation}
			
			\pffproofsubstep{B2.6}{Comparison near the singular point}
			We show that $W_{k,\mu}>0$ near $a_k$ for all $0<\mu\leqslant\mu_0$.
			Choose
			\begin{equation*}
				0<\bar r_k<\min\left\{1,\frac{\tau'}{2\lambda_k},
				\frac{|a_k|-(\mu_0+1)}2\right\}.
			\end{equation*}
			For $|x-a_k|\leqslant\bar r_k$, we have
			\begin{equation*}
				|\bar x_k+\lambda_kx|=\lambda_k|x-a_k|<\tau',
				\qquad |x|\geqslant|a_k|-\bar r_k>\mu_0+1.
			\end{equation*}
			Hence $\overline B_{\bar r_k}(a_k)\subset\pffOt\setminus\overline B_{\mu_0+1}$.
			For $x\in B_{\bar r_k}(a_k)\setminus\{a_k\}$ and $0<\mu\leqslant\mu_0$,
			\begin{equation*}
				|x^\mu|=\frac{\mu^2}{|x|}\leqslant\frac{\mu_0^2}{\mu_0+1}<\mu_0<R_k.
			\end{equation*}
			By \eqref{p45:normalization}, $w_k(x^\mu)\leqslant C_*$. The Kelvin formula gives
			\begin{equation}\label{p45:punctureKelvin}
				w_{k,\mu}(x)\leqslant C_*\left(\frac{\mu_0}{\mu_0+1}\right)^{n-2}\leqslant C_*
				\quad\text{for }x\in B_{\bar r_k}(a_k)\setminus\{a_k\},\quad0<\mu\leqslant\mu_0.
			\end{equation}
			By \eqref{p45:originalPDE}, $u(X)\to+\infty$ as $X\to0$. For this fixed $k$, we can therefore choose $\zeta_k>0$ such that
			$u(X)>2C_*\lambda_k^{-\beta}$ whenever $0<|X|<\zeta_k$.
			Choose $0<r_k<\min\{\bar r_k,\zeta_k/\lambda_k\}$.
			If $0<|x-a_k|<r_k$, then $X=\bar x_k+\lambda_kx=\lambda_k(x-a_k)$ satisfies $0<|X|<\zeta_k$. By \eqref{p45:scaling}, we have $w_k(x)=\lambda_k^\beta u(X)>2C_*$.
			Combining this with \eqref{p45:punctureKelvin}, we obtain
			\begin{equation}\label{p45:punctureprotection}
				w_k(x)-w_{k,\mu}(x)>C_*
				\quad\text{for }0<|x-a_k|<r_k,\quad0<\mu\leqslant\mu_0.
			\end{equation}
			Here $r_k$ may depend on $k$, but not on $\mu$.
			
			\pffproofsubstep{B2.7}{Continuation on the remaining compact set}
			We next compare $w_k$ and $w_{k,\mu}$ away from the sphere $\partial B_{\mu_k^*}$ and the point $a_k$. Set
			\begin{equation*}
				K_k:=\{x\in\overline\pffOt:|x|\geqslant\mu_k^*+\delta_k,
				\ |x-a_k|\geqslant r_k/2\}.
			\end{equation*}
			The set $K_k$ is compact and contains
			$\partial B_{\mu_k^*+\delta_k}$.
			The regularity in \hyperref[p45:step:B2.1]{Step B2.1}
			implies that $W_{k,\mu_k^*}$ is continuous on $K_k$.
			The interior positivity from \hyperref[p45:step:B2.4]{Step B2.4}
			and the boundary estimate \eqref{p45:outercomparison} on
			$\partial\pffOt$ therefore give
			\begin{equation*}
				\widehat m_k:=\min_{x\in K_k}W_{k,\mu_k^*}(x)>0.
			\end{equation*}
			For $x\in K_k$ and
			$\mu_k^*\leqslant\mu\leqslant\mu_k^*+\delta_k$,
			we have $|x^\mu|\leqslant\mu_k^*+\delta_k<\mu_0+1$.
			By the Kelvin formula, the map
			$(\mu,x)\mapsto W_{k,\mu}(x)$ is continuous on
			$[\mu_k^*,\mu_k^*+\delta_k]\times K_k$.
			Consequently, $W_{k,\mu}\to W_{k,\mu_k^*}$ uniformly on $K_k$
			as $\mu\downarrow\mu_k^*$.
			There exists $0<\varepsilon_{k,2}<\delta_k/2$ such that
			\begin{equation}\label{p45:compactcontinuation}
				W_{k,\mu}(x)\geqslant\frac{\widehat m_k}{2}>0
				\quad\text{for }x\in K_k,\quad
				\mu_k^*\leqslant\mu\leqslant\mu_k^*+\varepsilon_{k,2}.
			\end{equation}
			
			\pffproofsubstep{B2.8}{Completion of the continuation argument}
			To combine the comparisons in \hyperref[p45:step:B2.5]{Steps B2.5}--\hyperref[p45:step:B2.7]{B2.7}, take
			\begin{equation*}
				\varepsilon_k:=\min\{\varepsilon_{k,1},\varepsilon_{k,2},(\mu_0-\mu_k^*)/2\}>0.
			\end{equation*}
			Then $\varepsilon_k<\delta_k/2$ and $\mu_k^*+\varepsilon_k<\mu_0$.
			For $\mu_k^*\leqslant\mu\leqslant\mu_k^*+\varepsilon_k$,
			the annulus $\overline B_{\mu_k^*+2\delta_k}\setminus B_\mu$,
			the punctured ball $B_{r_k}(a_k)\setminus\{a_k\}$, and $K_k$
			cover $\pffOt\setminus(B_\mu\cup\{a_k\})$.
			Thus \eqref{p45:collar}, \eqref{p45:punctureprotection} and
			\eqref{p45:compactcontinuation} give
			\begin{equation*}
				w_{k,\mu}\leqslant w_k
				\quad\text{on }\pffOt\setminus(B_\mu\cup\{a_k\})
				\quad\text{for }\mu_k^*\leqslant\mu\leqslant\mu_k^*+\varepsilon_k.
			\end{equation*}
			The comparison is already known from \hyperref[p45:step:B2.1]{Step B2.1} for every $0<\mu<\mu_k^*$. It therefore holds for every $0<\mu\leqslant\mu_k^*+\varepsilon_k$.
			By the definition of $\mu_k^*$ in \eqref{p45:criticalradius}, we would have $\mu_k^*\geqslant\mu_k^*+\varepsilon_k$, a contradiction.
			We conclude that $\mu_k^*=\mu_0$ for all sufficiently large $k$,
			proving Step B2.
			
			Choose $k_0(\mu_0)\geqslant k_1$ so that $\mu_k^*=\mu_0$ and the inclusions in \hyperref[p45:step:B2.1]{Step B2.1} hold for every $k\geqslant k_0(\mu_0)$.
			By the definition of $\mu_k^*$ and the continuity argument in \hyperref[p45:step:B2.1]{Step B2.1}, we obtain
			\begin{equation*}
				w_{k,\mu}\leqslant w_k\quad\text{on }\pffOt\setminus(B_\mu\cup\{a_k\})
				\qquad\text{for }0<\mu\leqslant\mu_0,\quad k\geqslant k_0(\mu_0).
			\end{equation*}
			
			\pffclaimpart{B3}{Passage of the Kelvin comparison to the limit}
			Fix $\mu>0$, choose $\mu_0>\max\{\mu,\mu_1\}$, and fix any $y$ with $|y|>\mu$. Choose $L>|y|$. For all sufficiently large $k$,
			$\overline B_L\subset\pffOt\setminus\{a_k\}$ and $k\geqslant k_0(\mu_0)$, so the comparison proved above applies at $y$.
			The reflected point $y^\mu$ is fixed and satisfies $|y^\mu|<\mu<L$. Consequently, \eqref{p45:limits} gives
			\begin{equation*}
				\left(\frac\mu{|y|}\right)^{n-2}w_0(y^\mu)
				=\lim_{k\to\infty}w_{k,\mu}(y)
				\leqslant\lim_{k\to\infty}w_k(y)=w_0(y).
			\end{equation*}
			This holds for every fixed $y$ with $|y|>\mu$; on $|y|=\mu$ equality follows from the definition. Since $\mu>0$ was arbitrary, Claim B follows.
		\end{proof}
		
		\pffproofstep{Conclusion of the proof.}
		Let $\bar\lambda_{w_0}(x)$ denote the critical radius defined in
		\eqref{criticalradius} with $u=w_0$.
		Claim A shows that $w_0$ is a positive $C^2$ solution of
		\eqref{ourlimitPDEnonsing} on $\pffR^n$ with finite Hartree potential.
		Since $\bar\lambda_{w_0}(0)=\infty$ by Claim B, the second assertion
		of Lemma~\ref{lemmaA.2} yields $\bar\lambda_{w_0}(x)=\infty$
		for every $x\in\pffR^n$. Hence
		\begin{equation*}
			(w_0)_{x,\mu}(y)\leqslant w_0(y)
			\qquad\text{for every }x\in\pffR^n,\ \mu>0,\ |y-x|\geqslant\mu.
		\end{equation*}
		Lemma~\ref{lm:msfundamentallemma}(ii), with $\nu=n-2$, implies that $w_0$ is constant.
		Since $w_0(0)=1$, we have $w_0\equiv1$.
		This contradicts the nonconstant form of $w_0$ in \eqref{p45:bubble} and proves \eqref{p45:target}.
	\end{proof}
	
	\enlargethispage{\baselineskip}
	\subsection{Asymptotic radial symmetry}
	We keep the notation $p,F,\beta$ from the preceding sections.
	\begin{proposition}\label{prop:asymptoticsymmetry}
		Let $n\geqslant3$ and $0<\alpha<n$. Suppose that
		$u\in C^2(B_2\setminus\{0\})$ is a positive classical solution of
		\begin{equation*}
			-\Delta u=v_2[u]u^{p-1}\quad\hbox{in }B_2\setminus\{0\}.
		\end{equation*}
		Assume that $u$ satisfies \eqref{eq:blowup}. Then there exist $r_0\in(0,2)$ and $C>0$ such that
		\begin{equation*}
			\sup_{\theta\in\mathbb S^{n-1}}
			\left|\frac{u(r\theta)}{\bar u(r)}-1\right|\leqslant Cr
			\quad\text{ for }0<r<r_0,
		\end{equation*}
		where $\bar u$ is the spherical average defined in Theorem~\ref{thm2}.
	\end{proposition}
	
	\begin{proof}
		Write $v:=v_2[u]$, with the integral taken over $B_2$.
		
		Fix $0<\tau'<\tau<2$. We compare $u_{x,\mu}$ with $u$ in $B_{\tau'}$ and split the integral defining $v$ at $B_\tau$.
		By Proposition~\ref{prop:upperbound}, there exist $C_U,r_U>0$, with $r_U<2$, such that
		\begin{equation}\label{p46:upper}
			u(z)\leqslant C_U|z|^{-\beta}\quad\text{for }0<|z|<r_U.
		\end{equation}
		Choose $0<\varepsilon_*<\min\{\tau'/8,r_U/4\}$.
		Let $u_{x,\mu}$ be the Kelvin transform from
		Definition~\ref{def:kelvintransform}.
		For $0<\mu<|x|<\varepsilon_*$ and $|y-x|\geqslant\mu$,
		we have
		$\mathcal I_{x,\mu}(y)\in\overline B_\mu(x)
		\subset B_{\tau'}\setminus\{0\}$,
		so $u_{x,\mu}$ is well defined there.
		We will find $0<\varepsilon\leqslant\varepsilon_*$ such that
		\begin{equation}\label{p46:goal}
			u_{x,\mu}(y)\leqslant u(y)
			\quad\text{for }0<|x|<\varepsilon,\quad 0<\mu<|x|, \quad
			y\in B_{\tau'}\setminus(B_\mu(x)\cup\{0\}).
		\end{equation}
		
		\medskip\noindent\textbf{Claim 1.}
		\emph{For every $x$ with $0<|x|<\varepsilon_*$, there exists $0<\mu_x<|x|$ such that $u_{x,\mu}\leqslant u$ on $B_{\tau'}\setminus(B_\mu(x)\cup\{0\})$ for all $0<\mu\leqslant\mu_x$.}
		
		Since $u>0$ and $u(y)\to+\infty$ as $y\to0$, we have
		\begin{equation*}
			m_{\mathrm{in}}:=\inf_{B_{\tau'}\setminus\{0\}}u>0.
		\end{equation*}
		Fix $x$ with $0<|x|<\varepsilon_*$.
		The local monotonicity argument used in Claim 1 of
		Proposition~\ref{prop:symmetry} provides $0<r_x<|x|/2$
		such that $r^\beta u(x+r\theta)$ is increasing for
		$0<r\leqslant r_x$ and every $\theta\in\mathbb S^{n-1}$.
		For $0<\mu\leqslant r\leqslant r_x$, applying this monotonicity
		at $\mu^2/r$ and $r$, and using $2\beta=n-2$, we obtain
		\begin{equation*}
			u_{x,\mu}(x+r\theta)
			=\left(\frac{\mu}{r}\right)^{n-2}
			u\left(x+\frac{\mu^2}{r}\theta\right)
			\leqslant u(x+r\theta).
		\end{equation*}
		
		For $y\in B_{\tau'}\setminus\{0\}$ with $|y-x|\geqslant r_x$ and $0<\mu<r_x$, the reflected point lies in $B_{r_x}(x)$.
		With $M_x=\max_{\overline B_{r_x}(x)}u<\infty$, we have
		$u_{x,\mu}(y)\leqslant M_x(\mu/r_x)^{n-2}$, while $u(y)\geqslant m_{\mathrm{in}}$.
		We can therefore choose
		\begin{equation*}
			0<\mu_x<\min\{r_x,|x|\},\qquad
			M_x(\mu_x/r_x)^{n-2}\leqslant m_{\mathrm{in}}.
		\end{equation*}
		For $0<\mu\leqslant\mu_x$, this gives $u_{x,\mu}(y)\leqslant u(y)$ when $|y-x|\geqslant r_x$.
		The comparison for $\mu\leqslant|y-x|\leqslant r_x$ completes the proof of Claim 1.
		The constants $r_x,M_x,\mu_x$ may depend on $x$, whereas $m_{\mathrm{in}}$ is independent of $x$.
		
		\medskip
		For $0<|x|<\varepsilon_*$, define the critical radius for the comparison in $B_{\tau'}$ by
		\begin{equation*}
			\overline\mu(x)=\sup\left\{t\in(0,|x|):
			\begin{array}{l}
				u_{x,\mu}(y)\leqslant u(y)\text{ for every }0<\mu\leqslant t,\\
				y\in B_{\tau'}\setminus(B_\mu(x)\cup\{0\})
			\end{array}\right\}.
		\end{equation*}
		By Claim 1, $0<\overline\mu(x)\leqslant|x|$.
		
		\medskip\noindent\textbf{Claim 2.}
		\emph{There exists $0<\varepsilon\leqslant\varepsilon_*$ such that $\overline\mu(x)=|x|$ for every $0<|x|<\varepsilon$.}
		\par\nopagebreak
		
		We first compare $u_{x,\mu}$ and $u$ on the fixed outer annulus. For $|z|\geqslant\tau'$ and $0<\mu\leqslant|x|<\varepsilon_*$, we have
		\begin{equation*}
			|z-x|\geqslant\tau'-|x|\geqslant\tau'/2,\qquad
			|\mathcal I_{x,\mu}(z)-x|\leqslant\frac{2|x|^2}{\tau'}\leqslant\frac{|x|}{2}.
		\end{equation*}
		Consequently,
		\begin{equation*}
			\frac{|x|}{2}\leqslant|\mathcal I_{x,\mu}(z)|\leqslant\frac{3|x|}{2}<r_U.
		\end{equation*}
		Let $D_U=2^\beta C_U$. From \eqref{p46:upper},
		\begin{equation}\label{p46:exteriorkelvin}
			u_{x,\mu}(z)\leqslant D_U\mu^{n-2}|x|^{-\beta}|z-x|^{2-n}
			\leqslant D_U|x|^\beta|z-x|^{2-n}\quad\text{for }|z|\geqslant\tau'.
		\end{equation}
		Since $u$ has a positive minimum on $\{z:\tau'\leqslant|z|\leqslant\tau\}$, estimate \eqref{p46:exteriorkelvin} allows us to choose $0<\varepsilon_o\leqslant\varepsilon_*$, independent of $x,\mu$, such that
		\begin{equation}\label{p46:outer}
			u_{x,\mu}(z)\leqslant\frac12u(z)
			\quad\text{for }\tau'\leqslant|z|\leqslant\tau,\quad 0<\mu\leqslant|x|<\varepsilon_o.
		\end{equation}
		
		We next use $v_{x,\mu}$ from \eqref{prelim:potentialkelvin}. Extend $F(u)$ by zero outside $B_2$ and at the origin. The potential transformation \eqref{prelim:potentialcovariance} and the inversion and splitting in \hyperref[p45:step:B2.2]{Step B2.2} of Proposition~\ref{prop:upperbound} give, for $0<\mu<|x|<\varepsilon_*$ and $y\in B_{\tau'}\setminus(\overline B_\mu(x)\cup\{0\})$,
		\begin{align}\label{p46:difference}
			v(y)-v_{x,\mu}(y)
			={}&\int_{B_\tau\setminus B_\mu(x)}K_\alpha(x,\mu;y,z)
			[F(u(z))-F(u_{x,\mu}(z))]\,\ud z\notag\\
			&+\mathcal A_{x,\mu}(y)-\mathcal N_{x,\mu}(y),
		\end{align}
		where
		\begin{align*}
			\mathcal A_{x,\mu}(y)&:=\int_{B_2\setminus B_\tau}
			K_\alpha(x,\mu;y,z)F(u(z))\,\ud z,\\
			\mathcal N_{x,\mu}(y)&:=\int_{\mathbb R^n\setminus B_\tau}
			K_\alpha(x,\mu;y,z)F(u_{x,\mu}(z))\,\ud z.
		\end{align*}
		The negative tail is integrated over $\mathbb R^n\setminus B_\tau$, since $F(u_{x,\mu})$ need not vanish outside $B_2$.
		By Lemma~\ref{lm:differencecomparison}, $0<K_\alpha\leqslant|y-z|^{\alpha-n}$ for $|y-x|,|z-x|>\mu$ with $y\ne z$.
		Since $v(y)$ and $v_{x,\mu}(y)$ are finite, all terms in \eqref{p46:difference} are finite.
		
		Let $0<\mu<|x|<\varepsilon_*$ and $y\in B_{\tau'}$ with $r=|y-x|>\mu$.
		To compare the two tails near $r=\mu$, set
		\begin{equation*}
			\Theta_{x,\mu}(y):=\min\{(r^2-\mu^2)/\mu^2,1\}.
		\end{equation*}
		For $|z|\geqslant\tau$, the bounds $\mu<|x|<\tau/4$ and $|y|<\tau'$ give
		\begin{equation*}
			\tfrac12|z|^2\leqslant |z-x|^2-\mu^2\leqslant\tfrac{25}{16}|z|^2 \quad \text{and}  \quad (1-\tau'/\tau)|z|\leqslant|y-z|\leqslant2|z|.
		\end{equation*}
		By \eqref{prelim:kernelparameter}, $\min\{s_{\mathcal K}(x,\mu;y,z),1\}\simeq\Theta_{x,\mu}(y)$, with constants depending only on $\tau'/\tau$. By Lemma~\ref{lm:differencecomparison},
		\begin{equation}\label{p46:kernelbounds}
			c_K\Theta_{x,\mu}(y)|y-z|^{\alpha-n}
			\leqslant K_\alpha(x,\mu;y,z)
			\leqslant C_K\Theta_{x,\mu}(y)|y-z|^{\alpha-n},
		\end{equation}
		where $c_K,C_K>0$ depend only on $n,\alpha$ and $\tau'/\tau$.
		
		For $z\in B_2\setminus B_\tau$, we have $|y-z|\leqslant\tau'+2$. Since $\alpha-n<0$, the lower bound in \eqref{p46:kernelbounds} gives
		\begin{equation*}
			\mathcal A_{x,\mu}(y)
			\geqslant c_K(\tau'+2)^{\alpha-n}
			\left(\int_{B_2\setminus B_\tau}F(u(z))\,\ud z\right)\Theta_{x,\mu}(y)
			=:c_A\Theta_{x,\mu}(y).
		\end{equation*}
		The constant $c_A$ is finite and positive, since $u>0$ and $u\in L^p(B_2)$ by Proposition~\ref{togetlocaldis}.
		For the negative tail, \eqref{p46:exteriorkelvin} and $(n-2)p=n+\alpha$ give
		\begin{equation*}
			F(u_{x,\mu}(z))\leqslant\frac{D_U^p}{p}|x|^{(n+\alpha)/2}|z-x|^{-n-\alpha}
			\quad\text{for }|z|\geqslant\tau.
		\end{equation*}
		Using \eqref{p46:kernelbounds}, $|z-x|\geqslant|z|/2$, and $|y-z|\geqslant(1-\tau'/\tau)|z|$, we obtain
		\begin{equation*}
			\begin{aligned}
				\mathcal N_{x,\mu}(y)
				&\leqslant C|x|^{(n+\alpha)/2}\Theta_{x,\mu}(y)
				\int_{|z|\geqslant\tau}|z|^{-2n}\,\ud z\\
				&=C_N|x|^{(n+\alpha)/2}\Theta_{x,\mu}(y).
			\end{aligned}
		\end{equation*}
		Here $C_N$ depends only on $n,\alpha,C_U,\tau,\tau'$.
		To make the positive tail larger than the negative tail, choose
		\begin{equation}\label{p46:epsilon}
			0<\varepsilon\leqslant\varepsilon_o,\qquad
			C_N\varepsilon^{(n+\alpha)/2}\leqslant c_A/2.
		\end{equation}
		Since $c_A,C_N$ are independent of $x,\mu,y$, the same $\varepsilon$ works for all centers.
		For $0<\mu<|x|<\varepsilon$ and $y\in B_{\tau'}\setminus(\overline B_\mu(x)\cup\{0\})$, the two tail estimates give
		\begin{equation}\label{p46:tailgap}
			\begin{aligned}
				\mathcal A_{x,\mu}(y)-\mathcal N_{x,\mu}(y)
				&\geqslant\bigl(c_A-C_N|x|^{(n+\alpha)/2}\bigr)\Theta_{x,\mu}(y)\\
				&\geqslant\tfrac12c_A\Theta_{x,\mu}(y)>0.
			\end{aligned}
		\end{equation}
		Fix $0<|x|<\varepsilon$ and suppose that $\ell:=\overline\mu(x)<|x|$.
		As in \hyperref[p45:step:B2.1]{Step B2.1}, the comparison holds for every $\mu<\ell$ by definition of the critical radius, and continuity in $\mu$ extends it to $\ell$.
		Combining this with \eqref{p46:outer}, we find
		\begin{equation*}
			u-u_{x,\ell}\geqslant0
			\quad\hbox{on }B_\tau\setminus(B_\ell(x)\cup\{0\}),
			\qquad u-u_{x,\ell}>0\quad\hbox{on }\partial B_{\tau'}.
		\end{equation*}
		At $\mu=\ell$, the first integral in \eqref{p46:difference} is nonnegative. Hence \eqref{p46:tailgap} gives $v>v_{x,\ell}$ in
		$E_\ell:=B_{\tau'}\setminus(\overline B_\ell(x)\cup\{0\})$.
		Write $W_\mu=u-u_{x,\mu}$. Subtracting the Kelvin-transformed equation for $u_{x,\ell}$ from the equation for $u$, we find
		\begin{equation*}
			\begin{aligned}
				-\Delta W_\ell
				={}&v\bigl(u^{p-1}-u_{x,\ell}^{p-1}\bigr)
				+(v-v_{x,\ell})u_{x,\ell}^{p-1}>0
				\quad\text{in }E_\ell.
			\end{aligned}
		\end{equation*}
		Since $W_\ell\geqslant0$, the interior minimum argument in \hyperref[p45:step:B2.4]{Step B2.4} gives $W_\ell>0$ in $E_\ell$.
		
		Take $L=(\ell+|x|)/2<|x|$.
		The sphere $\partial B_\ell(x)$ stays away from both $0$ and
		$\partial B_{\tau'}$.
		The Hopf and continuity argument in
		\hyperref[p45:step:B2.5]{Step B2.5}
		of Proposition~\ref{prop:upperbound} allows us to choose
		$0<\delta<L-\ell$ such that
		\begin{equation*}
			W_\mu(y)\geqslant0
			\quad\text{for }\mu\leqslant|y-x|\leqslant\ell+2\delta,
			\quad\ell\leqslant\mu\leqslant\ell+\delta.
		\end{equation*}
		
		For $y$ sufficiently close to $0$, we have $|y-x|>L$.
		For every $\ell\leqslant\mu\leqslant L$, the reflected point
		lies in $\overline B_L(x)\Subset B_{\tau'}\setminus\{0\}$,
		and the Kelvin formula gives $u_{x,\mu}(y)\leqslant\max_{\overline B_L(x)}u<\infty$. Since $u(y)\to+\infty$ as $y\to0$, we can choose
		$0<\rho<|x|-L$ such that $W_\mu\geqslant0$ in
		$B_\rho\setminus\{0\}$ for all $\ell\leqslant\mu\leqslant L$.
		
		On the fixed compact set
		\begin{equation*}
			\mathcal K:=\overline B_{\tau'}\setminus
			\bigl(B_{\ell+\delta}(x)\cup B_\rho\bigr),
		\end{equation*}
		the strict positivity of $W_\ell$ and \eqref{p46:outer} give
		$\min_{\mathcal K}W_\ell>0$.
		Uniform convergence $W_\mu\to W_\ell$ on $\mathcal K$
		therefore yields $W_\mu>0$ there for
		$\ell\leqslant\mu\leqslant\ell+\eta$, with some $0<\eta<\delta$.
		Combining the three comparisons as in
		\hyperref[p45:step:B2.8]{Step B2.8}, we obtain
		\begin{equation*}
			u_{x,\mu}\leqslant u
			\quad\text{on }B_{\tau'}\setminus(B_\mu(x)\cup\{0\}),
			\quad\ell\leqslant\mu\leqslant\ell+\eta.
		\end{equation*}
		Thus the comparison holds for all $0<\mu\leqslant\ell+\eta$,
		which contradicts the definition of $\ell$.
		Here $\eta$ may depend on $x$, but $\varepsilon$ in
		\eqref{p46:epsilon} does not.
		This proves Claim 2 and \eqref{p46:goal}.
		
		\medskip\noindent\textbf{Claim 3.}
		\emph{There exist $r_0>0$ and $C>0$ such that}
		\begin{equation*}
			\sup_{\theta\in\mathbb S^{n-1}}
			\left|\frac{u(r\theta)}{\bar u(r)}-1\right|\leqslant Cr
			\quad\text{for }0<r<r_0.
		\end{equation*}
		
		Following Jin and Xiong~\cite[proof of Proposition~3.3]{arxiv:1901.01678},
		take $a=r_0=\varepsilon/4$ and fix $0<r<r_0$.
		Let $x_1,x_2\in\partial B_r$ be maximum and minimum points of $u$.
		If the two values agree, the conclusion is immediate. Otherwise, define
		\begin{equation*}
			d=|x_1-x_2|,\qquad e=\frac{x_1-x_2}{d},\qquad
			\xi=x_1+ae,\qquad \mu=\sqrt{a(a+d)}.
		\end{equation*}
		These choices give
		\begin{equation*}
			|\xi|^2=r^2+\mu^2,\qquad |\xi|\leqslant r+a<\varepsilon,
			\qquad |x_2-\xi|=a+d>\mu,\qquad
			\mathcal I_{\xi,\mu}(x_2)=x_1.
		\end{equation*}
		Thus $0<\mu<|\xi|<\varepsilon$.
		Applying \eqref{p46:goal} at $x_2$ and using $d\leqslant2r$,
		we obtain
		\begin{equation*}
			\max_{\partial B_r}u=u(x_1)\leqslant(1+d/a)^\beta u(x_2)\leqslant(1+8r/\varepsilon)^\beta\min_{\partial B_r}u.
		\end{equation*}
		Since the spherical average lies between the minimum and maximum, this estimate and the mean value theorem prove Claim 3
		and complete the proof.
	\end{proof}
	
	\begin{proof}[Proof of Theorem~\ref{thm2}]
		Let $\beta=(n-2)/2$ and $w(x)=(R/2)^\beta u(Rx/2)$ for
		$x\in B_2\setminus\{0\}$.
		Then $w\in C^2(B_2\setminus\{0\})$ is positive and
		$w(x)\to+\infty$ as $x\to0$. Since $(n-\alpha)+(n-2)(p-1)=n+2$, rescaling the equation and changing variables in the Hartree integral yield
		\begin{equation*}
			-\Delta w=v_2[w]w^{p-1}
			\quad\text{in }B_2\setminus\{0\}.
		\end{equation*}
		By Proposition~\ref{prop:upperbound}, there exist $C_0>0$ and
		$s_0\in(0,2)$ such that $w(x)\leqslant C_0|x|^{-\beta}$ for
		$0<|x|<s_0$. Scaling back yields
		\begin{equation*}
			u(y)=\left(\frac2R\right)^\beta w(2y/R)
			\leqslant C_0|y|^{-\beta}
			\quad\text{for }0<|y|<Rs_0/2.
		\end{equation*}
		Proposition~\ref{prop:asymptoticsymmetry} also applies to $w$.
		Let $C_1>0$ and $s_1\in(0,2)$ be the constants in that proposition.
		Since $\bar w(r)=(R/2)^\beta\bar u(Rr/2)$, substituting
		$r=2\rho/R$ yields
		\begin{equation*}
			\sup_{\theta\in\mathbb S^{n-1}}
			\left|\frac{u(\rho\theta)}{\bar u(\rho)}-1\right|
			=\sup_{\theta\in\mathbb S^{n-1}}
			\left|\frac{w(2\rho\theta/R)}{\bar w(2\rho/R)}-1\right|
			\leqslant\frac{2C_1}{R}\rho
			\quad\text{ for }0<\rho<Rs_1/2.
		\end{equation*}
		Taking $r_0=(R/2)\min\{s_0,s_1\}$ and
		$C=\max\{C_0,2C_1/R\}$ proves both conclusions.
	\end{proof}
	
\raggedbottom
\begin{acknowledgement}
		This paper was initiated during the first and third authors' visits to Zhejiang Normal University in 2023. 
		The authors would like to express their sincere appreciation for the warm hospitality extended to them during their stay.
	\end{acknowledgement}

\begingroup
\makeatletter\@secpenalty=0\makeatother
\section*{Declarations}
	
	\subsection*{Data Availability Statement}
	No data were generated or analyzed in this study.
	
	\subsection*{Conflict of Interest}
	The authors declare that they have no conflict of interest.
	
	\subsection*{Declaration of Generative AI Use}
	The present version is a revised one of the manuscript firstly posted on arXiv in May 2025 (arXiv:2505.19021). During the revision, the authors used a generative AI tool to check the mathematical arguments and internal consistency, and to suggest editorial corrections. The check by AI reminded the authors that positive singular solutions of \eqref{ourlimitPDE} might not satisfy the global integrability assumption. This observation suggests the authors to reconsider the integrability condition in the first version of the paper. Working directly with the equation, the authors established the integral representation of the solution in a suitable working space, and thereby obtain the main results. The problem, the main ideas of the proofs, and all computations in the paper are due solely to the authors, who verified every suggestion independently and take full responsibility for the content of this paper.
\par\endgroup
	

\end{document}